\documentclass{amsart}
\usepackage{url}
\usepackage{amsmath}
\usepackage{stmaryrd}
\usepackage{siunitx}
\usepackage{commath}
\usepackage{graphicx}
\usepackage[caption=false]{subfig}
\usepackage{caption}
\usepackage[numbers,square]{natbib}
\usepackage{enumerate}
\usepackage{bm}
\usepackage{lipsum}
\usepackage{amsfonts}
\usepackage{epstopdf}
\usepackage{algorithmic}
\usepackage{amsopn}
\usepackage{mathtools}
\usepackage{multirow}
\usepackage{booktabs}
\usepackage{tikz}
\usetikzlibrary{math}
\usetikzlibrary{patterns}
\usepackage{hyperref}
\hypersetup{
  colorlinks   = true, 
  urlcolor     = blue, 
  linkcolor    = blue, 
  citecolor   = red 
}
\usepackage{cleveref}
\newtheorem{theorem}{Theorem}[section]
\newtheorem{lemma}[theorem]{Lemma}

\newtheorem{corollary}[theorem]{Corollary}

\theoremstyle{definition}

\theoremstyle{remark}
\newtheorem{remark}[theorem]{Remark}

\numberwithin{equation}{section}
\makeatletter
\newcommand{\tnorm}{\@ifstar\@tnorms\@tnorm}
\newcommand{\@tnorms}[1]{%
	\left|\mkern-1.5mu\left|\mkern-1.5mu\left|
	#1
	\right|\mkern-1.5mu\right|\mkern-1.5mu\right|
}
\newcommand{\@tnorm}[2][]{%
	\mathopen{#1|\mkern-1.5mu#1|\mkern-1.5mu#1|}
	#2
	\mathclose{#1|\mkern-1.5mu#1|\mkern-1.5mu#1|}
}
\makeatother

\newcommand{\llangle}[0]{\langle\!\langle}
\newcommand{\rrangle}[0]{\rangle\!\rangle}
\usepackage{xspace,color}

\begin{document}
\title[Space-time HDG for linear free-surface waves]{A space-time
  hybridizable discontinuous Galerkin method for linear free-surface
  waves}
\author[G. Sosa Jones]{Giselle Sosa Jones}
\address{Department of Applied Mathematics, University of Waterloo, Ontario, Canada}
\email{gsosajon@uwaterloo.ca}

\author[J. J. Lee]{Jeonghun J. Lee}
\address{Department of Mathematics, Baylor University, Texas, USA}
\email{jeonghun\_lee@baylor.edu}

\author[S. Rhebergen]{Sander Rhebergen}
\address{Department of Applied Mathematics, University of Waterloo, Ontario, Canada}
\email{srheberg@uwaterloo.ca}
\thanks{Natural Sciences and Engineering Research Council of Canada
  through the Discovery Grant program (RGPIN-05606-2015) and the
  Discovery Accelerator Supplement (RGPAS-478018-2015).}
\date{}
\begin{abstract}
  We present and analyze a novel space-time hybridizable discontinuous
  Galerkin (HDG) method for the linear free-surface problem on
  prismatic space-time meshes. We consider a mixed formulation which
  immediately allows us to compute the velocity of the fluid. In order
  to show well-posedness, our space-time HDG formulation makes use of
  weighted inner products. We perform an \emph{a priori} error
  analysis in which the dependence on the time step and spatial mesh
  size is explicit. We provide two numerical examples: one that
  verifies our analysis and a wave maker simulation.
\end{abstract}
\maketitle
\section{Introduction}

The study of water waves is crucial when designing, for example,
ships, offshore structures, levees, and seawalls. For this reason
free-surface problems are of great interest in many fields of
engineering, such as naval and maritime engineering. In order to
analyze the interaction between waves and prototypes of ships and
other structures, experiments in water tanks using wave makers may be
used to simulate realistic maritime situations. However, many complex
cases can be difficult to reproduce in water tanks. Developing
efficient and accurate numerical models is therefore important to help
simulate water-wave phenomena.

In general, free-surface problems are mathematically described by a
set of partial differential equations that model the movement of the
fluid, and a set of boundary conditions that describe and determine
the free-surface. These problems are particularly hard to solve
because the free-surface that defines the shape of the domain is part
of the solution to the problem. In order to simplify the problem,
certain assumptions on the flow can be made. For example, many
applications can be modeled by considering the fluid to be
incompressible, inviscid, and irrotational. Moreover, assuming that
the wave displacement is small, the free-surface boundary conditions
may be linearized.

In order to effectively model water waves, we require a stable and
higher-order accurate numerical method to minimize numerical diffusion
and dispersion. These properties may be achieved, for example, by
using a discontinuous Galerkin (DG) method for the spatial
discretization combined with a higher-order accurate time stepping
scheme, as done, for example, for linear water waves in
\cite{Vegt:2005}. In \cite{Vegt:2005} they combined a second order
accurate time stepping scheme with a higher-order accurate DG method
for the spatial discretization. They proved stability and provided an
\emph{a priori} error analysis of their method.

Alternatively, one may use a DG method to discretize partial
differential equations in space and time simultaneously. These
space-time DG methods achieve higher-order accuracy in both space and
time simply by increasing the polynomial approximation in
space-time. Space-time DG methods have successfully been applied to
many different problems, such as compressible flows \cite{Klaij:2006,
  Vegt:2002}, incompressible flows \cite{Rhebergen:2013a,
  Tavelli:2015, Tavelli:2016, Vegt:2008}, depth-averaged flows
\cite{Ambati:2007, Rhebergen:2009} and non-linear free-surface
problems \cite{Gagarina:2014, Vegt:2007}. However, there is a major
drawback to space-time DG methods: since the dimension of the problem
is increased by one, the number of degrees-of-freedom is much higher
compared to standard DG methods. A consequence is that space-time DG
methods are generally computationally more expensive than traditional
approaches. 

Hybridizable discontinuous Galerkin (HDG) methods were first
introduced in \cite{Cockburn:2009a} as a computationally more
efficient alternative to discontinuous Galerkin methods. HDG methods
maintain the local conservation properties of DG methods, but have
been shown to be as computationally efficient in certain cases as
continuous Galerkin finite element methods \cite{Kirby:2012,
  Yakovlev:2016}. This is achieved by introducing a variable that
exists only on the facets. Communication between two neighbouring
elements is through this facet variable. By coupling
degrees-of-freedom on elements only through the degrees-of-freedom on
facets, the element degrees-of-freedom can be eliminated. This static
condensation results in a global linear system for the facet variable
only of which the size is significantly smaller than the global
linear system obtained by a DG discretization. Once the facet variable
is known the element unknowns can be reconstructed element-wise.

Space-time HDG methods apply the HDG method in space-time
\cite{Horvath:2019, Kirk:2019, Rhebergen:2012, Rhebergen:2013b,
  Wang:2014} and result in an efficient alternative to space-time DG
methods. In this paper we present a novel space-time HDG method for
the linear free-surface problem. We consider a mixed formulation based
on the splitting introduced originally for the wave equation in
\cite{Cockburn:2016}. This is different from previous works on DG
methods for free-surface problems in which the primal form of the
problem is considered \cite{Vegt:2005, Vegt:2007}. The reason to
consider the mixed formulation is that it allows us to immediately
obtain the velocity of the fluid without post-processing. To the best
of our knowledge, such a splitting has not been considered for
inviscid, incompressible and irrotational free-surface flow
problems. Furthermore, as opposed to standard discontinuous Galerkin
discretizations, our space-time HDG formulation uses weighted
inner-products. The idea of using weighted inner-products was
previously applied in \cite{French:1993} to the wave equation. Here,
we use weighted inner-products to prove well-posedness of the
space-time HDG formulation of the linear free-surface waves problem.

The \emph{a priori} error analysis in this paper is projection based
\cite{Cockburn:2010}, but modified for weighted inner-products and
space-time prismatic elements. Additionally, we derive scaling
identities between the reference and physical space-time prisms that
separate the spatial dimension from the temporal dimension. Such
anisotropic scaling identities were previously derived in the context
of anisotropic meshes in \cite{Georgoulis:2003} and space-time meshes
in \cite{Kirk:2019, Sudirham:2006}. Furthermore, these scaling
identities allow us to obtain \emph{a priori} error estimates in which
the dependence on the time step and the spatial mesh size is
explicit. This is in contrast to error bounds that depend on the
space-time mesh size as derived, for example, for parabolic problems
in \cite{Cangiani:2017}.

The outline of this paper is as follows. In \cref{sec:problem} we
formulate the linear free-surface problem in a space-time
setting. Next, in \cref{sec:ST_HDG}, we present the space-time HDG
method and show well-posedness of the discrete
formulation. \Cref{sec:analysistools} provides the tools required for
the \emph{a priori} error analysis that we present in
\cref{sec:erroranalysis}. Our theoretical results are verified by
numerical examples in \cref{sec:numerical_results} and conclusions are
drawn in \cref{sec:conclusions}.

\section{The linear free-surface problem}
\label{sec:problem}

Let $b(x_1) > 0$ be a piecewise linear polynomial representing the
bottom topography and let $H$ be the average water depth. We define
the flow domain as
$\Omega := \cbr{\boldsymbol{x}=(x_1,x_2) \in \mathbb{R}^2\, :\,
  -H+b(x_1) < x_2 < 0,\ L < x_1 < R}$ where $L$ and $R$ are given
constants. The boundary of the domain, $\partial\Omega$, is
partitioned into a free-surface boundary
$\Gamma_S := \cbr{\boldsymbol{x} \in \mathbb{R}^2 \,:\,x_1 \in [L,
  R],\, x_2 = 0 }$, periodic boundaries $\Gamma_P$, and a solid
boundary
$\Gamma_N := \cbr{ \boldsymbol{x} \in \mathbb{R}^2 \,:\, x_1 \in [L,
  R],\, x_2 = -H+b(x_1)}$. These boundaries do not overlap and are
such that
$\bar{\Gamma}_S \cup \bar{\Gamma}_N \cup \bar{\Gamma}_P =
\partial\Omega$. The boundary outward unit normal vector is denoted by
$\boldsymbol{n}$. See \cref{fig:domain} for an illustration of the
notation.

\begin{figure}[tbp]
  \begin{center}
    \begin{tikzpicture}[scale=0.8,
      important line/.style={thick},
      ]
      
      \draw (5,4.3) node{$\Gamma_S$};
      \draw (5,-0.3) node{$\Gamma_N$};
      \draw (-0.4,2) node{$\Gamma_P$};
      \draw (10.4,2) node{$\Gamma_P$};
      
      \draw[important line] (10,0) -- (10,4) ;
      \draw[important line] (10,4) -- (0,4) ;
      \draw[important line] (0,4) -- (0,0) ;
      \draw[important line] (0,0) -- (3,0.8) ;
      \draw[important line] (3,0.8) -- (7,0.1) ;
      \draw[important line] (7,0.1) -- (10,0) ;
      \draw[dashed] (0,0) -- (0,-1) ;
      \draw (0,-1.2) node{$x_1 = L$};
      \draw[dashed] (10,0) -- (10,-1) ;
      \draw (10,-1.2) node{$x_1 = R$};
      \draw[dashed] (10,4) -- (11,4) ;
      \draw (11.8,4) node{$x_2 = 0$};
      \draw[dashed] (0,0) -- (11,0) ;
      \draw (12.1,0) node{$x_2 = -H$};
      \draw (5,2) node{$\Omega$};
    \end{tikzpicture}    
  \end{center}
  \caption{Depiction of the flow domain $\Omega \subset
    \mathbb{R}^2$.}
  \label{fig:domain}
\end{figure}
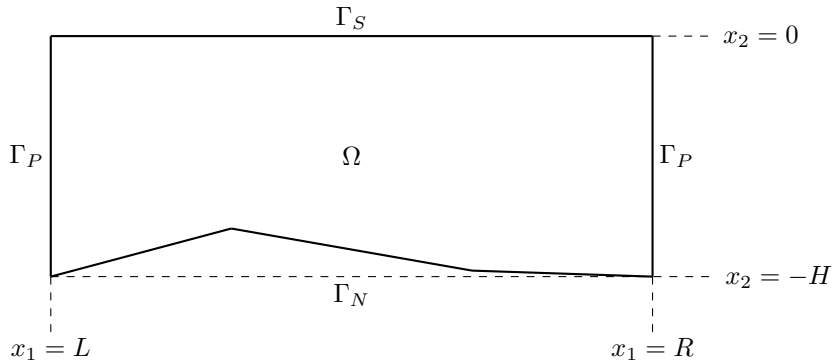

We consider an incompressible, inviscid fluid, with an irrotational
velocity field $\boldsymbol{u}$. Let
$\zeta : [L, R] \times [0,T] \to \mathbb{R}$ denote the wave height
and let $\phi : \Omega \times [0,T] \to \mathbb{R}$ denote the
velocity potential so that $\boldsymbol{u} = \nabla \phi$. The linear
free-surface problem for irrotational flow is given by
\begin{subequations}
  \begin{align}
    \label{eq:LaplaceEQ}
    -\nabla^2 \phi & = 0 && \text{in}\ \Omega,
    \\
    \label{eq:KinematicBC}
    \nabla\phi \cdot \boldsymbol{n}
                            &= \partial_t\zeta && \text{on}\ \Gamma_S,
    \\
    \label{eq:DynamicBC}
    \partial_t\phi + \zeta & = 0 &&\text{on}\ \Gamma_S,
    \\
    \label{eq:BottomBC}
    \nabla\phi \cdot \boldsymbol{n}
                            &= 0 && \text{on}\ \Gamma_N,
  \end{align}
  \label{eq:Problem}
\end{subequations}
where \cref{eq:KinematicBC} and \cref{eq:DynamicBC} are, respectively,
the kinematic and dynamic linear free-surface boundary conditions, and
where \cref{eq:BottomBC} imposes a no-normal flow condition on the
solid boundary. Periodic boundary conditions are applied on
$\Gamma_P$. To close the problem we require the initial conditions
$\phi(\boldsymbol{x}, 0) = \phi^0(\boldsymbol{x})$ and
$\zeta(x_1, 0) = \zeta^0(x_1)$.

In \cite{Vegt:2005} the kinematic \cref{eq:KinematicBC} and dynamic
\cref{eq:DynamicBC} boundary conditions were combined into a single
equation for $\phi$. The resulting linear free-surface model becomes:
\begin{subequations}
  \begin{align}
    \label{eq:LaplaceEQ_vegt}
    -\nabla^2 \phi & = 0 && \text{in}\ \Omega,
    \\
    \label{eq:KinematicDynamicBC}
    \partial_{tt}\phi + \nabla\phi \cdot \boldsymbol{n} & = 0 &&\text{on}\ \Gamma_S,
    \\
    \label{eq:BottomBC_vegt}
    \nabla\phi \cdot \boldsymbol{n}
                            &= 0 && \text{on}\ \Gamma_N.
  \end{align}
  \label{eq:Problem_vegt}
\end{subequations}
As discussed in the next section, we instead introduce a mixed
formulation of the linear free-surface problem \cref{eq:Problem} which
is more suitable for the space-time HDG method.

\section{The space-time hybridizable discontinuous Galerkin method}
\label{sec:ST_HDG}

To introduce the space-time HDG method we first formulate the mixed
space-time formulation of the linear free-surface problem
\cref{eq:Problem}.

Space-time methods do not distinguish between temporal and spatial
variables. A point at time $t=x_0$ with position vector
$\boldsymbol{x}=(x_1,x_2)$ has Cartesian coordinates
$(x_0, \boldsymbol{x})$. We therefore pose the linear free-surface
problem \cref{eq:Problem} on a space-time domain defined as
$\mathcal{E} := \cbr{(x_0, \boldsymbol{x}) \in \mathbb{R}^3\,:
  \boldsymbol{x} \in \Omega,\,0 < x_0 < T }$.  The boundary
$\partial\mathcal{E}$ of the space-time domain $\mathcal{E}$ consists
of
$\Omega_0 := \cbr{(x_0, \boldsymbol{x}) \in \partial\mathcal{E} : x_0
  = 0}$,
$\Omega_N := \cbr{(x_0, \boldsymbol{x}) \in \partial\mathcal{E} : x_0
  = T}$ and
$\mathcal{Q}_{\mathcal{E}} := \cbr{ (x_0, \boldsymbol{x}) \in
  \partial\mathcal{E} : 0 < x_0 < T }$.  Furthermore,
$\mathcal{Q}_{\mathcal{E}}$ is subdivided as
$\mathcal{Q}_{\mathcal{E}} = \partial\mathcal{E}_{\mathcal{S}} \cup
\partial\mathcal{E}_{\mathcal{N}} \cup
\partial\mathcal{E}_{\mathcal{P}}$, where
$\partial\mathcal{E}_{\mathcal{S}} := \cbr{ (x_0, \boldsymbol{x}) \in
  \partial\mathcal{E} : \boldsymbol{x} \in \Gamma_S, 0 < x_0 < T }$,
$\partial\mathcal{E}_{\mathcal{N}} := \cbr{ (x_0, \boldsymbol{x}) \in
  \partial\mathcal{E} : \boldsymbol{x} \in \Gamma_N, 0 < x_0 < T }$,
and
$\partial\mathcal{E}_{\mathcal{P}} := \cbr{ (x_0, \boldsymbol{x}) \in
  \partial\mathcal{E} : \boldsymbol{x}\in \Gamma_P, 0 < x_0 < T }$.

We next introduce two new variables, namely
$\boldsymbol{q} = -\nabla\phi$ and $v = -\partial_t \phi$. A similar
choice of variables was introduced in \cite{Nguyen:2011} for the wave
equation. The linear free-surface problem \cref{eq:Problem} on the
space-time domain $\mathcal{E}$ can then be written as a mixed
space-time formulation which is given by
\begin{subequations}
  \begin{align}
    \label{eq:qDef}
    \partial_t \boldsymbol{q} - \nabla v
    & = 0 && \text{in}\ \mathcal{E},
    \\
    \label{eq:LaplaceForq}
    \nabla\cdot\boldsymbol{q}
    & = 0  && \text{in}\ \mathcal{E},
    \\
    \label{eq:FS_cond}
    -\boldsymbol{q}\cdot\boldsymbol{n}
    & = \partial_t v && \text{on}\ \partial\mathcal{E}_{\mathcal{S}},
    \\
    \label{eq:NeumBCq}
    -\boldsymbol{q}\cdot \boldsymbol{n}
    & = 0 && \text{on}\ \partial\mathcal{E}_{\mathcal{N}},
    \\
    \boldsymbol{q}
    & = -\nabla\phi_0 && \text{on}\ \Omega_0,
    \\
    v & = -\partial_t\phi(0,\boldsymbol{x}) &&\text{on}\ \Omega_0.
  \end{align}
  \label{eq:mixedformlinfreesurface}
\end{subequations}
Note that the wave height $\zeta$ will be the restriction of $v$ to
the free-surface $\Gamma_S$.

\subsection{Notation}
\label{ss:ST_slabs}

Let $\mathcal{I} = [0,T]$ denote the time interval. We partition the
time interval into time levels
$0 = t_0 < t_1 < t_2 < \cdots < t_N = T$ and denote the
$n^{\text{th}}$ time interval by $I_n = (t_n,t_{n+1})$. The length of
each time interval is constant and is denoted by $\Delta t$. The
$n^{\text{th}}$ space-time slab is then defined as
$\mathcal{E}^n := \mathcal{E} \cap \del{I_n \times
  \mathbb{R}^2}$. Define
$\Omega_n := \cbr{(x_0, \boldsymbol{x}) \in \mathcal{E} : x_0 =
  t_n}$. The boundary of a space-time slab, $\partial\mathcal{E}^n$,
can then be divided into $\Omega_n$, $\Omega_{n+1}$ and
$\mathcal{Q}_{\mathcal{E}}^n := \partial\mathcal{E}^n \backslash
\del{\Omega_{n+1} \cup \Omega_n}$. We can further subdivide
$\mathcal{Q}_{\mathcal{E}}^n$ as
$\mathcal{Q}_{\mathcal{E}}^n = \partial\mathcal{E}_{\mathcal{S}}^n
\cup \partial\mathcal{E}_{\mathcal{N}}^n \cup
\partial\mathcal{E}_{\mathcal{P}}^n$, where
$\partial\mathcal{E}_{\mathcal{S}}^n := \cbr{ (x_0, \boldsymbol{x}) \in
  \partial\mathcal{E} : \boldsymbol{x} \in \Gamma_S, t_n < x_0 < t_{n+1} }$,
$\partial\mathcal{E}_{\mathcal{N}}^n := \cbr{ (x_0, \boldsymbol{x})
  \in \partial\mathcal{E} : \boldsymbol{x} \in \Gamma_N, t_n < x_0 <
  t_{n+1} }$,
and
$\partial\mathcal{E}_{\mathcal{P}}^n := \cbr{ (x_0, \boldsymbol{x})
  \in \partial\mathcal{E} : \boldsymbol{x}\in \Gamma_P, t_n < x_0 <
  t_{n+1} }$.

For linear free-surface waves the spatial domain $\Omega$ does not
change with time. We therefore introduce a space-time mesh as
follows. We first introduce a triangulation $\mathcal{T} := \cbr{K}$
of the domain $\Omega_n$, which is the same for all $n$. Each
space-time element $\mathcal{K} \subset \mathcal{E}^n$ is constructed as
$\mathcal{K} = K \times I_n$. The set of all space-time elements,
$\mathcal{T}^n := \cbr{\mathcal{K} : \mathcal{K} \subset \mathcal{E}^n}$ is then a triangulation of the
space-time slab $\mathcal{E}^n$. This is repeated for all $n$. See
\cref{fig:space_time_domain} for an illustration of $\mathcal{E}^n$.

\begin{figure}[tbp]
  \begin{center}
    \begin{tikzpicture}[important line/.style={very thick}, scale=0.8]
      
      \coordinate (n1) at (4,0.7);
      \coordinate (n2) at (6,2);
      
      \coordinate (l1) at (2,2.8);
      \coordinate (l2) at (7,2.8); 
      \coordinate (l3) at (2,-0.6);
      \coordinate (l4) at (7,-0.6); 
      
      \coordinate (l5) at (4.3,3.8);
      \coordinate (l6) at (8.8,3.8);
      \coordinate (l7) at (4.3,0.8);
      \coordinate (l8) at (8.8,0.8);
      
      \draw[important line] (l1) -- (n1) ;
      \draw[important line] (l3) -- (n1) ;
      \draw[important line] (l2) -- (n1) ;
      \draw[important line] (l4) -- (n1) ;
      
      \draw[dashed] (l5) -- (n2) ;
      \draw[dashed] (l6) -- (n2) ;
      \draw[dashed] (l7) -- (n2) ;
      \draw[dashed] (l8) -- (n2) ;
      
      \draw[important line] (l1) -- (l2) ;
      \draw[important line] (l1) -- (l3) ;
      \draw[important line] (l2) -- (l4) ;
      \draw[important line] (l3) -- (l4) ;
      
      \draw[important line] (l1) -- (l5) ;
      \draw[important line] (l2) -- (l6) ;
      \draw[important line] (l4) -- (l8) ;
      \draw[important line] (l5) -- (l6);
      \draw[important line] (l6) -- (l8);
      \draw[dashed] (l7) -- (l8);
      \draw[dashed] (l3) -- (l7);
      \draw[dashed] (l5) -- (l7);
      \draw[dashed] (n1) -- (n2);

      \draw[dashed] (l4) -- (8.4,-0.6);
      \node[right] at (8.5, -0.6){$x_0 = t_n$};
      \draw[dashed] (l8) -- (9.8,0.8);
      \node[right] at (9.9, 0.8){$x_0 = t_{n+1}$};
      
      
      \draw[important line,->] (-1,-1.5) -- (-1,4) ;
      \draw[important line,->] (-1,-1.5) -- (8,-1.5) ;
      \draw[important line,->] (-1,-1.5) -- (1.8,0.3) ;
      
      \draw[important line]  (8,-1.5)--(-1,-1.5) -- (-1,4) ;
      
      \node[left] at (-1, 4){$x_2$};
      \node[left] at (-1, 2.8){$x_2 = 0$};
      \draw[dashed] (-1,2.8) -- (2,2.8) ;
      \node[left] at (8, -1.8){$x_1$};
      \node[left] at (1.7, 0.4){$x_0$};      
      \draw[dashed] (-1,-0.6) -- (2,-0.6) ;
      \node[left] at (-1, -0.6){$x_2 = -H$};
      
    \end{tikzpicture}    
  \end{center}
  \caption{Depiction of a space-time slab
    $\mathcal{E}^n \subset \mathbb{R}^3$.}
  \label{fig:space_time_domain}
\end{figure}
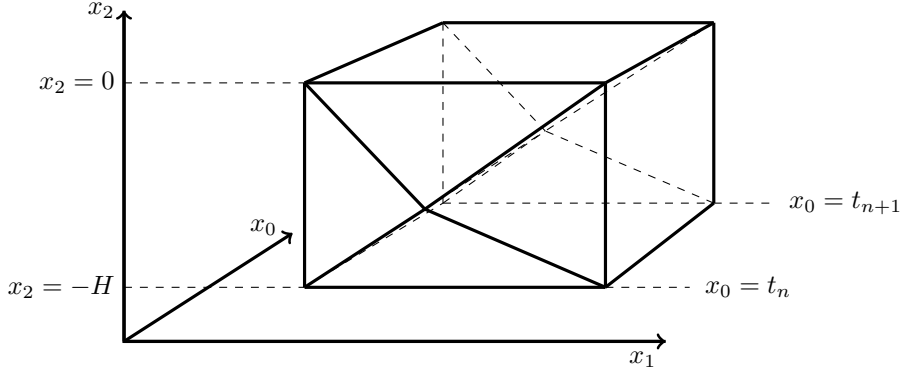

Consider a space-time element $\mathcal{K}_j^n \in \mathcal{T}^n$. Let
$K_j^n := \cbr[0]{ (x_0, \boldsymbol{x}) \in \partial\mathcal{K}_j^n :
  x_0 = t_n} \in \mathcal{T}$. The boundary of a space-time element
$\mathcal{K}_j^n \in \mathcal{T}^n$ is then composed of $K_j^n$,
$K_j^{n+1}$ and
$\mathcal{Q}_{\mathcal{K}_j^n} := \partial\mathcal{K}_j^n \backslash
\partial \mathcal{K}_0^n$ where
$\partial \mathcal{K}_0^n := K_j^n \cup K_j^{n+1}$.

The outward unit space-time normal vector field on
$\partial\mathcal{K}_j^n$ is denoted by
$\hat{\boldsymbol{n}}_j^n=((n_t)_j^n, \boldsymbol{n}_j^n)$, where
$(n_t)_j^n$ and $\boldsymbol{n}_j^n$ are, respectively, the temporal
and spatial parts of the space-time normal vector. Since the mesh does
not change with time, $\hat{\boldsymbol{n}}_j^n = (1, \boldsymbol{0})$
on $K_j^{n+1}$, $\hat{\boldsymbol{n}}_j^n = (-1, \boldsymbol{0})$ on
$K_j^{n}$, and $\hat{\boldsymbol{n}}_j^n = (0, \boldsymbol{n}_j^n)$ on
$\mathcal{Q}_{\mathcal{K}_j^n}$.

In the remainder of this paper we will omit the subscripts and
superscripts when referring to space-time elements, their boundaries,
and the normal vector wherever no confusion will occur.

In a space-time slab $\mathcal{E}^n$, the set and union of all faces
in $\partial\mathcal{E}_{\mathcal{S}}$ are denoted by
$\mathcal{F}_{\mathcal{S}}^n$ and $\Gamma_{\mathcal{S}}^n$,
respectively. Furthermore, the set and union of all interior and
boundary faces in $\mathcal{E}^n$ that are not on
$\Omega_n \cup \Omega_{n+1}$ are denoted by
$\mathcal{F}_{\mathcal{Q}}^n$ and $\Gamma_{\mathcal{Q}}^n$,
respectively. On the other hand, faces on $\Omega_n$ and
$\Omega_{n+1}$ are denoted by $\mathcal{F}_{\Omega}^n(t_n)$ and
$\mathcal{F}_{\Omega}^n(t_{n+1})$,
respectively. Finally,
$\partial\mathcal{E}_{\mathcal{S}}^n(t_n)$ denotes the set of edges
$e \in \partial\mathcal{E}_{\mathcal{S}}^n$ at $t = t_n$.

For triangular prismatic space-time elements we introduce the
following local spaces:
\begin{subequations}
  \begin{align}
    W_h(\mathcal{K})
    & := P_p(K) \otimes P_p(I_n),
    \\
    \boldsymbol{V}_h(\mathcal{K})
    & := \sbr[1]{P_p(K) \otimes P_p(I_n)}^2,
    \\
    M_h(\mathcal{F})
    &:= 
      Q_p(\mathcal{F}) \qquad \forall \mathcal{F} \subset \mathcal{Q}_{\mathcal{K}},
  \end{align}
\end{subequations}
where $P_p(D)$ is the space of polynomials of degree at most $p$ on a
domain $D$ and $Q_p(D)$ denotes the tensor-product polynomials of
degree $p$. The global finite element spaces are then defined as:
\begin{subequations}
  \begin{align}
    W_h
    & := \cbr{w \in L^2(\mathcal{E}^n)\, : w\vert_\mathcal{K} 
      \in W_h(\mathcal{K}),\,\forall \mathcal{K} \in \mathcal{T}^n},
    \\
    \boldsymbol{V}_h
    & := \cbr[2]{\boldsymbol{v} \in \sbr[1]{L^2(\mathcal{E}^n)}^2\, : 
      \boldsymbol{v}\vert_\mathcal{K} 
      \in \boldsymbol{V}_h(\mathcal{K}),
      \,\forall \mathcal{K}\in\mathcal{T}^n},
    \\
    M_h
    &:= 
      \cbr{\mu \in L^2(\Gamma_{\mathcal{Q}}^n) : \mu \vert_{\mathcal{F}} 
      \in M_h(\mathcal{F}), 
      \,\forall \mathcal{F} \in \mathcal{F}_{\mathcal{Q}}^n}.
  \end{align}
\end{subequations}

For scalar functions we introduce the following inner-products:
\begin{equation}
  \begin{aligned}
    \del[1]{v, w}_{\mathcal{K}}
    & = \int_{\mathcal{K}}vw \dif \mathcal{K},
    & 
    \langle v, w \rangle_{\mathcal{Q}}
    & = \int_{\mathcal{Q}}vw \dif \mathcal{Q},
    %
    & 
    \langle v, w \rangle_{K_j^n}
    & = \int_{K_j^{n}}vw \dif K,
    \\
    \llangle v, w \rrangle_{e}
    & = \int_{e} vw \dif e,
    &
    \langle v, w \rangle_{\mathcal{F}}
    & = \int_{\mathcal{F}}vw \dif \mathcal{F},
  \end{aligned}
\end{equation}
and
\begin{equation}
  \begin{aligned}
    \del[1]{v, w}_{\mathcal{T}^n}
    & = \sum_{\mathcal{K} \in \mathcal{T}^n}\del[1]{v, w}_{\mathcal{K}},
    & 
    \langle v, w \rangle_{\mathcal{F}_{\mathcal{Q}}^n}
    & = \sum_{\mathcal{K} \in \mathcal{T}^n} \langle v, w \rangle_{\mathcal{Q}},
    \\
    \langle v, w \rangle_{\mathcal{F}_{\Omega}^n(t_{n})}
    & = \sum_{\mathcal{K} \in \mathcal{T}^n} \langle v, w \rangle_{K_j^n},
    &
    \llangle v, w \rrangle_{\partial \mathcal{E}_{\mathcal{S}}^n(t_{n})}
    & = \sum_{e \in \partial\mathcal{E}_{\mathcal{S}}^n(t_{n})} \llangle v, w \rrangle_{e},
    \\
    \langle v, w \rangle_{\mathcal{F}_{\mathcal{S}}^n}
    & = \sum_{\mathcal{F} \in \mathcal{F}_{\mathcal{S}}^n}\langle v, w \rangle_{\mathcal{F}}.
  \end{aligned}
\end{equation}
Similar notation is used for vector functions, for example,
\begin{equation*}
  \del[1]{\boldsymbol{v}, \boldsymbol{w}}_{\mathcal{K}}
  = \int_{\mathcal{K}}\boldsymbol{v} \cdot \boldsymbol{w} \dif \mathcal{K}
  \qquad \text{and} \qquad
  \del[1]{\boldsymbol{v}, \boldsymbol{w}}_{\mathcal{T}^n} =
  \sum_{\mathcal{K} \in \mathcal{T}^n}\del[1]{\boldsymbol{v}, \boldsymbol{w}}_{\mathcal{K}}.
\end{equation*}
The $L^2$-norm of a function $v$ on a domain $D$ will be denoted by
$\norm{v}_D$.

\subsection{Discretization}
\label{subsec:weak_form}

In this section we present the space-time HDG discretization for
\cref{eq:Problem} on a space-time slab $\mathcal{E}^n$. To be able to
show existence and uniqueness of a solution to our discretization we
use weighted inner products, as done originally for the wave equation
in \cite{French:1993}.

Let $f_n$ be a weight function depending only on time and which is
defined as $f_n(t) = e^{-\alpha(t-t_n)}$ with $\alpha > 0$. The
space-time HDG discretization for the linear free-surface problem
\cref{eq:mixedformlinfreesurface} is: Find
$\del{\boldsymbol{q}_h, v_h, \lambda_h} \in \boldsymbol{V}_h \times
W_h \times M_h$ such that for all
$\del{\boldsymbol{r}_h, w_h, \mu_h} \in \boldsymbol{V}_h \times W_h
\times M_h$ the following relations are satisfied:
\begin{subequations}
	\begin{multline}
	\label{eq:HDG_q_equation}
	- \del[1]{\boldsymbol{q}_h, f_n\partial_t \boldsymbol{r}_h}_{\mathcal{T}^n}
	- \del[1]{\boldsymbol{q}_h, \boldsymbol{r}_h f'_n}_{\mathcal{T}^n}
	+ \langle \boldsymbol{q}_h, \boldsymbol{r}_h f_n \rangle_{\mathcal{F}_{\Omega}^n(t_{n+1})}
	\\
	+ \del[1]{v_h, f_n \nabla \cdot \boldsymbol{r}_h}_{\mathcal{T}^n}
	- \langle \lambda_h, \boldsymbol{r}_h\cdot \boldsymbol{n} f_n \rangle_{\mathcal{F}_{\mathcal{Q}}^n} =
	\langle \boldsymbol{q}^-_h, \boldsymbol{r}_h f_n \rangle_{\mathcal{F}_{\Omega}^n(t_{n})},                  
	\end{multline}
	\begin{equation}
	\label{eq:HDG_v_equation}
	-\del[1]{w_h, f_n\nabla \cdot \boldsymbol{q}_h}_{\mathcal{T}^n}
	+ \langle \tau\del{v_h - \lambda_h}, w_h f_n \rangle_{\mathcal{F}_{\mathcal{Q}}^n} = 0,                                    
	\end{equation}
	%
	%
	\begin{multline}
	\label{eq:HDG_lambda_equation}
	\langle \boldsymbol{q}_h 
	\cdot \boldsymbol{n} - \tau\del[1]{v_h - \lambda_h}, \mu_h f_n  \rangle_{\mathcal{F}_{\mathcal{Q}}^n}
	- \langle \lambda_h , f_n \partial_t \mu_h \rangle_{\mathcal{F}_{\mathcal{S}}^n}
	\\
	- \langle \lambda_h, \mu_h f'_n \rangle_{\mathcal{F}_{\mathcal{S}}^n}
	+ \llangle \lambda_h, \mu_h f_n \rrangle_{\partial \mathcal{E}_{\mathcal{S}}^n(t_{n+1})}
	=
	\llangle \lambda^-_h, \mu_h f_n \rrangle_{\partial \mathcal{E}_{\mathcal{S}}^n(t_{n})},
	\quad \forall \mu_h \in M_h,
	\end{multline}
	\label{eq:HDGProblem}
\end{subequations}
where $\tau > 0$ is a stabilization parameter, and $\boldsymbol{q}_h^-$ and $\lambda_h^-$ denote the known values of $\boldsymbol{q}_h$ and $\lambda_h$, respectively, at $x_0 = t_n$ from the previous space-time slab $\mathcal{E}^{n-1}$, or the initial condition if $n = 0$.

\subsection{Well posedness}
\label{ss:wellposedness}

We now show the existence of a unique solution to the space-time HDG
method \cref{eq:HDGProblem}.

\begin{theorem}[Existence and uniqueness]
  \label{thm:ex_uniq}
  A unique solution
  $\del{\boldsymbol{q}_h, v_h, \lambda_h} \in \boldsymbol{V}_h \times
  W_h \times M_h$ to \cref{eq:HDGProblem} exists if the stabilization
  parameter $\tau$ is positive.
\end{theorem}
\begin{proof}
  It is sufficient to show that if the data is equal to zero, the only
  solution to \cref{eq:HDGProblem} is the trivial one. We only need to
  show this for an arbitrary space-time slab $\mathcal{E}^n$ assuming
  $\lambda_h^- = 0$ and $\boldsymbol{q}_h^- = 0$.

  Take $\boldsymbol{r}_h = \boldsymbol{q}_h$ in
  \cref{eq:HDG_q_equation}, $w_h = v_h$ in \cref{eq:HDG_v_equation}
  and $\mu_h = \lambda_h$ in \cref{eq:HDG_lambda_equation}, and add
  the three equations together:
  \begin{multline}
    \label{eq:test_equals_trial}
    - \del[1]{\boldsymbol{q}_h, f_n\partial_t \boldsymbol{q}_h}_{\mathcal{T}^n}
    - \del[1]{\, \envert{\boldsymbol{q}_h}^2,  f'_n}_{\mathcal{T}^n}
    + \langle\, \envert{\boldsymbol{q}_h}^2, f_n \rangle_{\mathcal{F}_{\Omega}^n(t_{n+1})}
    \\
    + \langle \tau\del{v_h - \lambda_h}^2, f_n \rangle_{\mathcal{F}_{\mathcal{Q}}^n}
    - \langle \lambda_h f_n, \partial_t \lambda_h \rangle_{\mathcal{F}_{\mathcal{S}}^n}
    - \langle \lambda_h^2, f'_n \rangle_{\mathcal{F}_{\mathcal{S}}^n}
    \\
    + \llangle \lambda_h^2, f_n \rrangle_{\partial \mathcal{E}_{\mathcal{S}}^n(t_{n+1})} = 0.
  \end{multline}
  Note that since
  $\boldsymbol{q}_h \cdot \partial_t\boldsymbol{q}_h =
  \tfrac{1}{2}\partial_t(\,\envert{\boldsymbol{q}_h}^2)$, we may write the first
  term on the left hand side, after integration by parts in time, as
  \begin{multline}
    \label{eq:firsttermlhs}
    - \del[1]{\boldsymbol{q}_h, f_n\partial_t \boldsymbol{q}_h}_{\mathcal{T}^n}
    = \tfrac{1}{2}\del[1]{\, \envert{\boldsymbol{q}_h}^2, f'_n}_{\mathcal{T}^n}
    - \tfrac{1}{2}\langle\, \envert{\boldsymbol{q}_h}^2, f_n \rangle_{\mathcal{F}_{\Omega}^n(t_{n+1})}
    \\
    + \tfrac{1}{2}\langle\, \envert{\boldsymbol{q}_h}^2, f_n \rangle_{\mathcal{F}_{\Omega}^n(t_{n})}.
  \end{multline}
  Similarly, the fifth term on the left hand side of
  \cref{eq:test_equals_trial} may be written as
  \begin{multline}
    \label{eq:fifthtermlhs}
    - \langle \lambda_h f_n, \partial_t \lambda_h \rangle_{\mathcal{F}_{\mathcal{S}}^n}
    = \tfrac{1}{2}\langle \lambda_h^2, f'_n \rangle_{\mathcal{F}_{\mathcal{S}}^n}
    - \tfrac{1}{2}\llangle \lambda_h^2, f_n \rrangle_{\partial \mathcal{E}_{\mathcal{S}}^n(t_{n+1})}
    \\
    + \tfrac{1}{2}\llangle \lambda_h^2, f_n \rrangle_{\partial \mathcal{E}_{\mathcal{S}}^n(t_{n})}.
  \end{multline}
  Combining \cref{eq:test_equals_trial}--\cref{eq:fifthtermlhs}, we
  obtain
  \begin{multline}
    \label{eq:test_equals_trial2}
    0 =
    - \tfrac{1}{2}\del[1]{\, \envert{\boldsymbol{q}_h}^2, f'_n}_{\mathcal{T}^n}
    + \tfrac{1}{2}\langle\, \envert{\boldsymbol{q}_h}^2, f_n \rangle_{\mathcal{F}_{\Omega}^n(t_{n+1})}
    \\
    + \tfrac{1}{2}\langle\, \envert{\boldsymbol{q}_h}^2, f_n \rangle_{\mathcal{F}_{\Omega}^n(t_{n})}
    + \langle \tau\del{v_h - \lambda_h}^2,  f_n \rangle_{\mathcal{F}_{\mathcal{Q}}^n}
    \\
    - \tfrac{1}{2}\langle \lambda_h^2, f'_n \rangle_{\mathcal{F}_{\mathcal{S}}^n}
    + \tfrac{1}{2} \llangle \lambda_h^2, f_n \rrangle_{\partial \mathcal{E}_{\mathcal{S}}^n(t_{n+1})}
    + \tfrac{1}{2} \llangle \lambda_h^2, f_n \rrangle_{\partial \mathcal{E}_{\mathcal{S}}^n(t_{n})}.
  \end{multline}
  Since $\tau > 0$, $f_n > 0, \forall t$, and
  $f'_n < 0, \forall t$, we conclude that
  $\boldsymbol{q}_h = 0$ in $\mathcal{E}^n$ and on
  $K_j^n \cup K_j^{n+1}$,
  $v_h = \lambda_h$ on $\Gamma_{\mathcal{Q}}^n$, and
  $\lambda_h = 0$ on $\Gamma_{\mathcal{S}}^n$ and on
  $\partial\mathcal{E}_{\mathcal{S}}^n(t_n) \cup
  \partial\mathcal{E}_{\mathcal{S}}^n(t_{n+1})$.
  Substituting into \cref{eq:HDG_q_equation},
  \begin{equation}
    \begin{split}
    0  
    &=
    \del[1]{v_h, f_n \nabla \cdot \boldsymbol{r}_h}_{\mathcal{T}^n}
    - \langle \lambda_h, \boldsymbol{r}_h\cdot \boldsymbol{n} f_n \rangle_{\mathcal{F}_{\mathcal{Q}}^n}
    \\
    &=
    -\del[1]{\nabla v_h, f_n \boldsymbol{r}_h}_{\mathcal{T}^n}
    + \langle \del{v_h - \lambda_h}, \boldsymbol{r}_h\cdot \boldsymbol{n} f_n \rangle_{\mathcal{F}_{\mathcal{Q}}^n}
    \\
    &=
    -\del[1]{\nabla v_h, f_n \boldsymbol{r}_h}_{\mathcal{T}^n},
    \end{split}
  \end{equation}
  which holds for all $\boldsymbol{r}_h \in \boldsymbol{V}_h$. Since
  $f_n > 0$ we conclude that $\nabla v_h = 0$ in $\mathcal{E}^n$,
  implying that $v_h$ depends only on time. To show that $v_h = 0$ in
  $\mathcal{E}^n$, consider the following.
  Since $v_h$ depends only on time, then $v_h$ is constant on $\Omega \times \cbr[0]{t}$, for all $t \in (t_n, t_{n+1})$. In addition, $v_h = \lambda_h = 0$ on $\Gamma_{\mathcal{S}}^n$, so $v_h = 0$ on $\mathcal{E}^n$.

\end{proof}

\section{Analysis tools}
\label{sec:analysistools}

In this section, we develop some of the tools needed to perform the
error analysis in \cref{sec:erroranalysis}.

\subsection{Notation and anisotropic Sobolev spaces}
\label{ss:notation_spaces}

Let the multi-index $\alpha$ be a vector of non-negative integers
$\alpha_i$ and let $\abs{\alpha}$ be
defined as $\abs{\alpha} = \sum_i \alpha_i$. By $D^{\alpha}v$ we
denote the partial derivative of order $\envert{\alpha}$ of $v$, i.e.,
\begin{equation}
  D^{\alpha}v = \partial_{x_0}^{\alpha_0}\partial_{x_1}^{\alpha_1}\partial_{x_2}^{\alpha_2}v.
\end{equation}
We define the Sobolev space
$H^s(\Omega) = \cbr[1]{v \in L^2(\Omega) : D^{\alpha}v \in L^2(\Omega)
  \text{ for } \abs{\alpha} \leq s}$. This space is equipped with the
following norm and seminorm:
\begin{equation}
  \norm[0]{v}^2_{H^s(\Omega)} = \sum_{\abs{\alpha} \leq s}\norm[0]{D^{\alpha}v}^2_{L^2(\Omega)}
  \qquad \text{ and } \qquad
  \abs[0]{v}^2_{H^s(\Omega)} = \sum_{\abs{\alpha} = s}\norm[0]{D^{\alpha}v}^2_{L^2(\Omega)}.
\end{equation}
For $\alpha_{i} \geq 0$, $i = 0,1,2$, we introduce the
anisotropic Sobolev space of order $(s_t, s_s)$ on
$\mathcal{E} \subset \mathbb{R}^3$ by
\begin{equation}
  H^{(s_t,s_s)}(\mathcal{E}) = \cbr[1]{v \in L^2(\mathcal{E}) : D^{(\alpha_t,\alpha_s)} v \in L^2(\mathcal{E})
    \text{ for } \alpha_t \leq s_t,\, \envert[0]{\alpha_s} \leq s_s},
\end{equation}
where $\alpha_t = \alpha_0$ and $\alpha_s = (\alpha_{1}, \alpha_{2})$
and
$D^{(\alpha_t,\alpha_s)} v =
\partial_{x_0}^{\alpha_t}\partial_{x_1}^{\alpha_1}\partial_{x_2}^{\alpha_2}v$. The
anisotropic Sobolev norm and seminorm are given by, respectively,
\begin{equation}
  \norm[0]{v}^2_{H^{(s_t, s_s)}(\mathcal{E})} = \sum_{\substack{\alpha_t \leq s_t \\ \abs{\alpha_s} \leq s_s}}\norm[0]{D^{(\alpha_t,\alpha_s)} v}^2_{L^2(\mathcal{E})}
  \text{ and }
  \abs[0]{v}^2_{H^{(s_t, s_s)}(\mathcal{E})} = \sum_{\substack{\alpha_t = s_t \\ \abs{\alpha_s} = s_s}}\norm[0]{D^{(\alpha_t,\alpha_s)} v}^2_{L^2(\mathcal{E})}.  
\end{equation}

Let $\mathcal{K}$ denote any space-time prism that is constructed as
$\mathcal{K} = K \times I$, where $K$ is a spatial triangular
element and $I$ is an interval. Let $h_K$ and $\rho_K$ denote, respectively, the radii of the
2-dimensional circumcircle and inscribed circle of $K$. We will assume
spatial shape regularity, i.e., we assume there exists a constant
$c_r > 0$ such that
\begin{equation}
	\label{eq:shape_reg}
	\frac{h_K}{\rho_K} \leq c_r, \quad \forall \mathcal{K} \in \mathcal{T}^n.
\end{equation}
Additionally, we assume that $\mathcal{T}^n$ does not have any hanging
nodes. Throughout the remainder of this paper we will denote by
$C > 0$ a generic constant that is independent of $h_K$ and
$\Delta t$.

\subsection{Space-time mappings}
\label{ss:space_time_mappings}

Let $\widehat{K}$ denote the reference triangle defined by the
vertices $(0,0), \,(1,0), \,(0,1)$, and with reference coordinates
$\widehat{\boldsymbol{x}} = (\widehat{x}_1,
\widehat{x}_2)$. Furthermore, let $H_K$ denote the affine mapping
$H_K : \widehat{K} \rightarrow K$ defined as
$H_K(\widehat{\boldsymbol{x}}) = B_K \widehat{\boldsymbol{x}} +
\boldsymbol{c}$, where $B_K$ is a matrix and $\boldsymbol{c}$ is a
vector.

To construct each space-time prism we follow a similar approach as
\cite{Sudirham:2006}. Consider a reference prism
$\widehat{\mathcal{K}}$ defined by the vertices
$(-1,0,0), \,(-1,1,0), \,(-1,0,1), \,(1,0,0), \,(1,1,0),
\,(1,0,1)$. The reference coordinates in $\widehat{\mathcal{K}}$ are
denoted by $(\widehat{x}_0, \widehat{\boldsymbol{x}})$. The space-time
prism $\mathcal{K}$ is obtained as follows. First, we construct an
intermediate element $\check{\mathcal{K}}$ from an affine mapping
$F_{\mathcal{K}} : \widehat{\mathcal{K}} \mapsto \check{\mathcal{K}}$
defined as
$F_{\mathcal{K}}(\widehat{x}_0, \widehat{\boldsymbol{x}}) =
A_{\mathcal{K}}\sbr[0]{\widehat{x}_0, \widehat{\boldsymbol{x}}}^T +
\boldsymbol{b}$, where
$A_{\mathcal{K}} = \text{diag}\del[1]{\frac{\Delta t}{2}, 1, 1}$ and
$\boldsymbol{b}$ is a vector of the form $[b_0, 0, 0]^T$. The
coordinates on $\check{\mathcal{K}}$ are
$(\check{x}_0, \check{\boldsymbol{x}})$. Then, $\mathcal{K}$ is
obtained via the affine mapping,
$G_{\mathcal{K}} : \check{\mathcal{K}} \mapsto \mathcal{K}$ defined
as:
\begin{equation}
  G_{\mathcal{K}}(\check{x}_0, \check{\boldsymbol{x}}) =
  \begin{bmatrix}
    1 & \boldsymbol{0} \\
    \boldsymbol{0}^T & B_K
  \end{bmatrix}
  \begin{bmatrix}
    \check{x}_0 \\ \check{\boldsymbol{x}}
  \end{bmatrix}
  +
  \begin{bmatrix}
    0 \\ \boldsymbol{c}
  \end{bmatrix},
\end{equation}
where $\boldsymbol{0} = [0, 0]$ and $B_K$ denotes the matrix
associated with the mapping $H_K$ defined above. See
\cref{fig:space_time_mappings}.

We denote by $\partial \widehat{\mathcal{K}}_1$ the boundary face of
$\widehat{\mathcal{K}}$ with $\widehat{x}_1 = 0$. Similarly,
$\partial \widehat{\mathcal{K}}_2$ and
$\partial \widehat{\mathcal{K}}_3$ are the boundary faces of
$\widehat{\mathcal{K}}$ with, respectively, $\widehat{x}_2 = 0$ and
$\widehat{x}_1 + \widehat{x}_2 = 1$. By
$\partial\widehat{\mathcal{K}}_0$ we denote the boundary faces of
$\widehat{\mathcal{K}}$ with $\widehat{x}_0 = -1$ and
$\widehat{x}_0 = 1$. Furthermore, $\partial \check{\mathcal{K}}_i$,
$i = 0,1,2,3$, will denote the boundary faces of $\check{\mathcal{K}}$
which are obtained by applying the transformation $F_{\mathcal{K}}$ to
$\widehat{\mathcal{K}}$.

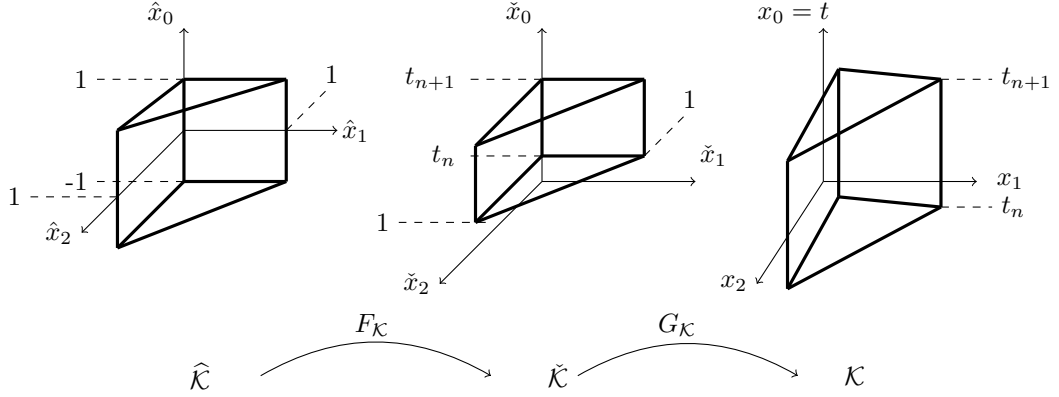
\begin{figure}[tbp]
  \begin{center}
    \begin{tikzpicture}[important line/.style={very thick}, scale=0.68]
      
      \tikzmath{\x1 = 0; \y1 = 0;};
      \draw[important line] (\x1,\y1) -- (\x1+2,\y1);
      \draw[important line] (\x1,\y1) -- (\x1,\y1+2);
      \draw[important line] (\x1,\y1) -- (\x1-1.3,\y1-1.3);
      \draw[important line] (\x1-1.3,\y1-1.3) -- (\x1+2,\y1);
      \draw[important line] (\x1,\y1+2) -- (\x1+2,\y1+2);
      \draw[important line] (\x1+2,\y1) -- (\x1+2,\y1+2);
      \draw[important line] (\x1,\y1+2) -- (\x1-1.3,\y1+1);
      \draw[important line] (\x1-1.3,\y1-1.3) -- (\x1-1.3,\y1+1);
      \draw[important line] (\x1-1.3,\y11) -- (\x1+2,\y1+2);
      
      \node[left] at (\x1-1.7,\y1) {-1};
      \draw[dashed] (\x1,\y1) -- (\x1-1.7,\y1);
      \node[left] at (\x1-1.7,\y1+2) {1};
      \draw[dashed] (\x1,\y1+2) -- (\x1-1.7,\y1+2);	
      \node[left] at (\x1-1.3-1.7,\y1-0.3) {1};
      \draw[dashed] (\x1-1.3,\y1-0.3) -- (\x1-1.3-1.7,\y1-0.3);	
      \node[left] at (\x1+2+1.2,\y1+2.1) {1};
      \draw[dashed] (\x1+2,\y1+1.0) -- (\x1+2+0.8,\y1+1.8);	
      %
            
      \draw[->] (\x1,\y1+1.0) -- (\x1,\y1+3) ;
      \draw[->] (\x1,\y1+1.0) -- (\x1+3,\y1+1.0) ;
      \draw[->] (\x1,\y1+1.0) -- (\x1-2,\y1-1.0) ;
      
      
      \node[left] at (\x1, \y1+3.3){$\hat{x}_0$};
      \node[left] at (\x1+3.8, \y1+1.0){$\hat{x}_1$};
      \node[left] at (\x1-2,\y1-1){$\hat{x}_2$};
      
      \node[left] (1) at (\x1+0.7,\y1-3.8){$\widehat{\mathcal{K}}$};
      
      \tikzmath{\x2 = \x1+7; \y2 = \y1+0.5;};
      \draw[important line] (\x2,\y2) -- (\x2+2,\y2);
      \draw[important line] (\x2-1.3,\y2-1.3) -- (\x2+2,\y2);
      \draw[important line] (\x2,\y2) -- (\x2-1.3,\y2-1.3);
      
      \draw[important line] (\x2,\y2+1.5) -- (\x2+2,\y2+1.5);
      \draw[important line] (\x2-1.3,\y2-1.3+1.5) -- (\x2+2,\y2+1.5);
      \draw[important line] (\x2,\y2+1.5) -- (\x2-1.3,\y2-1.3+1.5);
      
      \draw[important line] (\x2-1.3,\y2-1.3) -- (\x2-1.3,\y2-1.3+1.5);
      \draw[important line] (\x2+2,\y2+1.5) -- (\x2+2,\y2);
      \draw[important line] (\x2,\y2) -- (\x2,\y2+1.5);
      
      %
      
      \node[left] at (\x2-1.5,\y2) {$t_{n}$};
      \draw[dashed] (\x2,\y2) -- (\x2-1.5,\y2);
      \node[left] at (\x2-1.5,\y2+1.5) {$t_{n+1}$};
      \draw[dashed] (\x2,\y2+1.5) -- (\x2-1.5,\y2+1.5);
      \node[left] at (\x2-1.1-1.7,\y2-1.3) {1};
      \draw[dashed] (\x2-1.1,\y2-1.3) -- (\x2-1.1-1.7,\y2-1.3);	
      \node[left] at (\x2+2+1.2,\y2+1.1) {1};
      \draw[dashed] (\x2+2,\y2) -- (\x2+2+0.8,\y2+0.8);	
      
      \draw[->] (\x2,\y1) -- (\x2,\y1+3) ;
      \draw[->] (\x2,\y1) -- (\x2+3,\y1) ;
      \draw[->] (\x2,\y1) -- (\x2-2,\y1-2) ;
      
      \node[left] at (\x2, \y2+2.8){$\check{x}_0$};
      \node[left] at (\x2+3.8, \y2){$\check{x}_1$};
      \node[left] at (\x2-2,\y1-2){$\check{x}_2$};
      
      \node[left] (2) at (\x2+0.7,\y1-3.8){$\check{\mathcal{K}}$};
      
      \tikzmath{\x3 = \x2+5.5+0.3; \y3 = \y1-0.3;};
      
      \draw[important line] (\x3,\y3) -- (\x3+2,\y3-0.2);
      \draw[dashed] (\x3+2, \y3-0.2) -- (\x3+3, \y3-0.2);
      \node[right] at (\x3+3, \y3-0.2){$t_n$};
      \draw[important line] (\x3,\y3) -- (\x3-1,\y3-1.8);
      \draw[important line] (\x3-1,\y3-1.8) -- (\x3+2,\y3-0.2);
      
      \draw[important line] (\x3,\y3+2.5) -- (\x3+2,\y3-0.2+2.5);
      \draw[dashed] (\x3+2, \y3-0.2+2.5) -- (\x3+3, \y3-0.2+2.5);
      \node[right] at (\x3+3, \y3-0.2+2.5){$t_{n+1}$};
      \draw[important line] (\x3,\y3+2.5) -- (\x3-1,\y3-1.8+2.5);
      \draw[important line] (\x3-1,\y3-1.8+2.5) -- (\x3+2,\y3-0.2+2.5);
      
      \draw[important line] (\x3,\y3) -- (\x3,\y3+2.5);
      \draw[important line] (\x3-1,\y3-1.8) -- (\x3-1,\y3-1.8+2.5);
      \draw[important line] (\x3+2,\y3-0.2) -- (\x3+2,\y3-0.2+2.5);
      
      %
      %
      %
      %
      
      \draw[->] (\x2+5.5,\y1) -- (\x2+5.5,\y1+3) ;
      \draw[->] (\x2+5.5,\y1) -- (\x2+5.5+3,\y1) ;
      \draw[->] (\x2+5.5,\y1) -- (\x2+5.5-1.3,\y1-2) ;
      
      \node[left] at (\x3, \y1+3.3){$x_0 = t$};
      \node[left] at (\x3+3.8, \y1){$x_1$};
      \node[left] at (\x2+5.5-1.3,\y1-2){$x_2$};
      
      \node[left] (3) at (\x3+0.7,\y1-3.8){$\mathcal{K}$};
      
      \draw[->] (1.5,-3.8) to [bend left,looseness=1.0] (6,-3.8);
      \draw[->] (7.7,-3.8) to [bend left,looseness=1.0] (12,-3.8);
      
      \node[left] at (4.2,-2.8) {$F_{\mathcal{K}}$};
      \node[left] at (10.2,-2.8) {$G_{\mathcal{K}}$};
    \end{tikzpicture}    
  \end{center}
  \caption{An illustration of the different space-time mappings.}
  \label{fig:space_time_mappings}
\end{figure}
%
\subsection{Trace and inverse trace identities}
\label{ss:scaling_identities}

In this section we prove anisotropic trace and inverse trace
identities. This is achieved by first considering different scaling
identities. Similar identities were shown on hexahedra in
\cite{Georgoulis:2003, Sudirham:2006}, but are modified here for
prisms.

\begin{lemma}
  \label{lem:identities_Khat_to_checkK}
  Let $\check{u} \in H^{(s_t, s_s)}(\check{\mathcal{K}})$,
  $\alpha_t = \alpha_0$, $\alpha_s = (\alpha_1, \alpha_2)$, and
  $\alpha = (\alpha_t, \alpha_s)$. Then, the following identities hold
  for $\alpha = (\alpha_t, \alpha_s)$, $\alpha_i \geq 0$, $i = 0,1,2$,
  \begin{subequations}
    \begin{align}
      \label{eq:scaling_element}
      \norm[0]{\check{D}^{\alpha}\check{u}}_{\check{\mathcal{K}}}^2
      & = \del{\dfrac{2}{\Delta t}}^{2\alpha_0 - 1} \norm[0]{\widehat{D}^{\alpha}\widehat{u}}_{\widehat{\mathcal{K}}}^2,        
      \\
      \label{eq:scaling_K0}
      \norm[0]{\check{D}^{\alpha}\check{u}}_{\partial \check{\mathcal{K}}_0}^2
      & = \del{\dfrac{2}{\Delta t}}^{2\alpha_0} \norm[0]{\widehat{D}^{\alpha}\widehat{u}}_{\partial \widehat{\mathcal{K}}_0}^2,        
      \\
      \label{eq:scaling_Ki}
      \norm[0]{\check{D}^{\alpha}\check{u}}_{\partial \check{\mathcal{K}}_j}^2
      & = \del{\dfrac{2}{\Delta t}}^{2\alpha_0 - 1} \norm[0]{\widehat{D}^{\alpha}\widehat{u}}_{\partial \widehat{\mathcal{K}}_j}^2,
        \quad j = 1,2,3,        
    \end{align}
  \end{subequations}
  where $\widehat{u} = \check{u} \circ F_{\mathcal{K}}$.
\end{lemma}

\begin{proof}
  Note that by the chain rule,
  \begin{equation}
    \label{eq:chainrule}
    \check{D}^{\alpha}\check{u} = \del{\dfrac{2}{\Delta t}}^{\alpha_0} \widehat{D}^{\alpha}\del{\check{u} \circ F_{\mathcal{K}}}
    \circ F_{\mathcal{K}}^{-1}.
  \end{equation}
  We first show \cref{eq:scaling_element}. By \cref{eq:chainrule}
  \begin{equation}
    \norm[0]{\check{D}^{\alpha}\check{u}}_{\check{\mathcal{K}}}^2 = \del{\dfrac{2}{\Delta t}}^{2\alpha_0} \int_{\check{\mathcal{K}}}
    \del{\widehat{D}^{\alpha}\del{\check{u} \circ F_{\mathcal{K}}} \circ F_{\mathcal{K}}^{-1}}^2 \dif \check{x}_0 \dif \check{\boldsymbol{x}}.
  \end{equation}
  Changing variables, we obtain
  \begin{equation}
    \begin{split}
      \norm[0]{\check{D}^{\alpha}\check{u}}_{\check{\mathcal{K}}}^2 
      & =
      \del{\dfrac{2}{\Delta t}}^{2\alpha_0} \int_{\widehat{\mathcal{K}}} \del{\widehat{D}^{\alpha}\del{\check{u} \circ F_{\mathcal{K}}} \circ F_{\mathcal{K}}^{-1} \circ F_{\mathcal{K}}}^2
      \envert{\det A_{\mathcal{K}}} \dif \widehat{x}_0 \dif \widehat{\boldsymbol{x}} \\
      & =
      \del{\dfrac{2}{\Delta t}}^{2\alpha_0} \int_{\widehat{\mathcal{K}}} \del{\widehat{D}^{\alpha}\del{\check{u} \circ F_{\mathcal{K}}}}^2
      \envert{\det A_{\mathcal{K}}} \dif \widehat{x}_0 \dif \widehat{\boldsymbol{x}} \\
      & =
      \del{\dfrac{2}{\Delta t}}^{2\alpha_0} \frac{\Delta t}{2} \int_{\widehat{\mathcal{K}}} \del{\widehat{D}^{\alpha}\widehat{u}}^2
      \dif \widehat{x}_0 \dif \widehat{\boldsymbol{x}} \\
      & =
      \del{\dfrac{2}{\Delta t}}^{2\alpha_0 - 1}\norm[0]{\widehat{D}^{\alpha}\widehat{u}}_{\widehat{\mathcal{K}}}^2.
    \end{split}
  \end{equation}

  To show \cref{eq:scaling_K0} we note that time is fixed on
  $\partial \check{\mathcal{K}}_0$. Furthermore, since
  $\check{\boldsymbol{x}} = \widehat{\boldsymbol{x}}$ we note that for
  any function $\check{v}$ on $\check{\mathcal{K}}$ we have
  $\check{v}(t_n, \check{\boldsymbol{x}}) = \widehat{v}(1,
  \widehat{\boldsymbol{x}})$, i.e.,
  $\check{v}|_{\partial \check{\mathcal{K}}_0} =
  \widehat{v}|_{\partial \widehat{\mathcal{K}}_0}$. In particular, if
  $\check{v} = \check{D}^{\alpha}\check{u}$,
  \begin{equation}
    \label{eq:chain_rule_K0}
    \del[1]{\check{D}^{\alpha}\check{u}} |_{\partial \check{\mathcal{K}}_0}
    = \del{\frac{2}{\Delta t}}^{\alpha_0} \del[1]{\widehat{D}^{\alpha}\del[1]{\check{u} \circ F_{\mathcal{K}}}}|_{\partial \widehat{\mathcal{K}}_0}.
  \end{equation}
  Therefore,
  \begin{equation}
    \begin{split}
      \norm[0]{\check{D}^{\alpha}\check{u}}_{\partial \check{\mathcal{K}}_0}^2
      & =
      \int_{\partial \check{\mathcal{K}}_0} \del[1]{\check{D}^{\alpha}\check{u}}^2 \dif \check{\boldsymbol{x}}
      \\
      & =
      \del{\frac{2}{\Delta t}}^{2\alpha_0} \int_{\partial \widehat{\mathcal{K}}_0} \del[1]{\widehat{D}^{\alpha}\widehat{u}}^2 \dif \widehat{\boldsymbol{x}}
      \\
      & =
      \del{\dfrac{2}{\Delta t}}^{2\alpha_0} \norm[0]{\widehat{D}^{\alpha}\widehat{u}}_{\partial \widehat{\mathcal{K}}_0}^2,
    \end{split}
  \end{equation}
  which concludes the proof for \cref{eq:scaling_K0}.

  Finally, we prove \cref{eq:scaling_Ki} for $j = 1$. The proofs for
  $j = 2,3$ are analogous. First let us define the mapping
  $J_{\partial\mathcal{K}}$ that maps
  $\partial\widehat{\mathcal{K}}_1$ onto
  $\partial\check{\mathcal{K}}_1$. This mapping is given by
  \begin{equation}
    J_{\partial\mathcal{K}}(\widehat{x}_0,\widehat{x}_2) = 
    \begin{bmatrix}
      \frac{\Delta t}{2} & 0 \\
      0 & 1
    \end{bmatrix}
    \begin{bmatrix}
      \widehat{x}_0 \\ \widehat{x}_2
    \end{bmatrix}
    + \boldsymbol{d},
  \end{equation}
  where $\boldsymbol{d}$ is a constant two dimensional vector. Then,
  by the chain rule,
  \begin{equation}
    \label{eq:chain_rule_K1}
    \begin{split}
      \del[1]{\check{D}^{\alpha}\check{u}}|_{\partial\check{\mathcal{K}}_1}
      & = \del[1]{\check{D}^{\alpha}\check{u}}(\check{x}_0,0,\check{x}_2)		  
      \\
      & = \del{\frac{2}{\Delta t}}^{\alpha_0}\del{\widehat{D}^{\alpha}\del[1]{\check{u} \circ F_{\mathcal{K}}}}|_{\partial \widehat{\mathcal{K}}_1} \circ J_{\partial\mathcal{K}}^{-1}
      \\
      & = \del{\frac{2}{\Delta t}}^{\alpha_0}\sbr{\del{\widehat{D}^{\alpha}\del[1]{\check{u} \circ F_{\mathcal{K}}}}(\widehat{x}_0,0,\widehat{x}_2)} \circ J_{\partial\mathcal{K}}^{-1}.
    \end{split}
  \end{equation}
  We now find:
  \begin{equation}
    \begin{split}
      \norm[0]{\check{D}^{\alpha}\check{u}}_{\partial \check{\mathcal{K}}_1}^2
      & =
      \int_{\partial \check{\mathcal{K}}_1} \del[1]{\check{D}^{\alpha}\check{u}}^2 \dif \check{x}_0 \dif \check{x}_2
      \\
      & =
      \del{\frac{2}{\Delta t}}^{2\alpha_0}\int_{\partial \check{\mathcal{K}}_1} \del[2]{\sbr[1]{\del[1]{\widehat{D}^{\alpha}\del[1]{\check{u} \circ F_{\mathcal{K}}}}(\widehat{x}_0,0,\widehat{x}_2)} \circ J_{\partial\mathcal{K}}^{-1}}^2 \dif \check{x}_0 \dif \check{x}_2.
    \end{split}
  \end{equation}
  Note that the determinant of the Jacobian of
  $J_{\partial\mathcal{K}}$ is $\Delta t/2$. Changing variables,
  \begin{equation}
    \begin{split}
      \norm[0]{\check{D}^{\alpha}\check{u}}_{\partial \check{\mathcal{K}}_1}^2
      & =
      \del{\frac{2}{\Delta t}}^{2\alpha_0} \frac{\Delta t}{2} \int_{\partial 	\widehat{\mathcal{K}}_1} \del[1]{\del[0]{\widehat{D}^{\alpha}\widehat{u} }(\widehat{x}_0,0,\widehat{x}_2)}^2 \dif \widehat{x}_0 \dif \widehat{x}_2
      \\
      & = \del{\frac{2}{\Delta t}}^{2\alpha_0-1}\int_{\partial 	\widehat{\mathcal{K}}_1} \del[1]{\del[0]{\widehat{D}^{\alpha}\widehat{u} }(\widehat{x}_0,0,\widehat{x}_2)}^2 \dif \widehat{x}_0 \dif \widehat{x}_2
      \\
      & = \del{\frac{2}{\Delta t}}^{2\alpha_0-1}\norm[0]{\widehat{D}^{\alpha}\widehat{u}}_{\partial \widehat{\mathcal{K}}_1}^2.
    \end{split}
  \end{equation}
  The result follows.
\end{proof}

In what follows we use the following inequalities that can be shown by
standard scaling arguments:
\begin{subequations}
  \label{eq:standard_scaling_arguments}
  \begin{align}
    \label{eq:scaling_spatialK}    
    \envert{\det B_K} &\leq C h_K^2 ,
    \\
    \label{eq:scaling_spatialQK}
    \norm{\tilde{u}}_{\tilde{F}}^2 & \leq C h_K^{-1} \norm{u}_{F_K}^2,
    \\
    \label{eq:inv_scaling_spatialQK}
    \norm{u}_{F_K}^2 & \leq C h_K \norm{\tilde{u}}_{\tilde{F}}^2,
  \end{align}
\end{subequations}
where $F_K \in \partial K$, and where $\tilde{u} = \widehat{u}$ and
$\tilde{F} = F_{\widehat{K}} \in \partial \widehat{K}$, or
$\tilde{u} = \check{u}$ and
$\tilde{F} = F_{\check{K}} \in \partial \check{K}$.

\begin{lemma}
  \label{lem:identities_checkK_to_K}
  Let $u \in H^{(s_t, s_s)}(\mathcal{K})$. The following inequalities hold
  \begin{subequations}
    \begin{align}
      \label{eq:scaling_element_K}
      \norm[0]{u}_{\mathcal{K}}^2
      & \leq C h_K^2 \norm[0]{\check{u}}_{\check{\mathcal{K}}}^2,       
      \\
      \label{eq:inv_scaling_element_K}
      \norm[0]{\check{u}}_{\check{\mathcal{K}}}^2
      & \leq C h_K^{-2} \norm[0]{u}_{\mathcal{K}}^2,
      \\
      \label{eq:scaling_Qk}
      \norm[0]{\check{u}}_{\mathcal{F}_{\check{\mathcal{K}}}}^2
      &\leq C h_K^{-1} \norm[0]{u}_{\mathcal{F}_{\mathcal{K}}}^2,   
    \end{align}
  \end{subequations}
  where $\check{u} = u \circ G_{\mathcal{K}}$,
  $\mathcal{F}_{\check{\mathcal{K}}} \in
  \mathcal{Q}_{\check{\mathcal{K}}}$, and
  $\mathcal{F}_{\mathcal{K}} \in \mathcal{Q}_{\mathcal{K}}$.
\end{lemma}

\begin{proof}
  The results follow \cref{eq:standard_scaling_arguments}.
\end{proof}
We end this section by stating a trace inequality and two inverse trace
inequalities.
\begin{lemma}
  \label{lem:trace_ineq}
  Let $X_h(\mathcal{K}) \subset W_h(\mathcal{K})$ be a finite dimensional subspace such that the trace map
  $\gamma_{\mathcal{F}_{\mathcal{K}}}: X_h(\mathcal{K}) \mapsto
  M_h(\mathcal{F})$ defined by
  $\gamma_{\mathcal{F}_{\mathcal{K}}}(v_h) =
  v_h|_{\mathcal{F}_{\mathcal{K}}}$, for a face
  $\mathcal{F}_{\mathcal{K}} \subset \mathcal{Q}_{\mathcal{K}}$ is
  injective, then
  \begin{equation}
    \norm[0]{v_h}_{\mathcal{K}}^2 \leq C h_K \norm[0]{v_h}_{\mathcal{F}_{\mathcal{K}}}^2, \quad \forall v_h \in X_h(\mathcal{K}).
  \end{equation}
\end{lemma}
\begin{proof}
  By \cref{eq:scaling_element_K}, we obtain
  \begin{equation}
    \label{eq:vhhk2vcheck}
    \norm[0]{v_h}_{\mathcal{K}}^2
    \leq C h_K^2 \norm[0]{\check{v}_h}_{\check{\mathcal{K}}}^2.
  \end{equation}
  Since $v_h$ is a polynomial and $\gamma_{\mathcal{F}_{\mathcal{K}}}$
  is injective, we have
  $\norm[0]{\check{v}_h}_{\check{\mathcal{K}}}^2 \leq C
  \norm[0]{\check{v}_h}_{\mathcal{F}_{\check{\mathcal{K}}}}^2$. Combining
  this with \cref{eq:vhhk2vcheck}, the result follows after using
  \cref{eq:scaling_Qk}.
\end{proof}

\begin{lemma}
  \label{lem:inv_trace_ineq}
  Let $\mathcal{K} = K\times I_n$ a space-time element in
  $\mathcal{T}^n$,
  $\mathcal{F} \subset \partial\mathcal{E}_{\mathcal{S}}^n$ a face on
  the free-surface boundary and $\partial\mathcal{F}_0$ the two edges
  of the face $\mathcal{F}$ that are on the time levels. For
  $v_h \in W_h(\mathcal{K})$ and $\lambda_h \in M_h(\mathcal{F})$, the
  following inverse trace inequalities hold
  \begin{subequations}
  	\begin{align}
	  	\label{eq:inv_trace_ineq_F0}
	  	\norm[0]{\lambda_h}_{\partial \mathcal{F}_0}^2 & \leq C \Delta t^{-1} \norm[0]{\lambda_h}_{\mathcal{F}}^2,
	  	\\
  		\label{eq:inv_trace_ineq_K0}
  		\norm[0]{v_h}_{\partial \mathcal{K}_0}^2 & \leq C \Delta t^{-1} \norm[0]{v_h}_{\mathcal{K}}^2,
  		\\
  		\label{eq:inv_trace_ineq_QK}
  		\norm[0]{v_h}_{\mathcal{Q}_{\mathcal{K}}}^2 & \leq C h_K^{-1} \norm[0]{v_h}_{\mathcal{K}}^2.
  	\end{align}
  \end{subequations}
\end{lemma}
\begin{proof}
  Since the face $\mathcal{F}$ is a quadrilateral, the proof for
  \cref{eq:inv_trace_ineq_F0} can be found, e.g., in \cite[Corollary
  3.49]{Georgoulis:2003}.  \Cref{eq:inv_trace_ineq_K0} and
  \cref{eq:inv_trace_ineq_QK} can be obtained by the results in
  \Cref{lem:identities_Khat_to_checkK} and standard scaling arguments
  in space.
\end{proof}

\section{Error analysis}
\label{sec:erroranalysis}

In this section we present an \emph{a priori} error analysis for the
space-time HDG method \cref{eq:HDGProblem}. For this we require the
following spaces:
\begin{equation}
  \widetilde{W}_h(\mathcal{K})
  := P_{p-1}(K) \otimes P_p(I_n),
  \qquad
  \widetilde{\boldsymbol{V}}_h(\mathcal{K})
  := \sbr[1]{P_{p-1}(K) \otimes P_p(I_n)}^2.  
\end{equation}
We require also the $f_n$-weighted $L^2$-norm defined on a domain
$D$. For any $v \in L^2(D)$ this norm is defined as
$\norm[0]{v}_{f_n,D}^2 := \del{f_nv,v}_{D}$ while for
$\boldsymbol{q} \in \sbr[0]{L^2(D)}^2$ it is defined as
$\norm[0]{\boldsymbol{q}}_{f_n,D}^2 :=
\del{f_n\boldsymbol{q},\boldsymbol{q}}_{D}$.
	
\subsection{The projection}
\label{ss:projection}

The projection $\Pi_h$ onto $\boldsymbol{V}_h \times W_h$ used here is
based on the projection defined in \cite{Cockburn:2010}, but tailored
to the spaces used in this work. The projected function is denoted by
$\Pi_h\del[0]{\boldsymbol{q},v}$ or
$\del[0]{\boldsymbol{\Pi_V q}, \Pi_W v}$, and is defined by requiring
that the following equations are satisfied on each space-time element
$\mathcal{K} \in \mathcal{T}^n$:
\begin{subequations}
  \begin{align}
    \del[0]{\boldsymbol{\Pi_V q}, \boldsymbol{s}_h f_n}_{\mathcal{K}}
    &= \del[0]{\boldsymbol{q}, \boldsymbol{s}_h f_n}_{\mathcal{K}}
    && \forall \boldsymbol{s}_h \in \widetilde{\boldsymbol{V}}_h(\mathcal{K}),
    \\
    \del[0]{\Pi_W v, z_h f_n}_{\mathcal{K}}
    &= \del[0]{v, z_h f_n}_{\mathcal{K}}
    && \forall z_h \in \widetilde{W}_h(\mathcal{K}),
    \\
    \langle \boldsymbol{\Pi_V q}\cdot\boldsymbol{n} - \tau \Pi_Wv, \sigma_h f_n \rangle_{\mathcal{F}}
    &= \langle \boldsymbol{q}\cdot\boldsymbol{n} - \tau v, \sigma_hf_n \rangle_{\mathcal{F}}
    && \forall \sigma_h \in M_h(\mathcal{F}),\ \mathcal{F} \subset \mathcal{Q}_{\mathcal{K}}.
  \end{align}
  \label{eq:def_projection}
\end{subequations}

Notice that $\Pi_h$ is well defined for functions $\boldsymbol{q}$ and
$v$ such that their traces are in
$L^2(\mathcal{Q}_{\mathcal{K}})$. Therefore, the domain of $\Pi_h$ is in
$\sbr[1]{H^1(\mathcal{T}^n)}^2 \times H^1(\mathcal{T}^n)$, where
$H^1(\mathcal{T}^n) := \prod_{\mathcal{K} \in
  \mathcal{T}^n}H^1(\mathcal{K})$.

In order to show existence and uniqueness of the projection and its approximation properties, it will be useful to define the following spaces:
\begin{subequations}
	\begin{align}
	W_h^{\perp}(\mathcal{K})
	& := \cbr[1]{w \in W_h(\mathcal{K}) : \del[1]{w, \widetilde{w} f_n}_{\mathcal{K}} = 0, \forall \, \widetilde{w} \in \widetilde{W}_h(\mathcal{K})},
	\\
	\boldsymbol{V}^{\perp}_h(\mathcal{K})
	& := \cbr[1]{\boldsymbol{v} \in \boldsymbol{V}_h(\mathcal{K}) : \del[1]{\boldsymbol{v},\widetilde{\boldsymbol{v}} f_n}_{\mathcal{K}} = 0, \forall \, \widetilde{\boldsymbol{v}} \in \widetilde{\boldsymbol{V}}_h(\mathcal{K})}.
	\end{align}
\end{subequations}

The following lemma will be useful when showing the existence and
uniqueness of the projection and its approximation properties.
\begin{lemma}
  \label{lem:scaling_identities}
  For any space-time element $\mathcal{K}$, the following is satisfied
  for any face
  $\mathcal{F}_{\mathcal{K}} \in \mathcal{Q}_{\mathcal{K}}$:
  \begin{subequations}
    \begin{align}
      \label{eq:w_norm}
      w_h \in W_h^{\perp}(\mathcal{K}) \text{ and } w_h \vert_{\mathcal{F}_{\mathcal{K}}} = 0,
      & \quad \text{ implies } w_h = 0 \text{ on } \mathcal{K},
      \\
      \label{eq:v_norm}      
      \boldsymbol{v}_h \in \boldsymbol{V}_h^{\perp}(\mathcal{K})
      \text{ and } \boldsymbol{v}_h \cdot \boldsymbol{n}\vert_{\mathcal{Q}_{\mathcal{K}}\backslash \mathcal{F}_{\mathcal{K}}} = 0
      & \quad \text{ implies } \boldsymbol{v}_h = 0 \text{ on } \mathcal{K}.
    \end{align}
    \label{eq:w_v_norms}
  \end{subequations}
	%
	Moreover, the following estimates are satisfied
	\begin{subequations}
		\begin{align}
		\label{eq:wv_fn_estimates_a}
		\norm[0]{w_h}_{f_n,\mathcal{K}} \leq C h_K^{1/2} \norm[0]{w_h}_{f_n,\mathcal{F}_{\mathcal{K}}}
		& \quad \forall w_h \in W_h^{\perp}(\mathcal{K}),
		\\
		\label{eq:wv_fn_estimates_b}
		\norm[0]{\boldsymbol{v}_h}_{f_n,\mathcal{K}} \leq C h_K^{1/2} \norm[0]{\boldsymbol{v}_h \cdot \boldsymbol{n}}_{f_n,\mathcal{Q}_{\mathcal{K}}}
		& \quad \forall \boldsymbol{v}_h \in \boldsymbol{V}_h^{\perp}(\mathcal{K}).
		\end{align}
		\label{eq:wv_fn_estimates}
	\end{subequations}
\end{lemma}
\begin{proof}
	We first show \cref{eq:w_norm}. Take
	$w_h \in W_h^{\perp}(\mathcal{K})$ such that
	$w_h|_{\mathcal{F}_{\mathcal{K}}} = 0$ and let $L$ be a nonzero linear
	function that vanishes on $\mathcal{F}_{\mathcal{K}}$. Then, $w_h$ can
	be written as $w_h = L\widetilde{p}$, where
	$\widetilde{p} \in \widetilde{W}_h(\mathcal{K})$ \cite[Lemma
	3.1.10]{BrennerScott:book}. Since
	$w_h \in W^{\perp}_h(\mathcal{K})$, we have that
	$\del[1]{L\widetilde{p}, f_n\widetilde{p}}_{\mathcal{K}}= 0$. Since $L$
	cannot be zero on $\mathcal{K}$, we conclude that $\widetilde{p}$
	must be zero and therefore, $w_h$ is zero on $\mathcal{K}$.
	
	We next show \cref{eq:v_norm}. Let
	$p_h = \boldsymbol{v}_h \cdot \boldsymbol{n}_{\mathcal{F}}$, for any
	face $\mathcal{F}$ different than
	$\mathcal{F}_{\mathcal{K}}$. Notice that $p_h \in W_h^{\perp}(\mathcal{K})$
	and $p_h |_{\mathcal{F}} = 0$. Using a similar argument as above, we
	can conclude that $p_h = 0$ on $\mathcal{K}$. Therefore,
	$\boldsymbol{v}_h \cdot \boldsymbol{n}_{\mathcal{F}} = 0$ on
	$\mathcal{K}$. Since the set
	$\cbr[1]{\boldsymbol{n}_{\mathcal{F}} : \mathcal{F} \in
		\mathcal{Q}_{\mathcal{K}} \backslash \mathcal{F}_{\mathcal{K}}}$
	is a basis of $\mathbb{R}^2$, we conclude that $\boldsymbol{v}_h$
	must be zero on $\mathcal{K}$.
	
	To show \cref{eq:wv_fn_estimates_a}, we use the scaling identities
	from \cref{ss:scaling_identities}. Since
	$\norm[0]{\cdot}_{f_n,\mathcal{K}}$ is a weighted norm and
	$w_h f_n$, for $w_h \in W_h^{\perp}(\mathcal{K})$, is not a broken
	polynomial, as required in the proof of \Cref{lem:trace_ineq}, we
	cannot use \Cref{lem:trace_ineq} directly. However, since $f_n$
	is uniformly bounded, there exists a constant $C_{f_n} > 0$ such
	that $\sup_t f_n(t) \leq C_{f_n}$ for all $n$. Therefore,
	\begin{equation}
	\norm[0]{w_h}_{f_n,\mathcal{K}}^2 \leq C_{f_n} \norm[0]{w_h}_{\mathcal{K}}^2.
	\end{equation}
	Notice that \cref{eq:w_norm} implies that the trace map
	$\gamma_{\mathcal{F}_{\mathcal{K}}}: W_h^{\perp}(\mathcal{K})
	\mapsto M_h(\mathcal{F}_{\mathcal{K}})$ defined by
	$\gamma_{\mathcal{F}_{\mathcal{K}}}(w_h) = w_h
	|_{\mathcal{F}_{\mathcal{K}}}$ is injective. By
	\Cref{lem:trace_ineq} we then obtain
	\begin{equation}
	\norm[0]{w_h}_{f_n,\mathcal{K}}^2 \leq C h_K \norm[0]{w_h}_{\mathcal{F}_{\mathcal{K}}}^2.
	\end{equation}
	\Cref{eq:wv_fn_estimates_a} now follows by equivalence of norms on
	finite-dimensional spaces.
	
	Finally, we show \cref{eq:wv_fn_estimates_b}. Let
	$\mathcal{F}_1, \mathcal{F}_2 \in \mathcal{Q}_{\mathcal{K}}$ with
	$\mathcal{F}_1 \neq \mathcal{F}_{\mathcal{K}}$ and
	$\mathcal{F}_2 \neq \mathcal{F}_{\mathcal{K}}$. Notice that
	$\cbr[0]{\boldsymbol{n}_{\mathcal{F}_1},
		\boldsymbol{n}_{\mathcal{F}_2}}$ is a basis of $\mathbb{R}^2$,
	therefore, we can write $\boldsymbol{v}_h$ as
	$\boldsymbol{v}_h = \boldsymbol{v}_1 \boldsymbol{v}_h\cdot
	\boldsymbol{n}_{\mathcal{F}_1} + \boldsymbol{v}_2
	\boldsymbol{v}_h\cdot \boldsymbol{n}_{\mathcal{F}_2}$, where
	$\boldsymbol{v}_1, \boldsymbol{v}_2 \in \mathbb{R}^2$ are constant
	vectors. Thus,
	\begin{equation}
	\begin{split}
	\norm[1]{\boldsymbol{v}_h}_{f_n, \mathcal{K}} & = \norm[1]{\boldsymbol{v}_1 \boldsymbol{v}_h\cdot \boldsymbol{n}_{\mathcal{F}_1} + \boldsymbol{v}_2 \boldsymbol{v}_h\cdot \boldsymbol{n}_{\mathcal{F}_2}}_{f_n, \mathcal{K}}
	\\
	& \leq \norm[1]{\boldsymbol{v}_1}_{f_n, \mathcal{K}} \norm[1]{\boldsymbol{v}_h\cdot \boldsymbol{n}_{\mathcal{F}_1}}_{f_n, \mathcal{K}} + \norm[1]{\boldsymbol{v}_2}_{f_n, \mathcal{K}}\norm[1]{\boldsymbol{v}_h\cdot \boldsymbol{n}_{\mathcal{F}_2}}_{f_n, \mathcal{K}}
	\\
	& \leq C\del[1]{\, \norm[1]{\boldsymbol{v}_h\cdot \boldsymbol{n}_{\mathcal{F}_1}}_{f_n, \mathcal{K}} + \norm[1]{\boldsymbol{v}_h\cdot \boldsymbol{n}_{\mathcal{F}_2}}_{f_n, \mathcal{K}}}.
	\end{split}
	\end{equation}
	Applying \cref{eq:wv_fn_estimates_a} to the scalar functions $\boldsymbol{v}_h \cdot \boldsymbol{n}_{\mathcal{F}_1}$ and $\boldsymbol{v}_h \cdot \boldsymbol{n}_{\mathcal{F}_2}$, we have
	\begin{equation}
	\begin{split}
	\norm[1]{\boldsymbol{v}_h}_{f_n, \mathcal{K}} & \leq Ch_K^{1/2}\del[1]{\norm[1]{\boldsymbol{v}_h\cdot \boldsymbol{n}_{\mathcal{F}_1}}_{f_n, \mathcal{F}_1} + \norm[1]{\boldsymbol{v}_h\cdot \boldsymbol{n}_{\mathcal{F}_2}}_{f_n, \mathcal{F}_2}}
	\\
	& \leq Ch_K^{1/2} \norm[0]{\boldsymbol{v}_h \cdot \boldsymbol{n}}_{f_n,\mathcal{Q}_{\mathcal{K}}},
	\end{split}
	\end{equation}
	proving the result.
\end{proof}

We next prove existence and uniqueness of $\Pi_h$.

\begin{lemma}
  \label{lem:projectionexists}
  The projection $\Pi_h$ defined by \cref{eq:def_projection}
  exists and is unique.
\end{lemma}
\begin{proof}
  To see that $\Pi_h$ exists and is unique, we first verify that
  \cref{eq:def_projection} is a square system. First, recall that in
  two dimensions,
  \begin{equation}
    \dim P_p(K_j^n) = \tfrac{1}{2}(p+1)(p+2), \quad \dim Q_p(\mathcal{F}) = (p+1)^2.
  \end{equation}
  Thus,
  \begin{equation}
    \begin{split}
      \dim W_h(\mathcal{K}) &= \tfrac{1}{2}(p+1)^2(p+2),
      \\ 
      \dim \boldsymbol{V}_h(\mathcal{K}) &= (p+1)^2(p+2),
      \\ 
      \dim M_h(\mathcal{F}) &= (p+1)^2.
    \end{split}
  \end{equation}
  Moreover,
  \begin{equation}
    \begin{split}
      \dim \widetilde{W}_h(\mathcal{K}) &= \tfrac{1}{2}p(p+1)^2, \\
      \dim \widetilde{\boldsymbol{V}}_h(\mathcal{K}) &= p(p+1)^2.    
    \end{split}
  \end{equation}
  It follows that the number of unknowns in \cref{eq:def_projection}
  is
  \begin{equation}
    \dim W_h(\mathcal{K}) + \dim \boldsymbol{V}_h(\mathcal{K}) = \tfrac{3}{2}(p+1)^2(p+2).
  \end{equation}
  Since any space-time prism element $\mathcal{K}$ has only three
  faces in $\mathcal{Q}_{\mathcal{K}}$, the number of equations in
  \cref{eq:def_projection} is
  \begin{multline}
    \dim \widetilde{W}_h(\mathcal{K}) + \dim \widetilde{\boldsymbol{V}}_h(\mathcal{K}) + 3\dim M_h(\mathcal{F})
    \\
    = \tfrac{1}{2}p(p+1)^2 + p(p+1)^2 + 3(p+1)^2 = \tfrac{3}{2}(p+1)^2(p+2).
  \end{multline}
  Since the number of equations and unknowns in
  \cref{eq:def_projection} coincide, \cref{eq:def_projection} is a
  square system.

  Now, taking $\boldsymbol{q} = 0$ and $v = 0$ in
  \cref{eq:def_projection}, we see that
  \begin{subequations}
  	\begin{align}
	  	\label{eq:q_zero_rhs}
	  	\del[0]{\boldsymbol{\Pi_V q}, \boldsymbol{s}_h f_n}_{\mathcal{K}}
	  	&= 0
	  	&& \forall \boldsymbol{s}_h \in \widetilde{\boldsymbol{V}}_h(\mathcal{K}),
	  	\\
	  	\label{eq:v_zero_rhs}
	  	\del[0]{\Pi_W v, z_h f_n}_{\mathcal{K}}
	  	&= 0
	  	&& \forall z_h \in \widetilde{W}_h(\mathcal{K}),
	  	\\
	  	\label{eq:trace_zero_rhs}
	  	\langle \boldsymbol{\Pi_V q}\cdot\boldsymbol{n} - \tau \Pi_Wv, \sigma_h f_n \rangle_{\mathcal{F}}
	  	&= 0
	  	&& \forall \sigma_h \in M_h(\mathcal{F}),\ \mathcal{F} \subset \mathcal{Q}_{\mathcal{K}}.
  	\end{align}
  	\label{eq:projection_zero_rhs}
  \end{subequations}
  Let $w_h \in W_h^{\perp}(\mathcal{K})$. By \cref{eq:q_zero_rhs},
  since $\nabla w_h \in \widetilde{\boldsymbol{V}}_h(\mathcal{K})$, we
  have
  \begin{equation}
	  \del[0]{\boldsymbol{\Pi_V q}, \nabla w_h f_n}_{\mathcal{K}} = 0.
  \end{equation}
  Applying integration by parts in space,
  \begin{equation}
	 -\del[0]{\nabla \cdot \boldsymbol{\Pi_V q}, w_h f_n}_{\mathcal{K}} + \langle \boldsymbol{\Pi_V q} \cdot \boldsymbol{n}, w_h f_n \rangle_{\mathcal{Q}_{\mathcal{K}}} = 0.
  \end{equation}
  Since $\nabla \cdot \boldsymbol{\Pi_V q} \in \widetilde{W}_h(\mathcal{K})$, then $\del[1]{\nabla \cdot \boldsymbol{\Pi_V q}, w_h f_n}_{\mathcal{K}} = 0$, thus
  \begin{equation}
	  \label{eq:trace_q}
	  \langle \boldsymbol{\Pi_V q} \cdot \boldsymbol{n}, w_h f_n \rangle_{\mathcal{Q}_{\mathcal{K}}} = 0 \quad \forall w_h \in W_h^{\perp}(\mathcal{K}).
  \end{equation}
  By \cref{eq:trace_zero_rhs} and recalling that $\tau > 0$, we have
  \begin{equation}
	  \label{eq:trace_v}
	  \langle \Pi_W v, w_h f_n \rangle_{\mathcal{Q}_{\mathcal{K}}} = 0 \quad \forall w_h \in W_h^{\perp}(\mathcal{K}).
  \end{equation}
  Note that by \cref{eq:v_zero_rhs},
  $\Pi_W v \in W_h^{\perp}(\mathcal{K})$. Thus, taking $w_h = \Pi_W v$
  in \cref{eq:trace_v}, we see that $\Pi_W v = 0$ on
  $\mathcal{Q}_{\mathcal{K}}$. Then, using \cref{eq:w_norm} we
  conclude that $\Pi_W v = 0$ in $\mathcal{K}$. Taking
  $\sigma_h = \boldsymbol{\Pi_V q} \cdot \boldsymbol{n}$ in
  \cref{eq:trace_zero_rhs}, since $\Pi_W v = 0$, we see that
  $\boldsymbol{\Pi_V q}\cdot \boldsymbol{n} = 0$ on
  $\mathcal{Q}_{\mathcal{K}}$. Using \cref{eq:v_norm}, we conclude
  that $\boldsymbol{\Pi_V q} = 0$ in $\mathcal{K}$. The result
  follows.
\end{proof}

In addition to the projection $\Pi_h$, we define also $P^{f_n}_M$ as
the $f_n$-weighted $L^2$-projection onto $M_h$, so that
\begin{equation}
  \langle P^{f_n}_M v, \sigma_h f_n \rangle_{\mathcal{F}} = \langle v, \sigma_h f_n \rangle_{\mathcal{F}} \qquad
  \forall \sigma_h \in M_h(\mathcal{F}), \mathcal{F} \subset \mathcal{Q}_{\mathcal{K}}.
\end{equation}
Note that the domain of $P^{f_n}_M$ is $L^2(\Gamma_{\mathcal{Q}}^n)$.

We next show approximation properties of $\Pi_h$. For this, let
$P^{f_n}_W$, $P^{f_n}_{\widetilde{W}}$, and $\boldsymbol{P_V}^{f_n}$
denote, respectively, the $f_n$-weighted $L^2$-projections onto $W_h$,
$\widetilde{W}_h$, and $\boldsymbol{V}_h$.

\begin{theorem}[Approximation properties of the projection]
  \label{thm:proj_bounds}
  Assume that $\tau$ is uniformly bounded above and below by constants
  $C_{\tau}^{\max}$ and $C_{\tau}^{\min}$, respectively. The
  projection $\Pi_h$ satisfies the following bounds
  \begin{subequations}
    \begin{equation}
      \label{eq:estimate_v_PIWv}
      \begin{split}
	      \norm[0]{v - \Pi_W v}_{f_n,\mathcal{K}}
	      \leq &
	      \norm[0]{v - P^{f_n}_W v}_{f_n,\mathcal{K}}
	      + C C_{\tau} h_K^{1/2}\norm[0]{v - P^{f_n}_W v}_{f_n,\mathcal{Q}_{\mathcal{K}}}
	      \\
	      &+ \frac{C}{C_{\tau}^{\min}}h_K \norm[0]{\nabla\cdot \boldsymbol{q} - P_{\widetilde{W}}^{f_n}\nabla\cdot\boldsymbol{q}}_{f_n,\mathcal{K}},
      \end{split}
    \end{equation}
    \begin{equation}
      \label{eq:estimate_q_PIVq}
      \begin{split}
        \norm[0]{\boldsymbol{q} - \boldsymbol{\Pi_V q}}_{f_n,\mathcal{K}}
        \leq &
        \norm[0]{\boldsymbol{q} - \boldsymbol{P_V}^{f_n}\boldsymbol{q}}_{f_n,\mathcal{K}}
        + C h_K^{1/2} \norm[0]{\del[0]{\boldsymbol{q} - \boldsymbol{P_V}^{f_n}\boldsymbol{q}} \cdot \boldsymbol{n}}_{f_n, \mathcal{Q}_{\mathcal{K}}}
        \\
        &+ C C_{\tau}^{\max}h_K^{1/2}\norm[0]{v - P^{f_n}_W v}_{f_n, \mathcal{Q}_{\mathcal{K}}},        
      \end{split}
    \end{equation}
    \begin{equation}
      \label{eq:estimate_q_PIVq_K0}
      \begin{split}
        \norm[0]{\boldsymbol{q} - \boldsymbol{\Pi_V q}}_{f_n,\partial \mathcal{K}_0}
        \leq &
        \norm[0]{\boldsymbol{q} - \boldsymbol{P_V}^{f_n}\boldsymbol{q}}_{f_n,\partial \mathcal{K}_0}
        + C \Delta t^{-1/2} \norm[0]{\boldsymbol{q} - \boldsymbol{P_V}^{f_n}\boldsymbol{q}}_{f_n, \mathcal{K}}
        \\
        &+ C \Delta t^{-1/2} \norm[0]{\boldsymbol{q} - \boldsymbol{\Pi_V q}}_{f_n, \mathcal{K}},
      \end{split}
    \end{equation}
  \end{subequations}
  where $C_{\tau} = C_{\tau}^{\max}/C_{\tau}^{\min}$.
\end{theorem}
\begin{proof}
  Let
  $\boldsymbol{\delta_q} := \boldsymbol{\Pi_V q} -
  \boldsymbol{P_V}^{f_n}\boldsymbol{q}$ and
  $\delta_v := \Pi_W v - P^{f_n}_W v$. Note that
  $\boldsymbol{\delta_q}$ and $\delta_v$ satisfy the following
  equations
  \begin{subequations}
    \begin{align}
      \label{eq:deltaq}
      \del[0]{\boldsymbol{\delta_q}, \boldsymbol{s}_h f_n}_{\mathcal{K}}
      &= 0
      && \forall \boldsymbol{s}_h \in \widetilde{\boldsymbol{V}}_h(\mathcal{K}),
      \\
      \label{eq:deltav}
      \del[0]{\delta_v, z_h f_n}_{\mathcal{K}}
      &= 0
      && \forall z_h \in \widetilde{W}_h(\mathcal{K}),
      \\
      \label{eq:deltaq_detav_faces}
      \langle \del[0]{\boldsymbol{\delta_q}\cdot\boldsymbol{n} - \tau \delta_v}, \sigma_h f_n \rangle_{\mathcal{F}}
      &= \langle \del[0]{\boldsymbol{I_q}\cdot\boldsymbol{n} - \tau I_v}, \sigma_h f_n \rangle_{\mathcal{F}}
      && \forall \sigma_h \in M_h(\mathcal{F}),\ \mathcal{F} \subset \mathcal{Q}_{\mathcal{K}},         
    \end{align}
    \label{eq:deltaq_deltav_def}
  \end{subequations}
  where
  $\boldsymbol{I_q} = \boldsymbol{q} -
  \boldsymbol{P_V}^{f_n}\boldsymbol{q}$ and $I_v = v - P^{f_n}_W v$.

  We first prove \cref{eq:estimate_v_PIWv}. Notice that for any
  $w_h \in W_h^{\perp}(\mathcal{K})$,
  $w_h \vert_{\mathcal{F}} \in M_h(\mathcal{F})$, for any
  $\mathcal{F} \subset \mathcal{Q}_{\mathcal{K}}$. Therefore, by
  \cref{eq:deltaq_detav_faces},
  \begin{equation}
    \label{eq:deltaq_deltav_faces_Wperp}
    \langle \del[0]{\boldsymbol{\delta_q}\cdot\boldsymbol{n} - \tau \delta_v}, w_h f_n \rangle_{\mathcal{Q}_{\mathcal{K}}}
    = \langle \del[0]{\boldsymbol{I_q}\cdot\boldsymbol{n} - \tau I_v}, w_h f_n \rangle_{\mathcal{Q}_{\mathcal{K}}}
    \quad \forall w_h \in W_h^{\perp}(\mathcal{K}).  
  \end{equation}
  Using integration by parts in space, note that
  \begin{equation}
    \langle \boldsymbol{\delta_q} \cdot \boldsymbol{n}, w_h f_n \rangle_{\mathcal{Q}_{\mathcal{K}}}
    = \del[1]{\nabla \cdot \boldsymbol{\delta_q}, w_h f_n}_{\mathcal{K}}
    + \del[1]{\boldsymbol{\delta_q}, f_n \nabla w_h}_{\mathcal{K}}.
  \end{equation}
  Since
  $\nabla \cdot \boldsymbol{\delta_q} \in
  \widetilde{W}_h(\mathcal{K})$ and
  $w_h \in W_h^{\perp}(\mathcal{K})$, then
  $\del[1]{\nabla \cdot \boldsymbol{\delta_q}, w_h f_n}_{\mathcal{K}}
  = 0$. Also, since
  $\nabla w_h \in \widetilde{\boldsymbol{V}}_h(\mathcal{K})$, by
  \cref{eq:deltaq},
  $\del[1]{\boldsymbol{\delta_q}, f_n \nabla w_h}_{\mathcal{K}} =
  0$. Thus,
  \begin{equation}
    \label{eq:deltaq_dot_n}
    \langle \boldsymbol{\delta_q} \cdot \boldsymbol{n}, w_h f_n \rangle_{\mathcal{Q}_{\mathcal{K}}} = 
    \quad \forall w_h \in W_h^{\perp}(\mathcal{K}).    
  \end{equation}
  Similarly,
  \begin{equation}
    \langle \boldsymbol{I_q} \cdot \boldsymbol{n}, w_h f_n \rangle_{\mathcal{Q}_{\mathcal{K}}}
    = \del[1]{\nabla \cdot \boldsymbol{I_q}, w_h f_n}_{\mathcal{K}}
    + \del[1]{\boldsymbol{I_q}, f_n \nabla w_h}_{\mathcal{K}}
    \quad \forall w_h \in W_h^{\perp}(\mathcal{K}).    
  \end{equation}
  By definition of $\boldsymbol{I_q}$, the second term on the right
  hand side is zero. Furthermore, note that
  \begin{equation}
    \del[1]{\nabla \cdot \boldsymbol{I_q}, w_h f_n}_{\mathcal{K}}
    = \del[1]{\nabla \cdot \boldsymbol{q}, w_h f_n}_{\mathcal{K}}
    - \del[1]{\nabla \cdot \boldsymbol{P_V}^{f_n}\boldsymbol{q}, w_h f_n}_{\mathcal{K}},
  \end{equation}
  and
  $\del[1]{\nabla \cdot \boldsymbol{P_V}^{f_n}\boldsymbol{q}, w_h f_n}_{\mathcal{K}} =
  0$ since
  $\nabla \cdot \boldsymbol{P_V}^{f_n}\boldsymbol{q} \in \widetilde{W}_h(\mathcal{K})$
  and $w_h \in W_h^{\perp}(\mathcal{K})$. Therefore,
  \begin{equation}
    \del[1]{\nabla \cdot \boldsymbol{I_q}, w_h f_n}_{\mathcal{K}}
    = \del[1]{\nabla \cdot \boldsymbol{q}, w_h f_n}_{\mathcal{K}}.
  \end{equation}
  Also, since
  $P^{f_n}_{\widetilde{W}}(\nabla \cdot \boldsymbol{q}) \in
  \widetilde{W}_h(\mathcal{K})$ and
  $(\widetilde{w}_h, w_hf_n)_{\mathcal{K}} = 0$ for all
  $\widetilde{w}_h\in \widetilde{W}_h(\mathcal{K})$, we can write
  \begin{equation}
	  \del[1]{\nabla \cdot \boldsymbol{I_q}, w_h f_n}_{\mathcal{K}}
	  = \del[1]{\del[1]{ \mathit{Id} - P^{f_n}_{\widetilde{W}}}\nabla \cdot \boldsymbol{q}, w_h f_n}_{\mathcal{K}},
  \end{equation}
  where $\mathit{Id}$ denotes the identity operator. Thus,
  \begin{equation}
    \label{eq:Iq_dot_n}
    \langle \boldsymbol{I_q} \cdot \boldsymbol{n}, w_h f_n \rangle_{\mathcal{Q}_{\mathcal{K}}}
    = \del[1]{\del[1]{\mathit{Id} - P^{f_n}_{\widetilde{W}}}\nabla \cdot \boldsymbol{q}, w_h f_n}_{\mathcal{K}}
    \quad \forall w_h \in W_h^{\perp}(\mathcal{K}).     
  \end{equation}
  Using \cref{eq:deltaq_dot_n} and \cref{eq:Iq_dot_n} in
  \cref{eq:deltaq_deltav_faces_Wperp} and rearranging terms, we obtain
  \begin{equation}
    \label{eq:delta_v_simplified}
    \langle \tau \delta_v, w_h f_n \rangle_{\mathcal{Q}_{\mathcal{K}}}
    = \langle \tau I_v, w_h f_n \rangle_{\mathcal{Q}_{\mathcal{K}}}
    + \del[1]{\del[1]{P^{f_n}_{\widetilde{W}} - \mathit{Id} }\nabla \cdot \boldsymbol{q}, w_h f_n}_{\mathcal{K}}
    \quad \forall w_h \in W_h^{\perp}(\mathcal{K}).
  \end{equation}
  By \cref{eq:deltav}, $\delta_v \in W_h^{\perp}(\mathcal{K})$. Taking
  $w_h = \delta_v$ in \cref{eq:delta_v_simplified} we obtain
  \begin{equation}
    \label{eq:w_equals_delta_v}
    \langle \tau \delta_v, \delta_v f_n \rangle_{\mathcal{Q}_{\mathcal{K}}}
    = \langle \tau I_v, \delta_v f_n \rangle_{\mathcal{Q}_{\mathcal{K}}}
    + \del[1]{\del[1]{P^{f_n}_{\widetilde{W}} - \mathit{Id}}\nabla \cdot \boldsymbol{q}, \delta_v f_n}_{\mathcal{K}}.
  \end{equation}
  Apply the Cauchy--Schwarz inequality to the right hand side of
  \cref{eq:w_equals_delta_v}:
  \begin{equation}
    \label{eq:w_equals_delta_v2}
    \langle \tau \delta_v, \delta_v f_n \rangle_{\mathcal{Q}_{\mathcal{K}}}
    \leq \norm[0]{\tau I_v}_{f_n,\mathcal{Q}_{\mathcal{K}}} \norm[0]{\delta_v}_{f_n,\mathcal{Q}_{\mathcal{K}}}
    + \norm[0]{\del[0]{\mathit{Id} - P^{f_n}_{\widetilde{W}}} \nabla \cdot \boldsymbol{q}}_{f_n,\mathcal{K}}
    \norm[0]{\delta_v}_{f_n,\mathcal{K}}.    
  \end{equation}
  Using \cref{eq:wv_fn_estimates_a} on the right hand side,
  \begin{equation}
  \label{eq:w_equals_delta_v3}
  \langle \tau \delta_v, \delta_v f_n \rangle_{\mathcal{Q}_{\mathcal{K}}}
  \leq \norm[0]{\tau I_v}_{f_n,\mathcal{Q}_{\mathcal{K}}}\norm[0]{\delta_v}_{f_n,\mathcal{Q}_{\mathcal{K}}}
  + C h_K^{1/2}\norm[0]{\del[1]{\mathit{Id} - P^{f_n}_{\widetilde{W}}}\nabla \cdot \boldsymbol{q}}_{f_n,\mathcal{K}}
  \norm[0]{\delta_v}_{f_n,\mathcal{Q}_{\mathcal{K}}}.  
  \end{equation}
  Since $\tau$ is uniformly bounded above and below by constants
  $C_{\tau}^{\max}$ and $C_{\tau}^{\min}$, respectively,
  \begin{equation}
	  \label{eq:w_equals_delta_v4}
	  C_{\tau}^{\min} \norm[0]{\delta_v}_{f_n,\mathcal{Q}_{\mathcal{K}}}^2
	  \leq C_{\tau}^{\max} \norm[0]{I_v}_{f_n,\mathcal{Q}_{\mathcal{K}}}\norm[0]{\delta_v}_{f_n,\mathcal{Q}_{\mathcal{K}}}
	  + C h_K^{1/2}\norm[0]{\del[1]{\mathit{Id} - P^{f_n}_{\widetilde{W}}}\nabla \cdot \boldsymbol{q}}_{f_n,\mathcal{K}}
	  \norm[0]{\delta_v}_{f_n,\mathcal{Q}_{\mathcal{K}}}.    
  \end{equation}
  Canceling terms and using \cref{eq:wv_fn_estimates_a} on the left
  hand side, we obtain the following bound
  \begin{equation}
	  \label{eq:w_equals_delta_v5}
	  \dfrac{h_K^{-1/2}}{C} C_{\tau}^{\min} \norm[0]{\delta_v}_{f_n,\mathcal{K}}
	  \leq C_{\tau}^{\max} \norm[0]{I_v}_{f_n,\mathcal{Q}_{\mathcal{K}}}
	  + C h_K^{1/2}\norm[0]{\del[1]{\mathit{Id} - P^{f_n}_{\widetilde{W}}}\nabla \cdot \boldsymbol{q}}_{f_n,\mathcal{K}}.
  \end{equation}
  Rearranging,
  \begin{equation}
	  \label{eq:final_delta_v_bound}
	  \norm[0]{\delta_v}_{f_n,\mathcal{K}}
	  \leq C C_{\tau}h_K^{1/2}\norm[0]{I_v}_{f_n,\mathcal{Q}_{\mathcal{K}}}
	  + \dfrac{C}{C_{\tau}^{\min}} h_K\norm[0]{\del[1]{\mathit{Id} - P^{f_n}_{\widetilde{W}}}\nabla \cdot \boldsymbol{q}}_{f_n,\mathcal{K}}.    
  \end{equation}
  The estimate \cref{eq:estimate_v_PIWv} follows by applying the
  triangle inequality.

  We next prove \cref{eq:estimate_q_PIVq}. We first find an estimate
  for $\boldsymbol{\delta_q}$. Note that by \cref{eq:deltaq},
  $\boldsymbol{\delta_q}$ belongs to
  $\boldsymbol{V}_h^{\perp}(\mathcal{K})$. Therefore, by
  \cref{eq:wv_fn_estimates_b}, we have
  \begin{equation}
  	\label{eq:deltaq_initial_estimate}
    \norm[0]{\boldsymbol{\delta_q}}_{f_n,\mathcal{K}}
    \leq
    C h_K^{1/2} \norm[0]{\boldsymbol{\delta_q} \cdot \boldsymbol{n}}_{f_n, \mathcal{Q}_{\mathcal{K}}}.
  \end{equation}
  Taking $\sigma_h = \boldsymbol{\delta_q} \cdot \boldsymbol{n}$ in
  \cref{eq:deltaq_detav_faces}, we obtain
  \begin{equation}
    \norm[0]{\boldsymbol{\delta_q} \cdot \boldsymbol{n}}_{f_n, \mathcal{Q}_{\mathcal{K}}}^2 = \langle \boldsymbol{I_q}\cdot\boldsymbol{n}, \boldsymbol{\delta_q} \cdot \boldsymbol{n} f_n \rangle_{\mathcal{Q}_{\mathcal{K}}} - \tau \langle I_v, \boldsymbol{\delta_q} \cdot \boldsymbol{n} f_n \rangle_{\mathcal{Q}_{\mathcal{K}}} + \tau \langle \delta_v, \boldsymbol{\delta_q} \cdot \boldsymbol{n} f_n \rangle_{\mathcal{Q}_{\mathcal{K}}}.
  \end{equation}
  Since $\delta_v \in W_h^{\perp}(\mathcal{K})$, by \cref{eq:deltaq_dot_n}, $\langle \delta_v, \boldsymbol{\delta_q} \cdot \boldsymbol{n} f_n \rangle_{\mathcal{Q}_{\mathcal{K}}} = 0$. Thus,
  \begin{equation}
    \norm[0]{\boldsymbol{\delta_q} \cdot \boldsymbol{n}}_{f_n, \mathcal{Q}_{\mathcal{K}}}^2 = \langle \boldsymbol{I_q}\cdot\boldsymbol{n}, \boldsymbol{\delta_q} \cdot \boldsymbol{n} f_n \rangle_{\mathcal{Q}_{\mathcal{K}}} - \tau \langle I_v, \boldsymbol{\delta_q} \cdot \boldsymbol{n} f_n \rangle_{\mathcal{Q}_{\mathcal{K}}}.
  \end{equation}
  Applying the Cauchy--Schwarz inequality and substituting into \cref{eq:deltaq_initial_estimate}, we obtain
  \begin{equation}
    \norm[0]{\boldsymbol{\delta_q}}_{f_n,\mathcal{K}}
    \leq
    C h_K^{1/2} \del[1]{\norm[0]{\boldsymbol{I_q} \cdot \boldsymbol{n}}_{f_n, \mathcal{Q}_{\mathcal{K}}}
      + C_{\tau}^{\max}\norm[0]{I_v}_{f_n, \mathcal{Q}_{\mathcal{K}}}}.
  \end{equation}
  The estimate \cref{eq:estimate_q_PIVq} follows by applying the
  triangle inequality.
  
  Finally, we show \cref{eq:estimate_q_PIVq_K0}. Note that, by the
  triangle inequality,
  \begin{equation}
    \norm[0]{\boldsymbol{q} - \boldsymbol{\Pi_V q}}_{f_n,\partial \mathcal{K}_0}
    \leq \norm[0]{\boldsymbol{q} - \boldsymbol{P_V}^{f_n} \boldsymbol{q}}_{f_n,\partial \mathcal{K}_0} + \norm[0]{\boldsymbol{P_V}^{f_n} \boldsymbol{q} - \boldsymbol{\Pi_V q}}_{f_n,\partial \mathcal{K}_0}.
  \end{equation}
  Next, we apply the inverse trace inequality in
  \cref{eq:inv_trace_ineq_K0} to the second term on the right hand side
  to obtain:
  \begin{equation}
    \norm[0]{\boldsymbol{q} - \boldsymbol{\Pi_V q}}_{f_n,\partial \mathcal{K}_0}
    \leq \norm[0]{\boldsymbol{q} - \boldsymbol{P_V}^{f_n} \boldsymbol{q}}_{f_n,\partial \mathcal{K}_0} + C\Delta t^{-1/2} \norm[0]{\boldsymbol{P_V}^{f_n} \boldsymbol{q} - \boldsymbol{\Pi_V q}}_{f_n,\mathcal{K}}.
  \end{equation}
  The result follows after adding and subtracting $\boldsymbol{q}$ to
  the second term on the right hand side and applying the triangle
  inequality.
\end{proof}


We next prove the equivalence between the standard and the weighted
$L^2$-projections.

\begin{lemma}
  \label{lem:equiv_proj}
  Let $\mathcal{K} \in \mathcal{T}^n$ and
  $\mathcal{F} \in \mathcal{F}_{\mathcal{Q}}^n$. Let
  $P_W, \boldsymbol{P_V}$ and $P_M$ denote the
  $L^2$-orthogonal projections onto $W_h$,
  $\boldsymbol{V}_h$ and $M_h$, respectively. If
  $f_n$ is uniformly bounded, then the following relations are
  satisfied
  \begin{subequations}
    \begin{equation}
      \label{eq:equivW}
      \norm[0]{v - P^{f_n}_W v}_{f_n, \mathcal{K}}
      \leq C \norm[0]{v - P_W v}_{\mathcal{K}},
    \end{equation}
    \begin{equation}
      \label{eq:equivV}
      \norm[0]{\boldsymbol{q} - \boldsymbol{P_V}^{f_n}\boldsymbol{q}}_{f_n, \mathcal{K}}
      \leq C \norm[0]{\boldsymbol{q} - \boldsymbol{P_V q}}_{\mathcal{K}},
    \end{equation}
    \begin{equation}
      \label{eq:equivM}
      \norm[0]{v - P^{f_n}_M v}_{f_n, \mathcal{F}}
      \leq C \norm[0]{v - P_M v}_{\mathcal{F}}.
    \end{equation}
  \end{subequations}
\end{lemma}

\begin{proof}
  We will only show \cref{eq:equivW}. The proofs for \cref{eq:equivV}
  and \cref{eq:equivM} are analogous. Note that by definition of
  $P_W^{f_n}$, for all $w_h \in W_h$
  \begin{equation}
    \del[1]{P^{f_n}_W v - P_W v, f_n w_h}_{\mathcal{K}} = \del[1]{v - P_W v, f_n w_h}_{\mathcal{K}}.
  \end{equation}
  Let $w_h = P^{f_n}_W v - P_W v$, then,
  \begin{equation}
    \begin{split}
      \norm[1]{P^{f_n}_W v - P_W v}^2_{f_n, \mathcal{K}} & = \del[1]{v - P_W v, f_n \del[0]{P^{f_n}_W v - P_W v}}_{\mathcal{K}}
      \\
      & \leq \norm[1]{v - P_W v}_{f_n, \mathcal{K}} \norm[1]{P^{f_n}_W v - P_W v}_{f_n, \mathcal{K}}.
    \end{split}
  \end{equation}
  Thus,
  \begin{equation}
    \norm[1]{P^{f_n}_W v - P_W v}_{f_n, \mathcal{K}} \leq \norm[1]{v - P_W v}_{f_n, \mathcal{K}} \leq C \norm[1]{v - P_W v}_{\mathcal{K}},
  \end{equation}
  since $f_n$ is uniformly bounded. The result follows by the triangle
  inequality.
\end{proof}

In order to obtain an equivalence between the standard and the weighted $L^2$-projections on the boundary $\mathcal{Q}_{\mathcal{K}}$ of a space-time element $\mathcal{K}$, we require the following continuous trace inequality:
\begin{lemma}[Continuous trace inequality]
	\label{lem:cont_trace_ineq}
	Let $\mathcal{K} = K \times I_n$. For $\phi \in H^{(0,1)}(\mathcal{K})$, the following holds:
	\begin{equation}
		\norm[1]{\phi}^2_{\mathcal{Q}_{\mathcal{K}}} \leq C\del{\norm{\nabla \phi}_{\mathcal{K}} + h_K^{-1}\norm[0]{\phi}_{\mathcal{K}}}\norm[0]{\phi}_{\mathcal{K}}.
	\end{equation}
\end{lemma}

\begin{proof}
  The proof is analogous to the proof of \cite[Lemma
  1.49]{Pietro:book}.
\end{proof}

We next find $L^2$ projection esimates of the different projection
operators.

\begin{theorem}[$L^2$ orthogonal projection estimates]
  \label{thm:L2proj_bounds}
  Let $\mathcal{K} = K \times I_n$ and $\mathcal{F}$ be a face on the
  free-surface boundary, i.e.,
  $\mathcal{F} \in \mathcal{F}_{\mathcal{S}}^n$. Let
  $\partial \mathcal{F}_0$ denote the two edges of the face
  $\mathcal{F}$ that are on the time levels. Assume that the spatial
  shape-regularity condition \cref{eq:shape_reg} holds and that the
  triangulation $\mathcal{T}^n $does not have any hanging
  nodes. Suppose that $(\boldsymbol{q}, v)$ are such that
  $\boldsymbol{q} |_{\mathcal{K}} \in \sbr[1]{H^{(s_t,
      s_s)}(\mathcal{K})}^2$,
  $v|_{\mathcal{K}} \in H^{(s_t,s_s)}(\mathcal{K})$, where
  $1/2 < s_t \leq p+1$ and $1 \leq s_s \leq p+1$. Then we have the following
  estimates:
  \begin{subequations}
    \begin{equation}
      \label{eq:L2_v_K}
      \norm[0]{v - P_W v}_{\mathcal{K}}
      \leq C\del[1]{h_K^{s_s} + \Delta t^{s_t}}  \norm{v}_{H^{(s_t, s_s)}(\mathcal{K})},
    \end{equation}
    \begin{equation}
      \label{eq:L2_v_QK}
      \norm[0]{v - P_W v}_{\mathcal{Q}_{\mathcal{K}}}
      \leq C \del[0]{h_K^{s_s-1/2} + h_K^{-1/2}\Delta t^{s_t}}  \norm{v}_{H^{(s_t, s_s)}(\mathcal{K})},
    \end{equation}
    \begin{equation}
      \label{eq:L2_q_K}
      \norm[0]{\boldsymbol{q} - \boldsymbol{P_V q}}_{\mathcal{K}}
      \leq C \del[0]{h_K^{s_s} + \Delta t^{s_t}} \norm{\boldsymbol{q}}_{H^{(s_t, s_s)}(\mathcal{K})},
    \end{equation}
    \begin{equation}
      \label{eq:L2_q_QK}
      \norm[0]{\del[0]{\boldsymbol{q} - \boldsymbol{P_V q}} \cdot \boldsymbol{n}}_{\mathcal{Q}_{\mathcal{K}}}
      \leq C \del[0]{h_K^{s_s-1/2} + h_K^{-1/2}\Delta t^{s_t}} \norm{\boldsymbol{q}}_{H^{(s_t, s_s)}(\mathcal{K})},
    \end{equation}
    \begin{equation}
      \label{eq:L2_w_K0}
      \norm[0]{v - P_W v}_{\partial \mathcal{K}_0}
      \leq C \del[0]{\Delta t^{-1/2} h_K^{s_s} + \Delta t^{s_t-1/2}} \norm{v}_{H^{(s_t,s_s)}(\mathcal{K})},
    \end{equation}
    \begin{equation}
      \label{eq:L2_q_K0}
      \norm[0]{\boldsymbol{q} - \boldsymbol{P_V q}}_{\partial \mathcal{K}_0}
      \leq C \del[0]{\Delta t^{-1/2} h_K^{s_s} + \Delta t^{s_t-1/2}} \norm{\boldsymbol{q}}_{H^{(s_t,s_s)}(\mathcal{K})},
    \end{equation}
    \begin{equation}
      \label{eq:L2_M}
      \norm[0]{v - P_M v}_{\mathcal{F}}
      \leq C \del[0]{h_K^{s_s} + \Delta t^{s_t}}  \norm{v}_{H^{(s_t, s_s)}(\mathcal{F})},
    \end{equation}
    \begin{equation}
      \label{eq:L2_M_F0}
      \norm[0]{v - P_M v}_{\partial \mathcal{F}_0}
      \leq C \del[0]{\Delta t^{-1/2}h_K^{s_s} + \Delta t^{s_t-1/2}} \norm{v}_{H^{(s_t, s_s)}(\mathcal{F})}.
    \end{equation}
  \end{subequations}
\end{theorem}

\begin{proof}
  First note that \cref{eq:L2_q_K} and \cref{eq:L2_q_K0} result from
  applying \cref{eq:L2_v_K} and \cref{eq:L2_w_K0}, respectively, on
  each component of $\boldsymbol{q}$, so the proof for those estimates
  is not shown. Since the face $\mathcal{F}$ is a quadrilateral, the
  proof for \cref{eq:L2_M} and \cref{eq:L2_M_F0} can be found in
  \cite[Lemma B.14]{Sudirham:2005}.
	
  Let $\pi^t$ and $\pi^s$ denote the orthogonal $L^2$-projections onto $L^2(K) \otimes P_p(I_n)$ and onto $P_p(K) \otimes L^2(I_n)$, respectively. We can define $P_W$ as
  $P_W := \pi^t \circ \pi^s$.

  We first show \cref{eq:L2_v_K}. Notice that
  \begin{multline}
    \norm{v - P_W v}^2_{\mathcal{K}} = \norm{v - \pi^t \circ \pi^s v}^2_{\mathcal{K}} 
    = \norm{v - \pi^t v + \pi^t(v - \pi^s v)}^2_{\mathcal{K}}
    \\
    \leq C\del[1]{\,\norm{v - \pi^t v}^2_{\mathcal{K}} + \norm{\pi^t(v - \pi^s v)}^2_{\mathcal{K}}}.    
  \end{multline}
  Since $\pi^t$ is bounded and $\norm{\pi^t} = 1$, then
  \begin{equation}
    \label{eq:norm_split}
    \norm{v - P_W v}^2_{\mathcal{K}} \leq C\del[1]{\, \norm{v - \pi^t v}^2_{\mathcal{K}} + \norm{v - \pi^s v}^2_{\mathcal{K}}}.
  \end{equation}
  Let us treat each term separately. For the second term on the right
  hand side of \cref{eq:norm_split}, by \cite[Lemma
  1.58]{Pietro:book}, since we assume \cref{eq:shape_reg} and no
  hanging nodes,
  \begin{multline}
    \label{eq:spat_K}
    \norm{v - \pi^s v}^2_{\mathcal{K}} = \int_{t_n}^{t_{n+1}} \int_{K}(v - \pi^s v)^2 \dif \boldsymbol{x} \dif x_0
    \leq C h_K^{2s_s}\int_{t_n}^{t_{n+1}} \abs{v}^2_{H^{s_s}(K)} \dif x_0
    \\
    \le Ch_K^{2s_s} \norm{v}^2_{H^{(0,s_s)}(\mathcal{K})},    
  \end{multline}  
  where $0 \leq s_s \leq p+1$. Similarly, for the temporal projection
  we have
  \begin{equation}
    \label{eq:temp_K}
    \norm{v - \pi^t v}^2_{\mathcal{K}} = \int_{K} \int_{t_n}^{t_{n+1}} (v - \pi^t v)^2 \dif x_0 \dif \boldsymbol{x}
    \leq C \Delta t^{2s_t} \norm{v}^2_{H^{(s_t,0)}(\mathcal{K})},
  \end{equation}
  where $0 \leq s_t \leq p+1$. \Cref{eq:L2_v_K} follows by combining
  \cref{eq:spat_K} and \cref{eq:temp_K}.
  
  Next, we show \cref{eq:L2_v_QK}. Similarly as above, we have
  \begin{equation}
    \label{eq:norm_splitQK}
    \norm{v - P_W v}^2_{\mathcal{Q}_{\mathcal{K}}} \leq
    C\del[1]{\, \norm{v - \pi^t v}^2_{\mathcal{Q}_{\mathcal{K}}} + \norm{v - \pi^s v}^2_{\mathcal{Q}_{\mathcal{K}}}}.	
  \end{equation}
  For the spatial projection, using \cite[Lemma 1.59]{Pietro:book}, we
  have
  \begin{multline}
    \label{eq:spat_QK}
    \norm{v - \pi^s v}^2_{\mathcal{Q}_{\mathcal{K}}} = \int_{t_n}^{t_{n+1}} \int_{\partial K}(v - \pi^s v)^2 \dif \boldsymbol{x} \dif x_0
    \\
    \leq C h_K^{2s_s-1}\int_{t_n}^{t_{n+1}} \abs{v}^2_{H^{s_s}(K)} \dif x_0
    \le Ch_K^{2s_s-1} \norm{v}^2_{H^{(0,s_s)}(\mathcal{K})}.    
  \end{multline}
  For the temporal projection, similarly as above,
  \begin{equation}
    \label{eq:temp_QK}
    \norm{v - \pi^t v}^2_{\mathcal{Q}_{\mathcal{K}}} = \int_{\partial K} \int_{t_n}^{t_{n+1}} (v - \pi^t v)^2 \dif x_0 \dif \boldsymbol{x}
    \leq C \Delta t^{2s_t} \norm{v}^2_{H^{(s_t,0)}(\mathcal{Q}_\mathcal{K})}.
  \end{equation}
  Combining \cref{eq:spat_QK} and \cref{eq:temp_QK}, we obtain
  \begin{equation}
  	\norm{v - P_W v}^2_{\mathcal{Q}_{\mathcal{K}}} \leq Ch_K^{2s_s-1} \norm{v}^2_{H^{(0,s_s)}(\mathcal{K})} + C \Delta t^{2s_t} \norm{v}^2_{H^{(s_t,0)}(\mathcal{Q}_\mathcal{K})}.
  \end{equation}
  Note that, by \Cref{lem:cont_trace_ineq}, since $v \in H^{(s_t,s_s)}(\mathcal{K})$ with $s_s \geq 1$, we have
  \begin{equation}
  	\norm{\partial_{x_0}^{\alpha_t} v}_{\mathcal{Q}_{\mathcal{K}}}^2 \le C \norm{\nabla \partial_{x_0}^{\alpha_t} v}_{\mathcal{K}} \norm{\partial_{x_0}^{\alpha_t} v}_{\mathcal{K}} + C h_K^{-1}\norm{ \partial_{x_0}^{\alpha_t} v}_{\mathcal{K}}^2,
  \end{equation}
  for all $0 \leq \alpha_t \leq s_t$. Thus,
  \begin{equation}
  	\norm{v}^2_{H^{(s_t,0)}(\mathcal{Q}_\mathcal{K})} \leq C \norm{v}^2_{H^{(s_t,s_s)}(\mathcal{K})} + C h_K^{-1}\norm{v}^2_{H^{(s_t,s_s)}(\mathcal{K})},
  \end{equation}
  and the result follows.
  
  To show \cref{eq:L2_q_QK}, we note the following
  \begin{equation}
    \norm[0]{\del[0]{\boldsymbol{q} - \boldsymbol{P_V q}} \cdot \boldsymbol{n}}_{\mathcal{Q}_{\mathcal{K}}} \leq \norm[0]{\boldsymbol{q} - \boldsymbol{P_V q}}_{\mathcal{Q}_{\mathcal{K}}}.
  \end{equation}
  The result follows by applying \cref{eq:L2_v_QK} component wise.
  
  Finally, we show \cref{eq:L2_w_K0}. For the spatial component of the
  projection, notice that
  \begin{equation}
    \begin{aligned}
      \norm{v - \pi^s v}^2_{\partial \mathcal{K}_0} & = \norm{v - \pi^s v}^2_{K_j^{n+1}} + \norm{v - \pi^s v}^2_{K_j^{n}}
      \\
      & \leq Ch_K^{2s_s}\del[1]{\abs{v}^2_{H^{s_s}(K_j^{n+1})} + \abs{v}^2_{H^{s_s}(K_j^{n})}}
      \\
      & \leq Ch_K^{2s_s} \norm[0]{v}^2_{H^{s_s}(\partial \mathcal{K}_0)}.
    \end{aligned}
  \end{equation}
  By the Sobolev embedding theorem, the definition of fractional
  Sobolev norms \cite[Chapter 14]{BrennerScott:book}, and a standard
  scaling argument, we have for $s_t > 1/2$,
  \begin{equation}
  \norm[0]{v}_{H^{s_s}(\partial \mathcal{K}_0)} \le C\Delta t^{-1/2}\norm{v}_{H^{(s_t,s_s)}(\mathcal{K})}.
  \end{equation}
  Thus,
  \begin{equation}
  \norm{v - \pi^s v}^2_{\partial \mathcal{K}_0} \leq C\Delta t^{-1}h_K^{2s_s} \norm{v}^2_{H^{(s_t,s_s)}(\mathcal{K})}.
  \end{equation}
  
  
  For the temporal projection, we have
  \begin{equation}
    \begin{aligned}
      \norm{v - \pi^t v}^2_{\partial \mathcal{K}_0} & = \norm{v - \pi^t v}^2_{K_j^{n+1}} + \norm{v - \pi^t v}^2_{K_j^{n}}
      \\
      & = \int_{K}\del[1]{(v - \pi^t v)(t_{n+1},\boldsymbol{x})}^2 \dif \boldsymbol{x} + \int_{K}\del[1]{(v - \pi^t v)(t_{n},\boldsymbol{x})}^2 \dif \boldsymbol{x} && \text{(by def.)}
      \\
      & \leq h_K^2 \int_{\widehat{K}}\del[1]{(v - \pi^t v)(t_{n+1},\widehat{\boldsymbol{x}})}^2 \dif \widehat{\boldsymbol{x}} + h_K^2 \int_{\widehat{K}}\del[1]{(v - \pi^t v)(t_{n},\widehat{\boldsymbol{x}})}^2 \dif \widehat{\boldsymbol{x}} && \text{(by \cref{eq:scaling_spatialK})}
      \\
      & = h_K^2 \norm{\check{v} - \pi^t \check{v}}^2_{\partial \check{\mathcal{K}}_0} && \text{(by def.)}
      \\
      & = h_K^2 \norm{\widehat{v} - \widehat{\pi}^t \widehat{v}}^2_{\partial \widehat{\mathcal{K}}_0} && \text{(by \cref{eq:scaling_K0})}
      \\
      & = h_K^2 \int_{\widehat{K}}\del[1]{(\widehat{v} - \widehat{\pi}^t \widehat{v})(1,\widehat{\boldsymbol{x}})}^2 \dif \widehat{\boldsymbol{x}} + h_K^2 \int_{\widehat{K}}\del[1]{(\widehat{v} - \widehat{\pi}^t \widehat{v})(-1,\widehat{\boldsymbol{x}})}^2 \dif \widehat{\boldsymbol{x}} && \text{(by def.)}
      \\
      & \leq C h_K^2 \norm[1]{\partial_{\widehat{x}_0}^{s_t} \widehat{v}}^2_{\widehat{\mathcal{K}}} && \text{(by \cite{Houston:2002})}
      \\
      & = C h_K^2 \del{\frac{\Delta t}{2}}^{2s_t - 1} \norm[1]{\partial_{\check{x}_0}^{s_t} \check{v}}^2_{\check{\mathcal{K}}} && \text{(by \cref{eq:scaling_element})}
      \\
      & \leq C h_K^2 \del{\frac{\Delta t}{2}}^{2s_t - 1} h_K^{-2} \norm[1]{\partial_{x_0}^{s_t} v}^2_{\mathcal{K}} && \text{(by \cref{eq:scaling_spatialK})}
      \\
      & \leq C \Delta t^{2s_t - 1}\norm{v}^2_{H^{(s_t,0)}(\mathcal{K})}.
    \end{aligned}
  \end{equation}
  Note that here we have used that
  $\abs{\widehat{v} - \widehat{\pi}^t \widehat{v}(\pm 1)} \leq C
  \norm[0]{\partial_{\widehat{x}_0} ^{s_t}\widehat{v}}_{\widehat{I}}$
  which was shown in \cite[Lemma 3.5]{Houston:2002}. This concludes
  the proof.
\end{proof}

\begin{corollary}
  \label{cor:weightedL2proj_bounds}
  Suppose that $f_n$ is uniformly bounded. Then, under the same
  assumptions as in \Cref{thm:L2proj_bounds}, the following estimates
  are satisfied:
  \begin{subequations}
    \begin{equation}
      \label{eq:fnL2_v_K}
      \norm[0]{v - P^{f_n}_W v}_{f_n,\mathcal{K}}
      \leq C \del[0]{h_K^{s_s} + \Delta t^{s_t}} \norm{v}_{H^{(s_t, s_s)}(\mathcal{K})},
    \end{equation}
    \begin{equation}
      \label{eq:fnL2_v_QK}
      \norm[0]{v - P^{f_n}_W v}_{f_n,\mathcal{Q}_{\mathcal{K}}}
      \leq C \del[0]{h_K^{s_s-1/2} + h_K^{-1/2}\Delta t^{s_t}}  \norm{v}_{H^{(s_t, s_s)}(\mathcal{K})},
    \end{equation}
    \begin{equation}
      \label{eq:fnL2_q_K}
      \norm[0]{\boldsymbol{q} - \boldsymbol{P_V}^{f_n}\boldsymbol{q}}_{f_n,\mathcal{K}}
      \leq C \del[0]{h_K^{s_s} + \Delta t^{s_t}} \norm{\boldsymbol{q}}_{H^{(s_t, s_s)}(\mathcal{K})},
    \end{equation}
    \begin{equation}
      \label{eq:fnL2_q_QK}
      \norm[0]{\del[0]{\boldsymbol{q} - \boldsymbol{P_V}^{f_n}\boldsymbol{q}} \cdot \boldsymbol{n}}_{f_n,\mathcal{Q}_{\mathcal{K}}}
      \leq C \del[0]{h_K^{s_s-1/2} + h_K^{-1/2}\Delta t^{s_t}} \norm{\boldsymbol{q}}_{H^{(s_t, s_s)}(\mathcal{K})},
    \end{equation}
    \begin{equation}
      \label{eq:fnL2_w_K0}
      \norm[0]{v - P^{f_n}_W v}_{f_n, \partial \mathcal{K}_0}
      \leq C \del[0]{\Delta t^{-1/2}h_K^{s_s} + \Delta t^{s_t-1/2}} \norm{v}_{H^{(s_t,s_s)}(\mathcal{K})},
    \end{equation}
    \begin{equation}
      \label{eq:fnL2_q_K0}
      \norm[0]{\boldsymbol{q} - \boldsymbol{P_V}^{f_n} \boldsymbol{q}}_{f_n, \partial \mathcal{K}_0}
      \leq C \del[0]{\Delta t^{-1/2}h_K^{s_s} + \Delta t^{s_t-1/2}} \norm{\boldsymbol{q}}_{H^{(s_t,s_s)}(\mathcal{K})},
    \end{equation}
    \begin{equation}
      \label{eq:fnL2_M}
      \norm[0]{v - P_M^{f_n} v}_{f_n, \mathcal{F}}
      \leq C \del[0]{h_K^{s_s} + \Delta t^{s_t}} \norm{v}_{H^{(s_t, s_s)}(\mathcal{F})},
    \end{equation}
    \begin{equation}
      \label{eq:fnL2_M_F0}
      \norm[0]{v - P_M^{f_n} v}_{f_n, \partial \mathcal{F}_0}
      \leq C \del[0]{\Delta t^{-1/2}h_K^{s_s} + \Delta t^{s_t-1/2}} \norm{v}_{H^{(s_t, s_s)}(\mathcal{F})}.
    \end{equation}
  \end{subequations}
\end{corollary}
\begin{proof}
  \Cref{eq:fnL2_v_K}, \cref{eq:fnL2_q_K} and \cref{eq:fnL2_M} follow directly from \Cref{lem:equiv_proj} and \Cref{thm:L2proj_bounds}. Moreover, \cref{eq:fnL2_q_QK} and \cref{eq:fnL2_q_K0} follow from \cref{eq:fnL2_v_QK} and \cref{eq:fnL2_w_K0}, respectively. Therefore, we only show \cref{eq:fnL2_v_QK}, \cref{eq:fnL2_w_K0} and \cref{eq:fnL2_M_F0}.
  
  In order to prove \cref{eq:fnL2_v_QK}, notice that since $f_n$ is uniformly bounded and using the triangle inequality, we obtain
  \begin{equation}
  	\label{eq:vQK_P_minus_Pfn}
  	\norm[0]{v - P^{f_n}_W v}_{f_n,\mathcal{Q}_{\mathcal{K}}} \leq C \norm[0]{v - P_W v}_{\mathcal{Q}_{\mathcal{K}}} + C \norm[0]{P_W v - P^{f_n}_W v}_{\mathcal{Q}_{\mathcal{K}}}.
  \end{equation}
  By \cref{eq:inv_trace_ineq_QK}, we obtain
  \begin{equation}
  	\norm[0]{v - P^{f_n}_W v}_{f_n,\mathcal{Q}_{\mathcal{K}}} \leq C \norm[0]{v - P_W v}_{\mathcal{Q}_{\mathcal{K}}} + C h_K^{-1/2} \norm[0]{P_W v - P^{f_n}_W v}_{\mathcal{K}}.
  \end{equation}
  Using equivalence of norms in finite dimensional spaces, and by the
  triangle inequality, we have
  \begin{equation}
  	\label{eq:vQK_P_minus_Pfn2}
	  \norm[0]{v - P^{f_n}_W v}_{f_n,\mathcal{Q}_{\mathcal{K}}} \leq C \norm[0]{v - P_W v}_{\mathcal{Q}_{\mathcal{K}}} + C h_K^{-1/2} \norm[0]{P_W v - v}_{f_n,\mathcal{K}} + C h_K^{-1/2} \norm[0]{v - P^{f_n}_W v}_{f_n,\mathcal{K}}.
  \end{equation}
  \Cref{eq:fnL2_v_QK} follows by recalling that $f_n$ is uniformly
  bounded and using the estimates \cref{eq:L2_v_QK}, \cref{eq:L2_v_K}
  and \cref{eq:fnL2_v_K}. 
  
  In order to show \cref{eq:fnL2_w_K0}, since $f_n$ is uniformly
  bounded, we obtain by the triangle inequality
  \begin{equation}
	  \label{eq:P_minus_Pfn}
	  \norm[0]{v - P_W^{f_n} v}_{f_n,\partial \mathcal{K}_0} \leq C \norm[0]{v - P_W v}_{\partial \mathcal{K}_0} + C \norm[0]{P_W v - P_W^{f_n} v}_{\partial \mathcal{K}_0}.
  \end{equation}
  Using \cref{eq:inv_trace_ineq_K0} on the second term of the right
  hand side of \cref{eq:P_minus_Pfn},
  \begin{equation}
	  \label{eq:P_minus_Pfn2}
	 \norm[0]{v - P_W^{f_n} v}_{f_n,\partial \mathcal{K}_0} \leq C \norm[0]{v - P_W v}_{\partial \mathcal{K}_0} + C\Delta t^{-1/2}\norm[0]{P_W v - P_W^{f_n} v}_{\mathcal{K}}.
  \end{equation}
  By the triangle inequality,
  \begin{equation}
  	\label{eq:P_minus_Pfn3}
	  \norm[0]{v - P_W^{f_n} v}_{f_n,\partial \mathcal{K}_0} \leq C \norm[0]{v - P_W v}_{\partial \mathcal{K}_0} + C\Delta t^{-1/2}\norm[0]{P_W v - v}_{f_n,\mathcal{K}} + C\Delta t^{-1/2}\norm[0]{v - P_W^{f_n} v}_{f_n,\mathcal{K}}.
  \end{equation}
  Then, \cref{eq:fnL2_w_K0} is obtained by \cref{eq:L2_w_K0},
  \cref{eq:fnL2_v_K} and \cref{eq:L2_v_K}.
  
  Finally, we show \cref{eq:fnL2_M_F0}. Using that $f_n$ is uniformly
  bounded and by the triangle inequality,
  \begin{equation}
  	\norm[0]{v - P_M^{f_n} v}_{f_n,\partial \mathcal{F}_0} \leq C \norm[0]{v - P_M v}_{\partial \mathcal{F}_0} + C \norm[0]{P_M v - P_M^{f_n} v}_{\partial \mathcal{F}_0}.
  \end{equation}
  By \cref{eq:inv_trace_ineq_F0},
  \begin{equation}
  	\norm[0]{v - P_M^{f_n} v}_{f_n,\partial \mathcal{F}_0} \leq C \norm[0]{v - P_M v}_{\partial \mathcal{F}_0} + C \Delta t^{-1/2} \norm[0]{P_M v - P_M^{f_n} v}_{\mathcal{F}}.
  \end{equation}
  By the triangle inequality, we then obtain:
  \begin{equation}
  	\norm[0]{v - P_M^{f_n} v}_{f_n,\partial \mathcal{F}_0} \leq C \norm[0]{v - P_M v}_{\partial \mathcal{F}_0} + C \Delta t^{-1/2} \norm[0]{P_M v - v}_{f_n,\mathcal{F}} + C \Delta t^{-1/2} \norm[0]{v - P_M^{f_n} v}_{f_n,\mathcal{F}}.
  \end{equation}
  \Cref{eq:fnL2_M_F0} follows by the uniform boundedness of $f_n$ and
  the estimates \cref{eq:L2_M_F0}, \cref{eq:L2_M}, and
  \cref{eq:fnL2_M}.
\end{proof}

To conclude this subsection, we show the error estimates of our
projection $\Pi_h$.

\begin{lemma}
  \label{thm:Piproj_bounds}
  Assume that $\tau$ and $f_n$ are uniformly bounded. Under the
  same conditions as in \Cref{thm:L2proj_bounds}, the following
  estimates hold:
  \begin{subequations}
    \begin{equation}
      \label{eq:fnPi_q_K}
      \begin{split}
        \norm[0]{\boldsymbol{q} - \boldsymbol{\Pi_V q}}_{f_n,\mathcal{K}}
        & \leq C\del[0]{h_K^{s_s} + \Delta t^{s_t}} \norm{\boldsymbol{q}}_{H^{(s_t, s_s)}(\mathcal{K})} 
        \\
        & + C\del[0]{h_K^{s_s} + \Delta t^{s_t}}\norm{v}_{H^{(s_t, s_s)}(\mathcal{K})},
      \end{split}
    \end{equation}
    \begin{equation}
      \label{eq:fnPi_q_K0}
      \begin{split}
        \norm[0]{\boldsymbol{q} - \boldsymbol{\Pi_Vq}}_{f_n, \partial \mathcal{K}_0}
        & \leq C\del[0]{\Delta t^{-1/2}h_K^{s_s} + \Delta t^{s_t-1/2}}\norm{\boldsymbol{q}}_{H^{(s_t, s_s)}(\mathcal{K})}
        \\
        & + C\del[0]{h_K^{s_s}\Delta t^{-1/2} + \Delta t^{s_t-1/2}}\norm{v}_{H^{(s_t, s_s)}(\mathcal{K})}.
      \end{split}
    \end{equation}
  \end{subequations}
\end{lemma}
\begin{proof}
  These estimates follow from substituting \cref{eq:fnL2_v_QK},
  \cref{eq:fnL2_q_K}, \cref{eq:fnL2_q_QK}, and \cref{eq:fnL2_q_K0} in
  \Cref{thm:proj_bounds}.
\end{proof}


\subsection{The \emph{a priori} error estimates}
\label{ss:apriorierrorestimates}
In this section we present the main result of this paper, namely
\emph{a priori} error estimates for the space-time HDG method
\cref{eq:HDGProblem}.

In order to obtain \emph{a priori} error estimates, we first require
to obtain the error equations and a bound for the projection
errors. Define these projection errors as
\begin{equation}
  \label{eq:projerrors}
  \begin{split}
    \boldsymbol{\varepsilon^{q}}_h &= \boldsymbol{\Pi_V q} - \boldsymbol{q}_h,
    \quad
    \boldsymbol{\varepsilon^{q-}}_h = \boldsymbol{\Pi_{V^-} q} - \boldsymbol{q^-}_h,
    \quad
    \varepsilon_h^v = \Pi_W v - v_h,
    \\
    \varepsilon_h^{\lambda} &= P^{f_n}_M v - \lambda_h,
    \quad
    \varepsilon_h^{\lambda -} = P^{f_{n-1}}_{M^-} v - \lambda_h^-,    
  \end{split}
\end{equation}
where $\boldsymbol{\Pi_{V^-}}$ and $P^{f_{n-1}}_{M^-}$ denote the
projection onto the spaces $\boldsymbol{V}_h$ and $M_h$, respectively,
defined on the time slab $n-1$ for $n > 0$. We furthermore remark that
$\boldsymbol{\varepsilon_h^{q-}}$ = 0 and $\varepsilon_h^{\lambda -} = 0$
when $n = 0$.


The following lemma describes the error equations.

\begin{lemma}
  \label{lem:errorequations}
  The error equations are given by
  \begin{subequations}
    \begin{equation}
      \label{eq:error_q_equation_simple}
      \begin{split}
        & \del[1]{\varepsilon_h^v, f_n \nabla \cdot \boldsymbol{r}_h }_{\mathcal{T}^n}
        - \del[1]{\boldsymbol{\varepsilon^q}_h, \boldsymbol{r}_h f'_n}_{\mathcal{T}^n}
        + \langle \boldsymbol{\varepsilon^q}_h, \boldsymbol{r}_h f_n \rangle_{\mathcal{F}_{\Omega}^n(t_{n+1})}
        \\
        &- \del[1]{\boldsymbol{\varepsilon^q}_h, f_n \partial_t \boldsymbol{r}_h}_{\mathcal{T}^n}
        - \langle \varepsilon_h^{\lambda}, \boldsymbol{r}_h\cdot \boldsymbol{n} f_n \rangle_{\mathcal{F}_{\mathcal{Q}}^n}
        \\
        =&
        \langle \boldsymbol{\varepsilon^{q-}}_h, \boldsymbol{r}_h f_n \rangle_{\mathcal{F}_{\Omega}^n(t_{n})} - \del[1]{\boldsymbol{\Pi_V q} - \boldsymbol{q}, f_n \partial_t \boldsymbol{r}_h}_{\mathcal{T}^n}
        - \del[1]{\boldsymbol{\Pi_V q} - \boldsymbol{q}, \boldsymbol{r}_h f'_n}_{\mathcal{T}^n}
        \\
        &+ \langle \boldsymbol{\Pi_V q} - \boldsymbol{q}, \boldsymbol{r}_h f_n \rangle_{\mathcal{F}_{\Omega}^n(t_{n+1})}
        - \langle \boldsymbol{\Pi_{V^-} q} - \boldsymbol{q}, \boldsymbol{r}_h f_n \rangle_{\mathcal{F}_{\Omega}^n(t_{n})},        
      \end{split}
    \end{equation}
    \begin{equation}
      \label{eq:error_v_equation_simple}
      \del[1]{\boldsymbol{\varepsilon^q}_h, f_n \nabla w_h }_{\mathcal{T}^n}
      - \langle \widehat{\boldsymbol{\varepsilon}}_h \cdot \boldsymbol{n}, w_h f_n \rangle_{\mathcal{F}_{\mathcal{Q}}^n}
      = 0,
    \end{equation}
    %
    %
    \begin{equation}
      \label{eq:error_lambda_equation_simple}
      \begin{split}
      \langle \widehat{\boldsymbol{\varepsilon}}_h
      \cdot \boldsymbol{n},& \mu_h f_n \rangle_{\mathcal{F}_{\mathcal{Q}}^n}
      - \langle \varepsilon_h^{\lambda}, f_n \partial_t \mu_h \rangle_{\mathcal{F}_{\mathcal{S}}^n}
      - \langle \varepsilon_h^{\lambda}, \mu_h f'_n \rangle_{\mathcal{F}_{\mathcal{S}}^n}
      + \llangle \varepsilon_h^{\lambda}, \mu_h f_n \rrangle_{\partial \mathcal{E}_{\mathcal{S}}^n(t_{n+1})}
      \\
      =&
      \llangle \varepsilon_h^{\lambda-}, \mu_h f_n \rrangle_{\partial \mathcal{E}_{\mathcal{S}}^n(t_{n})}
      + \llangle P^{f_n}_M v - v, \mu_h f_n \rrangle_{\partial \mathcal{E}_{\mathcal{S}}^n(t_{n+1})}
      \\
      &- \llangle P^{f_{n-1}}_{M^-} v - v, \mu_h f_n \rrangle_{\partial \mathcal{E}_{\mathcal{S}}^n(t_{n})},
      \end{split}
    \end{equation}
    \label{eq:error_equations_simple}
  \end{subequations}
  where
  $\widehat{\boldsymbol{\varepsilon}}_h =
  \boldsymbol{\varepsilon^{q}}_h - \tau \del[1]{\varepsilon_h^{v} -
    \varepsilon_h^{\lambda}}\boldsymbol{n}$.
\end{lemma}
\begin{proof}
  Substituting the exact solution $\del[0]{\boldsymbol{q}, v}$ to
  \cref{eq:mixedformlinfreesurface} into the space-time HDG method
  \cref{eq:HDGProblem}, we find:
  \begin{subequations}
    \begin{multline}
      \label{eq:exact_q_equation}
      - \del[1]{\boldsymbol{q}, f_n\partial_t \boldsymbol{r}_h}_{\mathcal{T}^n}
      - \del[1]{\boldsymbol{q}, \boldsymbol{r}_h f'_n}_{\mathcal{T}^n}
      + \langle \boldsymbol{q}, \boldsymbol{r}_h f_n \rangle_{\mathcal{F}_{\Omega}^n(t_{n+1})}
      \\
      + \del[1]{v, f_n \nabla \cdot \boldsymbol{r}_h}_{\mathcal{T}^n}
      - \langle v, \boldsymbol{r}_h\cdot \boldsymbol{n} f_n \rangle_{\mathcal{F}_{\mathcal{Q}}^n} =
      \langle \boldsymbol{q}, \boldsymbol{r}_h f_n \rangle_{\mathcal{F}_{\Omega}^n(t_{n})},
    \end{multline}
    \begin{equation}
      \label{eq:exact_v_equation}
      -\del[1]{w_h, f_n\nabla \cdot \boldsymbol{q}}_{\mathcal{T}^n} = 0,                                        
    \end{equation}
    %
    %
    \begin{multline}
      \label{eq:exact_lambda_equation}
      \langle \boldsymbol{q} \cdot \boldsymbol{n}, \mu_h f_n  \rangle_{\mathcal{F}_{\mathcal{Q}}^n}
      - \langle v , f_n \partial_t \mu_h \rangle_{\mathcal{F}_{\mathcal{S}}^n}
      - \langle v, \mu_h f'_n \rangle_{\mathcal{F}_{\mathcal{S}}^n}
      + \llangle v, \mu_h f_n \rrangle_{\partial \mathcal{E}_{\mathcal{S}}^n(t_{n+1})}
      \\
      =
      \llangle v, \mu_h f_n \rrangle_{\partial \mathcal{E}_{\mathcal{S}}^n(t_{n})}.
    \end{multline}
    \label{eq:exactProblem}
  \end{subequations}
  Subtracting now \eqref{eq:HDGProblem} from \eqref{eq:exactProblem},
  we obtain
  \begin{subequations}
    \begin{multline}
      \label{eq:error_q_equation}
      - \del[1]{\boldsymbol{q} - \boldsymbol{q}_h, f_n \partial_t\boldsymbol{r}_h}_{\mathcal{T}^n}
      - \del[1]{\boldsymbol{q} - \boldsymbol{q}_h, \boldsymbol{r}_h f'_n }_{\mathcal{T}^n}
      + \langle \boldsymbol{q} - \boldsymbol{q}_h, \boldsymbol{r}_h f_n \rangle_{\mathcal{F}_{\Omega}^n(t_{n+1})}
      \\
      + \del[1]{v - v_h, f_n \nabla \cdot \boldsymbol{r}_h}_{\mathcal{T}^n}
      - \langle v - \lambda_h, \boldsymbol{r}_h\cdot \boldsymbol{n} f_n \rangle_{\mathcal{F}_{\mathcal{Q}}^n} =
      \langle \boldsymbol{q} - \boldsymbol{q}_h^-, \boldsymbol{r}_h f_n \rangle_{\mathcal{F}_{\Omega}^n(t_{n})},
    \end{multline}
    \begin{equation}
      \label{eq:error_v_equation}
      -\del[1]{w_h f_n, \nabla\cdot(\boldsymbol{q} - \boldsymbol{q}_h)}_{\mathcal{T}^n} 
      - \langle \tau\del{v_h - \lambda_h}, w_h f_n \rangle_{\mathcal{F}_{\mathcal{Q}}^n} = 0,                                       
    \end{equation}
    %
    \begin{multline}
      \label{eq:error_lambda_equation}
      \langle \del[1]{\boldsymbol{q} - \boldsymbol{q}_h}\cdot \boldsymbol{n}
      + \tau\del[1]{v_h - \lambda_h}, \mu_h f_n \rangle_{\mathcal{F}_{\mathcal{Q}}^n}
      - \langle v - \lambda_h, f_n \partial_t \mu_h \rangle_{\mathcal{F}_{\mathcal{S}}^n}
      \\
      - \langle v - \lambda_h, \mu_h f'_n \rangle_{\mathcal{F}_{\mathcal{S}}^n}
      + \llangle v - \lambda_h, \mu_h f_n \rrangle_{\partial \mathcal{E}_{\mathcal{S}}^n(t_{n+1})} =
      \llangle v - \lambda_h^-, \mu_h f_n \rrangle_{\partial \mathcal{E}_{\mathcal{S}}^n(t_{n})}.
    \end{multline}
    \label{eq:error_equations}
  \end{subequations}

  We next split the numerical errors as
  $\boldsymbol{q} - \boldsymbol{q}_h = \boldsymbol{q} -
  \boldsymbol{\Pi_V q} + \boldsymbol{\varepsilon^q}_h$,
  $\boldsymbol{q} - \boldsymbol{q}_h^- = \boldsymbol{q} -
  \boldsymbol{\Pi_{V^-} q} + \boldsymbol{\varepsilon^{q-}}_h$,
  $v - v_h = v - \Pi_W v + \varepsilon_h^{v}$,
  $v - \lambda_h = v - P^{f_n}_M v + \varepsilon_h^{\lambda}$ and
  $v - \lambda_h^- = v - P^{f_{n-1}}_{M^-} v + \varepsilon_h^{\lambda
    -}$. Note also that
  \begin{equation}
    \begin{aligned}
      \boldsymbol{q} - \widehat{\boldsymbol{q}}_h
      & = \boldsymbol{q} - \boldsymbol{q}_h - \tau \del[1]{\lambda_h - v_h}\boldsymbol{n}
      = \boldsymbol{q} - \boldsymbol{q}_h - \tau \del[1]{v - v_h - v + \lambda_h}\boldsymbol{n}
      \\
      & = 
      \boldsymbol{\varepsilon^q}_h - \tau \del[1]{\varepsilon_h^{v} - \varepsilon_h^{\lambda}}\boldsymbol{n}
      + \boldsymbol{q} - \boldsymbol{\Pi_V q} - \tau \del[1]{v - \Pi_W v - \del[1]{v - P^{f_n}_M v}}\boldsymbol{n} 
      \\
      & = 
      \widehat{\boldsymbol{\varepsilon}}_h + \boldsymbol{q} - \boldsymbol{\Pi_V q}
      - \tau \del[1]{v - \Pi_W v - \del[1]{v - P^{f_n}_M v}}\boldsymbol{n}.
    \end{aligned}
  \end{equation}
  We will write \cref{eq:error_equations} in terms of the projection
  and approximation errors.

  Consider first \cref{eq:error_q_equation}:
  \begin{multline}
    \label{eq:error_q_equation2}
    - \del[1]{\boldsymbol{\varepsilon^q}_h, f_n \partial_t \boldsymbol{r}_h}_{\mathcal{T}^n}
    - \del[1]{\boldsymbol{\varepsilon^q}_h, \boldsymbol{r}_h f'_n}_{\mathcal{T}^n}
    + \langle \boldsymbol{\varepsilon^q}_h, \boldsymbol{r}_h f_n \rangle_{\mathcal{F}_{\Omega}^n(t_{n+1})}
    + \del[1]{\varepsilon_h^v, f_n \nabla \cdot \boldsymbol{r}_h }_{\mathcal{T}^n}
    \\
    - \langle \varepsilon_h^{\lambda}, \boldsymbol{r}_h\cdot \boldsymbol{n} f_n \rangle_{\mathcal{F}_{\mathcal{Q}}^n}
    \\
    = - \del[1]{\boldsymbol{\Pi_V q} - \boldsymbol{q}, f_n \partial_t \boldsymbol{r}_h}_{\mathcal{T}^n}
    - \del[1]{\boldsymbol{\Pi_V q} - \boldsymbol{q}, \boldsymbol{r}_h f'_n}_{\mathcal{T}^n}
    + \langle \boldsymbol{\Pi_V q} - \boldsymbol{q}, \boldsymbol{r}_h f_n \rangle_{\mathcal{F}_{\Omega}^n(t_{n+1})}
    \\
    + \del[1]{\Pi_W v - v, f_n \nabla \cdot \boldsymbol{r}_h }_{\mathcal{T}^n}
    - \langle P^{f_n}_M v - v, \boldsymbol{r}_h\cdot \boldsymbol{n} f_n  \rangle_{\mathcal{F}_{\mathcal{Q}}^n}
    \\
    + \langle \boldsymbol{\varepsilon^{q-}}_h, \boldsymbol{r}_h f_n \rangle_{\mathcal{F}_{\Omega}^n(t_{n})}
    - \langle \boldsymbol{\Pi_{V^-} q} - \boldsymbol{q}, \boldsymbol{r}_h f_n \rangle_{\mathcal{F}_{\Omega}^n(t_{n})},
  \end{multline}
  which simplifies to
  \begin{multline}
    \label{eq:error_q_equation_simple1}
    - \del[1]{\boldsymbol{\varepsilon^q}_h, f_n \partial_t \boldsymbol{r}_h}_{\mathcal{T}^n}
    - \del[1]{\boldsymbol{\varepsilon^q}_h, \boldsymbol{r}_h f'_n}_{\mathcal{T}^n}
    + \langle \boldsymbol{\varepsilon^q}_h, \boldsymbol{r}_h f_n \rangle_{\mathcal{F}_{\Omega}^n(t_{n+1})}
    \\
    + \del[1]{\varepsilon_h^v, f_n \nabla \cdot \boldsymbol{r}_h }_{\mathcal{T}^n}
    - \langle \varepsilon_h^{\lambda}, \boldsymbol{r}_h\cdot \boldsymbol{n} f_n \rangle_{\mathcal{F}_{\mathcal{Q}}^n}
    = - \del[1]{\boldsymbol{\Pi_V q} - \boldsymbol{q}, f_n \partial_t \boldsymbol{r}_h}_{\mathcal{T}^n}
    \\ 
    - \del[1]{\boldsymbol{\Pi_V q} - \boldsymbol{q}, \boldsymbol{r}_h f'_n}_{\mathcal{T}^n}
    + \langle \boldsymbol{\Pi_V q} - \boldsymbol{q}, \boldsymbol{r}_h f_n \rangle_{\mathcal{F}_{\Omega}^n(t_{n+1})}
    \\
    + \langle \boldsymbol{\varepsilon^{q-}}_h, \boldsymbol{r}_h f_n \rangle_{\mathcal{F}_{\Omega}^n(t_{n})}
    - \langle \boldsymbol{\Pi_{V^-} q} - \boldsymbol{q}, \boldsymbol{r}_h f_n \rangle_{\mathcal{F}_{\Omega}^n(t_{n})},
  \end{multline}
  using the properties of the projections $\Pi_W$ and
  $P^{f_n}_M$. This proves \cref{eq:error_q_equation_simple}.

  We consider next \cref{eq:error_v_equation}. Integrating
  \cref{eq:error_v_equation} by parts in space,
  \begin{equation}
    \del[1]{\boldsymbol{q} - \boldsymbol{q}_h, f_n \nabla w_h }_{\mathcal{T}^n} - 
    \langle \del[1]{\boldsymbol{q} - \boldsymbol{q}_h} 
    \cdot \boldsymbol{n} - \tau\del[1]{\lambda_h - v_h}, w_h f_n \rangle_{\mathcal{F}_{\mathcal{Q}}^n} = 0.
  \end{equation}
  We next write this equation in terms of the projection and
  approximation errors:
  \begin{equation}
    \label{eq:error_v_equation2}
    \begin{split}
      \del[1]{\boldsymbol{\varepsilon^q}_h, f_n \nabla w_h }_{\mathcal{T}^n} 
      - \langle & \widehat{\boldsymbol{\varepsilon}}_h \cdot \boldsymbol{n}, w_h f_n \rangle_{\mathcal{F}_{\mathcal{Q}}^n}
      \\
      =& 
      \del[1]{\boldsymbol{\Pi_V q} - \boldsymbol{q}, f_n \nabla w_h }_{\mathcal{T}^n}
      - \langle \tau \del[1]{P^{f_n}_M v - v}, w_h f_n \rangle_{\mathcal{F}_{\mathcal{Q}}^n}
      \\
      &- \langle \del{\boldsymbol{\Pi_V q} - \boldsymbol{q}}\cdot \boldsymbol{n}- \tau \del[1]{\Pi_W v - v},
      w_h f_n \rangle_{\mathcal{F}_{\mathcal{Q}}^n}.
    \end{split}
  \end{equation}
  Using the properties of the projections $\Pi_h$ and $P^{f_n}_M$, we
  obtain
  \begin{equation}
    \label{eq:error_v_equation_simple1}
    \del[1]{\boldsymbol{\varepsilon^q}_h, f_n \nabla w_h }_{\mathcal{T}^n}
    - \langle \widehat{\boldsymbol{\varepsilon}}_h \cdot \boldsymbol{n}, w_h f_n \rangle_{\mathcal{F}_{\mathcal{Q}}^n}
    = 0,
  \end{equation}
  proving \cref{eq:error_v_equation_simple}.

  Finally, we write \cref{eq:error_lambda_equation} in terms of the
  projection and approximation errors:
  \begin{equation}
    \label{eq:error_lambda_equation2}
    \begin{split}
      \langle \widehat{\boldsymbol{\varepsilon}}_h
      \cdot \boldsymbol{n},& \mu_h f_n \rangle_{\mathcal{F}_{\mathcal{Q}}^n}
      - \langle \varepsilon_h^{\lambda}, f_n \partial_t \mu_h \rangle_{\mathcal{F}_{\mathcal{S}}^n}
      - \langle \varepsilon_h^{\lambda}, \mu_h f'_n \rangle_{\mathcal{F}_{\mathcal{S}}^n}
      + \llangle \varepsilon_h^{\lambda}, \mu_h f_n \rrangle_{\partial \mathcal{E}_{\mathcal{S}}^n(t_{n+1})}
      \\
      =& 
      \langle \del{\boldsymbol{\Pi_V q} - \boldsymbol{q}}\cdot \boldsymbol{n} - \tau \del[1]{\Pi_W v - v},
      \mu_h f_n \rangle_{\mathcal{F}_{\mathcal{Q}}^n}
      + \langle \tau \del[1]{P^{f_n}_M v - v}, \mu_h f_n \rangle_{\mathcal{F}_{\mathcal{Q}}^n}
      \\
      &- \langle P^{f_n}_M v - v, f_n \partial_t \mu_h \rangle_{\mathcal{F}_{\mathcal{S}}^n}
      - \langle P^{f_n}_M v - v, \mu_h f'_n \rangle_{\mathcal{F}_{\mathcal{S}}^n}
      + \llangle P^{f_n}_M v - v, \mu_h f_n \rrangle_{\partial \mathcal{E}_{\mathcal{S}}^n(t_{n+1})}
      \\
      &+ \llangle \varepsilon_h^{\lambda-}, \mu_h f_n \rrangle_{\partial \mathcal{E}_{\mathcal{S}}^n(t_{n})}
      - \llangle P^{f_{n-1}}_{M^-} v - v, \mu_h f_n \rrangle_{\partial \mathcal{E}_{\mathcal{S}}^n(t_{n})}.      
    \end{split}
  \end{equation}
  Using the properties of the projections $\Pi_h$ and $P^{f_n}_M$, we
  obtain
  \begin{equation}
    \label{eq:error_lambda_equation_simple1}
    \begin{split}
    \langle \widehat{\boldsymbol{\varepsilon}}_h
    \cdot \boldsymbol{n}, &\mu_h f_n \rangle_{\mathcal{F}_{\mathcal{Q}}^n}
    - \langle \varepsilon_h^{\lambda}, f_n \partial_t \mu_h \rangle_{\mathcal{F}_{\mathcal{S}}^n}
    - \langle \varepsilon_h^{\lambda}, \mu_h f'_n \rangle_{\mathcal{F}_{\mathcal{S}}^n}
    + \llangle \varepsilon_h^{\lambda}, \mu_h f_n \rrangle_{\partial \mathcal{E}_{\mathcal{S}}^n(t_{n+1})}
    \\
    =&
     \llangle \varepsilon_h^{\lambda-}, \mu_h f_n \rrangle_{\partial \mathcal{E}_{\mathcal{S}}^n(t_{n})}
     + \llangle P^{f_n}_M v - v, \mu_h f_n \rrangle_{\partial \mathcal{E}_{\mathcal{S}}^n(t_{n+1})}
     \\
    &- \llangle P^{f_{n-1}}_{M^-} v - v, \mu_h f_n \rrangle_{\partial \mathcal{E}_{\mathcal{S}}^n(t_{n})},      
    \end{split}
  \end{equation}
  proving \cref{eq:error_lambda_equation_simple}.
\end{proof}

Next, we prove a bound for the projection errors.

\begin{lemma}
  \label{lem:bounds_epsilons}
  The following bound holds for the projection errors:
  \begin{equation}
    \begin{split}
	    \tfrac{\alpha}{2} \norm[0]{\boldsymbol{\varepsilon^q}_h}_{f_n, \mathcal{T}^n}^2
	    &+ \tfrac{\alpha}{2} \norm[0]{\varepsilon^{\lambda}_h}_{f_n, \mathcal{F}_{\mathcal{S}}^n}^2
	    +  e^{-\alpha \Delta t}\norm[0]{\boldsymbol{\varepsilon^{q}}_h}^2_{\mathcal{F}_{\Omega}^n(t_{n+1})}
	    +  e^{-\alpha \Delta t}\norm[0]{\varepsilon^{\lambda}_h}^2_{\partial\mathcal{E}_{\mathcal{S}}^n(t_{n+1})}
	    \\      
	    \leq
	    &  \norm[0]{\boldsymbol{\varepsilon^{q-}}_h}^2_{\mathcal{F}_{\Omega}^n(t_n)}
	    +  \norm[0]{\varepsilon^{\lambda-}_h}^2_{\partial\mathcal{E}_{\mathcal{S}}^n(t_n)}
	    + C \Delta t^{-2}\norm[0]{\boldsymbol{q} - \boldsymbol{\Pi_V q}}^2_{f_n, \mathcal{T}^n}
	    \\
	    &+ C \norm[0]{\boldsymbol{\Pi_V q} - \boldsymbol{q}}^2_{f_n, \mathcal{T}^n}
	    + C \Delta t^{-1}\norm[0]{\boldsymbol{\Pi_V q} - \boldsymbol{q}}^2_{f_n,\mathcal{F}_{\Omega}^n(t_{n+1})}
	    \\
	    &+ C \Delta t^{-1}\norm[0]{\boldsymbol{\Pi_{V^-} q} - \boldsymbol{q}}^2_{f_n,\mathcal{F}_{\Omega}^n(t_n)}
	    + C \Delta t^{-1} \norm[0]{P^{f_n}_M v - v}^2_{f_n, \partial \mathcal{E}_{\mathcal{S}}^n(t_{n+1})}
	    \\
	    &+ C \Delta t^{-1} \norm[0]{P^{f_{n-1}}_{M^-} v - v}^2_{f_n, \partial \mathcal{E}_{\mathcal{S}}^n(t_n)}.
    \end{split}
  \end{equation}
\end{lemma}
\begin{proof}
  Take $\boldsymbol{r}_h = \boldsymbol{\varepsilon^q}_h$ in
  \cref{eq:error_q_equation_simple}, $w_h = \varepsilon^v_h$ in
  \cref{eq:error_v_equation_simple} and
  $\mu_h = \varepsilon^{\lambda}_h$ in
  \cref{eq:error_lambda_equation_simple}. Adding the resulting
  equations we obtain
  \begin{equation}
    \begin{split}
    &- \del[1]{\boldsymbol{\varepsilon^q}_h, f_n \partial_t \boldsymbol{\varepsilon^q}_h}_{\mathcal{T}^n}
    - \del[1]{f'_n\boldsymbol{\varepsilon^q}_h, \boldsymbol{\varepsilon^q}_h}_{\mathcal{T}^n}
    + \langle f_n \boldsymbol{\varepsilon^q}_h, \boldsymbol{\varepsilon^q}_h \rangle_{\mathcal{F}_{\Omega}^n(t_{n+1})}
    \\
    &+ \del[1]{\varepsilon_h^v, f_n \nabla \cdot \boldsymbol{\varepsilon^q}_h }_{\mathcal{T}^n}
    - \langle \varepsilon_h^{\lambda}, \boldsymbol{\varepsilon^q}_h\cdot \boldsymbol{n} f_n \rangle_{\mathcal{F}_{\mathcal{Q}}^n}
    + \del[1]{\boldsymbol{\varepsilon^q}_h, f_n \nabla \varepsilon^v_h }_{\mathcal{T}^n}
    - \langle \widehat{\boldsymbol{\varepsilon}}_h \cdot \boldsymbol{n}, \varepsilon^v_h f_n \rangle_{\mathcal{F}_{\mathcal{Q}}^n}
    \\
    &+ \langle \widehat{\boldsymbol{\varepsilon}}_h
    \cdot \boldsymbol{n}, \varepsilon^{\lambda}_h f_n \rangle_{\mathcal{F}_{\mathcal{Q}}^n}
    - \langle \varepsilon_h^{\lambda}, f_n \partial_t \varepsilon^{\lambda}_h \rangle_{\mathcal{F}_{\mathcal{S}}^n}
    - \langle f'_n \varepsilon_h^{\lambda}, \varepsilon^{\lambda}_h \rangle_{\mathcal{F}_{\mathcal{S}}^n}
    + \llangle f_n \varepsilon_h^{\lambda}, \varepsilon^{\lambda}_h \rrangle_{\partial \mathcal{E}_{\mathcal{S}}^n(t_{n+1})}
    \\
    =&
    \langle \boldsymbol{\varepsilon^{q-}}_h, \boldsymbol{\varepsilon^q}_h f_n \rangle_{\mathcal{F}_{\Omega}^n(t_{n})}
    + \llangle \varepsilon_h^{\lambda-}, \varepsilon^{\lambda}_h f_n \rrangle_{\partial \mathcal{E}_{\mathcal{S}}^n(t_{n})}
    \\
    &- \del[1]{\boldsymbol{\Pi_V q} - \boldsymbol{q}, f_n \partial_t \boldsymbol{\varepsilon^q}_h}_{\mathcal{T}^n}
    - \del[1]{\boldsymbol{\Pi_V q} - \boldsymbol{q}, \boldsymbol{\varepsilon^q}_h f'_n}_{\mathcal{T}^n}
    \\
    &+ \langle \boldsymbol{\Pi_V q} - \boldsymbol{q}, \boldsymbol{\varepsilon^q}_h f_n \rangle_{\mathcal{F}_{\Omega}^n(t_{n+1})}
    - \langle \boldsymbol{\Pi_{V^-} q} - \boldsymbol{q}, \boldsymbol{\varepsilon^q}_h f_n \rangle_{\mathcal{F}_{\Omega}^n(t_{n})}
    \\
    &+ \llangle P^{f_n}_M v - v, \varepsilon^{\lambda}_h f_n \rrangle_{\partial \mathcal{E}_{\mathcal{S}}^n(t_{n+1})}   
    - \llangle P^{f_{n-1}}_{M^-} v - v, \varepsilon^{\lambda}_h f_n \rrangle_{\partial \mathcal{E}_{\mathcal{S}}^n(t_{n})}.      
    \end{split}
  \end{equation}
  Applying integration by parts with respect to time on the first and
  ninth terms, integration by parts with respect to space on the sixth
  term, and expanding out some terms:
  \begin{equation}
    \label{eq:test_equals_trial_error}
    \begin{split}
      &- \tfrac{1}{2}\del[1]{f'_n \boldsymbol{\varepsilon^q}_h, \boldsymbol{\varepsilon^q}_h}_{\mathcal{T}^n}
      + \tfrac{1}{2}\langle f_n\boldsymbol{\varepsilon^q}_h, \boldsymbol{\varepsilon^q}_h \rangle_{\mathcal{F}_{\Omega}^n(t_{n+1})} 
      + \tfrac{1}{2}\langle f_n\boldsymbol{\varepsilon^q}_h, \boldsymbol{\varepsilon^q}_h \rangle_{\mathcal{F}_{\Omega}^n(t_{n})} 
      \\
      &+ \langle \tau \del[1]{\varepsilon_h^v - \varepsilon^{\lambda}_h}, f_n \del[1]{\varepsilon_h^v - \varepsilon^{\lambda}_h} \rangle_{\mathcal{F}_{\mathcal{Q}}^n}
      - \tfrac{1}{2}\langle f'_n\varepsilon_h^{\lambda}, \varepsilon_h^{\lambda} \rangle_{\mathcal{F}_{\mathcal{S}}^n}
      + \tfrac{1}{2}\llangle f_n\varepsilon_h^{\lambda}, \varepsilon_h^{\lambda} \rrangle_{\partial \mathcal{E}_{\mathcal{S}}^n(t_{n+1})}
      \\
      &+ \tfrac{1}{2}\llangle f_n\varepsilon_h^{\lambda}, \varepsilon_h^{\lambda} \rrangle_{\partial \mathcal{E}_{\mathcal{S}}^n(t_{n})}
      \\
      =&
      \langle \boldsymbol{\varepsilon^{q-}}_h, \boldsymbol{\varepsilon^q}_h f_n \rangle_{\mathcal{F}_{\Omega}^n(t_{n})}
      + \llangle \varepsilon_h^{\lambda-}, \varepsilon^{\lambda}_h f_n \rrangle_{\partial \mathcal{E}_{\mathcal{S}}^n(t_{n})}
      - \del[1]{\boldsymbol{\Pi_V q} - \boldsymbol{q}, f_n \partial_t \boldsymbol{\varepsilon^q}_h}_{\mathcal{T}^n}
      \\
      &- \del[1]{\boldsymbol{\Pi_V q} - \boldsymbol{q}, \boldsymbol{\varepsilon^q}_h f'_n}_{\mathcal{T}^n}
      + \langle \boldsymbol{\Pi_V q} - \boldsymbol{q}, \boldsymbol{\varepsilon^q}_h f_n \rangle_{\mathcal{F}_{\Omega}^n(t_{n+1})}
      - \langle \boldsymbol{\Pi_{V^-} q} - \boldsymbol{q}, \boldsymbol{\varepsilon^q}_h f_n \rangle_{\mathcal{F}_{\Omega}^n(t_{n})}
      \\
      &+ \llangle P^{f_n}_M v - v, \varepsilon^{\lambda}_h f_n \rrangle_{\partial \mathcal{E}_{\mathcal{S}}^n(t_{n+1})}   
      - \llangle P^{f_{n-1}}_{M^-} v - v, \varepsilon^{\lambda}_h f_n \rrangle_{\partial \mathcal{E}_{\mathcal{S}}^n(t_{n})}.      
    \end{split}
  \end{equation}
  Moving the first two terms on the right hand side to the left hand
  side, we obtain
  \begin{equation}
  \label{eq:test_equals_trial_error2}
	  \begin{split}
		  &- \tfrac{1}{2}\del[1]{f'_n \boldsymbol{\varepsilon^q}_h, \boldsymbol{\varepsilon^q}_h}_{\mathcal{T}^n}
		  + \tfrac{1}{2}\langle f_n\boldsymbol{\varepsilon^q}_h, \boldsymbol{\varepsilon^q}_h \rangle_{\mathcal{F}_{\Omega}^n(t_{n+1})} 
		  + \tfrac{1}{2}\langle f_n\boldsymbol{\varepsilon^q}_h, \boldsymbol{\varepsilon^q}_h \rangle_{\mathcal{F}_{\Omega}^n(t_{n})} - \langle \boldsymbol{\varepsilon^{q-}}_h, \boldsymbol{\varepsilon^q}_h f_n \rangle_{\mathcal{F}_{\Omega}^n(t_{n})}
		  \\
		  &+ \langle \tau \del[1]{\varepsilon_h^v - \varepsilon^{\lambda}_h}, f_n \del[1]{\varepsilon_h^v - \varepsilon^{\lambda}_h} \rangle_{\mathcal{F}_{\mathcal{Q}}^n}
		  - \tfrac{1}{2}\langle f'_n\varepsilon_h^{\lambda}, \varepsilon_h^{\lambda} \rangle_{\mathcal{F}_{\mathcal{S}}^n}
		  + \tfrac{1}{2}\llangle f_n\varepsilon_h^{\lambda}, \varepsilon_h^{\lambda} \rrangle_{\partial \mathcal{E}_{\mathcal{S}}^n(t_{n+1})}
		  \\
		  &+ \tfrac{1}{2}\llangle f_n\varepsilon_h^{\lambda}, \varepsilon_h^{\lambda} \rrangle_{\partial \mathcal{E}_{\mathcal{S}}^n(t_{n})} - \llangle \varepsilon_h^{\lambda-}, \varepsilon^{\lambda}_h f_n \rrangle_{\partial \mathcal{E}_{\mathcal{S}}^n(t_{n})}
		  \\
		  =&
		  - \del[1]{\boldsymbol{\Pi_V q} - \boldsymbol{q}, f_n \partial_t \boldsymbol{\varepsilon^q}_h}_{\mathcal{T}^n}
		  \\
		  &- \del[1]{\boldsymbol{\Pi_V q} - \boldsymbol{q}, \boldsymbol{\varepsilon^q}_h f'_n}_{\mathcal{T}^n}
		  + \langle \boldsymbol{\Pi_V q} - \boldsymbol{q}, \boldsymbol{\varepsilon^q}_h f_n \rangle_{\mathcal{F}_{\Omega}^n(t_{n+1})}
		  - \langle \boldsymbol{\Pi_{V^-} q} - \boldsymbol{q}, \boldsymbol{\varepsilon^q}_h f_n \rangle_{\mathcal{F}_{\Omega}^n(t_{n})}
		  \\
		  &+ \llangle P^{f_n}_M v - v, \varepsilon^{\lambda}_h f_n \rrangle_{\partial \mathcal{E}_{\mathcal{S}}^n(t_{n+1})}   
		  - \llangle P^{f_{n-1}}_{M^-} v - v, \varepsilon^{\lambda}_h f_n \rrangle_{\partial \mathcal{E}_{\mathcal{S}}^n(t_{n})}.      
	  \end{split}
  \end{equation}
  Notice that
  \begin{equation}
	  \begin{split}
		  &\tfrac{1}{2}\langle f_n\boldsymbol{\varepsilon^q}_h, \boldsymbol{\varepsilon^q}_h \rangle_{\mathcal{F}_{\Omega}^n(t_{n})} - \langle \boldsymbol{\varepsilon^{q-}}_h, \boldsymbol{\varepsilon^q}_h f_n \rangle_{\mathcal{F}_{\Omega}^n(t_{n})} 
		  = \tfrac{1}{2}\langle \boldsymbol{\varepsilon^q}_h - \boldsymbol{\varepsilon^{q-}}_h, \boldsymbol{\varepsilon^q}_h f_n \rangle_{\mathcal{F}_{\Omega}^n(t_{n})} - \tfrac{1}{2} \langle \boldsymbol{\varepsilon^{q-}}_h, \boldsymbol{\varepsilon^q}_h f_n \rangle_{\mathcal{F}_{\Omega}^n(t_{n})}
		  \\
		  & = \tfrac{1}{2}\langle \boldsymbol{\varepsilon^q}_h - \boldsymbol{\varepsilon^{q-}}_h, \del[0]{\boldsymbol{\varepsilon^q}_h - \boldsymbol{\varepsilon^{q-}}_h} f_n \rangle_{\mathcal{F}_{\Omega}^n(t_{n})} - \tfrac{1}{2} \langle \boldsymbol{\varepsilon^{q-}}_h, \boldsymbol{\varepsilon^{q-}}_h f_n \rangle_{\mathcal{F}_{\Omega}^n(t_{n})}.
	  \end{split}
  \end{equation}
  Similarly,
  \begin{equation}
	  \begin{split}
		  &\tfrac{1}{2}\llangle f_n\varepsilon_h^{\lambda}, \varepsilon_h^{\lambda} \rrangle_{\partial \mathcal{E}_{\mathcal{S}}^n(t_{n})} - \llangle \varepsilon_h^{\lambda-}, \varepsilon^{\lambda}_h f_n \rrangle_{\partial \mathcal{E}_{\mathcal{S}}^n(t_{n})} 
		  \\
		  &= \tfrac{1}{2}\llangle \varepsilon_h^{\lambda} - \varepsilon_h^{\lambda-}, \del[0]{\varepsilon_h^{\lambda} - \varepsilon_h^{\lambda-}}f_n \rrangle_{\partial \mathcal{E}_{\mathcal{S}}^n(t_{n})} - \tfrac{1}{2}\llangle \varepsilon_h^{\lambda-}, \varepsilon^{\lambda-}_h f_n \rrangle_{\partial \mathcal{E}_{\mathcal{S}}^n(t_{n})}.
	  \end{split}
  \end{equation}
  Substituting these expressions in \cref{eq:test_equals_trial_error2} and rearranging terms, we obtain
  \begin{equation}
	  \label{eq:test_equals_trial_error3}
	  \begin{split}
		  &- \tfrac{1}{2}\del[1]{f'_n \boldsymbol{\varepsilon^q}_h, \boldsymbol{\varepsilon^q}_h}_{\mathcal{T}^n}
		  + \tfrac{1}{2}\langle f_n\boldsymbol{\varepsilon^q}_h, \boldsymbol{\varepsilon^q}_h \rangle_{\mathcal{F}_{\Omega}^n(t_{n+1})} 
		  + \tfrac{1}{2}\langle \boldsymbol{\varepsilon^q}_h - \boldsymbol{\varepsilon^{q-}}_h, \del[0]{\boldsymbol{\varepsilon^q}_h - \boldsymbol{\varepsilon^{q-}}_h} f_n \rangle_{\mathcal{F}_{\Omega}^n(t_{n})}
		  \\
		  &+ \langle \tau \del[1]{\varepsilon_h^v - \varepsilon^{\lambda}_h}, f_n \del[1]{\varepsilon_h^v - \varepsilon^{\lambda}_h} \rangle_{\mathcal{F}_{\mathcal{Q}}^n}
		  - \tfrac{1}{2}\langle f'_n\varepsilon_h^{\lambda}, \varepsilon_h^{\lambda} \rangle_{\mathcal{F}_{\mathcal{S}}^n}
		  + \tfrac{1}{2}\llangle f_n\varepsilon_h^{\lambda}, \varepsilon_h^{\lambda} \rrangle_{\partial \mathcal{E}_{\mathcal{S}}^n(t_{n+1})}
		  \\
		  &+ \tfrac{1}{2}\llangle \varepsilon_h^{\lambda} - \varepsilon_h^{\lambda-}, \del[0]{\varepsilon_h^{\lambda} - \varepsilon_h^{\lambda-}}f_n \rrangle_{\partial \mathcal{E}_{\mathcal{S}}^n(t_{n})}
		  \\
		  =&
                  \tfrac{1}{2} \langle \boldsymbol{\varepsilon^{q-}}_h, \boldsymbol{\varepsilon^{q-}}_h f_n \rangle_{\mathcal{F}_{\Omega}^n(t_{n})}                  
                  + \tfrac{1}{2}\llangle \varepsilon_h^{\lambda-}, \varepsilon^{\lambda-}_h f_n \rrangle_{\partial \mathcal{E}_{\mathcal{S}}^n(t_{n})}
		  \\
		  &- \del[1]{\boldsymbol{\Pi_V q} - \boldsymbol{q}, f_n \partial_t \boldsymbol{\varepsilon^q}_h}_{\mathcal{T}^n}
		  \\
		  &- \del[1]{\boldsymbol{\Pi_V q} - \boldsymbol{q}, \boldsymbol{\varepsilon^q}_h f'_n}_{\mathcal{T}^n}
		  + \langle \boldsymbol{\Pi_V q} - \boldsymbol{q}, \boldsymbol{\varepsilon^q}_h f_n \rangle_{\mathcal{F}_{\Omega}^n(t_{n+1})}
		  - \langle \boldsymbol{\Pi_{V^-} q} - \boldsymbol{q}, \boldsymbol{\varepsilon^q}_h f_n \rangle_{\mathcal{F}_{\Omega}^n(t_{n})}
		  \\
		  &+ \llangle P^{f_n}_M v - v, \varepsilon^{\lambda}_h f_n \rrangle_{\partial \mathcal{E}_{\mathcal{S}}^n(t_{n+1})}   
		  - \llangle P^{f_{n-1}}_{M^-} v - v, \varepsilon^{\lambda}_h f_n \rrangle_{\partial \mathcal{E}_{\mathcal{S}}^n(t_{n})}.      
	  \end{split}
  \end{equation}
  Recall that $f_n = e^{-\alpha(t-t_n)}$, where $\alpha > 0$ is
  constant, and so $f'_n = -\alpha f_n$. Moreover, $f_n(t_n) = 1$ and
  $f_n(t_{n+1}) = e^{-\alpha \Delta t}$. With these definitions and by
  the Cauchy--Schwarz inequality applied to the right hand side of
  \cref{eq:test_equals_trial_error3}, we obtain
  \begin{equation}
    \label{eq:norms_error}
    \begin{split}
      &\tfrac{\alpha}{2} \norm[0]{\boldsymbol{\varepsilon^q}_h}_{f_n, \mathcal{T}^n}^2
      + \tfrac{\alpha}{2} \norm[0]{\varepsilon^{\lambda}_h}_{f_n, \mathcal{F}_{\mathcal{S}}^n}^2
      + \tfrac{e^{-\alpha \Delta t}}{2} \norm[0]{\boldsymbol{\varepsilon^{q}}_h}^2_{\mathcal{F}_{\Omega}^n(t_{n+1})}
      + \tfrac{e^{-\alpha \Delta t}}{2} \norm[0]{\varepsilon^{\lambda}_h}^2_{\partial\mathcal{E}_{\mathcal{S}}^n(t_{n+1})}
      \\
      \leq
      & \tfrac{1}{2}\norm[0]{\boldsymbol{\varepsilon^{q-}}_h}^2_{\mathcal{F}_{\Omega}^n(t_n)}
      + \tfrac{1}{2}\norm[0]{\varepsilon^{\lambda-}_h}^2_{\partial\mathcal{E}_{\mathcal{S}}^n(t_n)}
      \\
      &+ \norm[0]{\boldsymbol{q} - \boldsymbol{\Pi_V q}}_{f_n, \mathcal{T}^n} \norm[0]{\partial_t \boldsymbol{\varepsilon^q}_h}_{f_n, \mathcal{T}^n}
      + \alpha \norm[0]{\boldsymbol{\Pi_V q} - \boldsymbol{q}}_{f_n, \mathcal{T}^n} \norm[0]{\boldsymbol{\varepsilon^q}_h}_{f_n, \mathcal{T}^n}
      \\
      &+ \norm[0]{\boldsymbol{\Pi_V q} - \boldsymbol{q}}_{f_n,\mathcal{F}_{\Omega}^n(t_{n+1})} \norm[0]{\boldsymbol{\varepsilon^q}_h}_{f_n,\mathcal{F}_{\Omega}^n(t_{n+1})}
      + \norm[0]{\boldsymbol{\Pi_{V^-} q} - \boldsymbol{q}}_{f_n,\mathcal{F}_{\Omega}^n(t_n)} \norm[0]{\boldsymbol{\varepsilon^q}_h}_{f_n,\mathcal{F}_{\Omega}^n(t_n)}
      \\
      &
      + \norm[0]{P^{f_n}_M v - v}_{f_n, \partial \mathcal{E}_{\mathcal{S}}^n(t_{n+1})} \norm[0]{\varepsilon_h^{\lambda}}_{f_n, \partial \mathcal{E}_{\mathcal{S}}^n(t_{n+1})}
      + \norm[0]{P^{f_{n-1}}_{M^-} v - v}_{f_n, \partial \mathcal{E}_{\mathcal{S}}^n(t_n)} \norm[0]{\varepsilon_h^{\lambda}}_{f_n, \partial \mathcal{E}_{\mathcal{S}}^n(t_n)}.
    \end{split}
  \end{equation}
  Note that, by \cref{eq:inv_trace_ineq_K0} and
  \cref{eq:inv_trace_ineq_F0},
  \begin{equation}
  	  \label{eq:byLemma44}
	  \begin{split}
	  \norm[0]{\boldsymbol{\varepsilon^q}_h}_{f_n,\mathcal{F}_{\Omega}^n(t_n)}
	  &\leq C \Delta t^{-1/2}\norm[0]{\boldsymbol{\varepsilon^q}_h}_{f_n,\mathcal{T}^n},
	  \\
	  \norm[0]{\varepsilon_h^{\lambda}}_{f_n, \partial \mathcal{E}_{\mathcal{S}}^n(t_n)}
	  &\leq C \Delta t^{-1/2}\norm[0]{\varepsilon_h^{\lambda}}_{f_n, \mathcal{F}_{\mathcal{S}}^n}.
	  \end{split}
  \end{equation}
  Furthermore, combining a standard inverse inequality and using
  equivalence of the norms $\norm{\cdot}_{\mathcal{K}}$ and
  $\norm{\cdot}_{f_n,\mathcal{K}}$ we also have
  \begin{equation}
    \label{eq:partialtepsinvineq}
    \norm[0]{\partial_t \boldsymbol{\varepsilon^q}_h}_{f_n, \mathcal{T}^n}
    \leq C \Delta t^{-1} \norm[0]{\boldsymbol{\varepsilon^q}_h}_{f_n, \mathcal{T}^n}.
  \end{equation}
  Using \cref{eq:byLemma44} and \cref{eq:partialtepsinvineq} for the
  right hand side of \cref{eq:norms_error} and multiplying by 2,
  \begin{equation}
    \label{eq:norms_error2}
    \begin{split}
      &\alpha \norm[0]{\boldsymbol{\varepsilon^q}_h}_{f_n, \mathcal{T}^n}^2
      + \alpha \norm[0]{\varepsilon^{\lambda}_h}_{f_n, \mathcal{F}_{\mathcal{S}}^n}^2
      +  e^{-\alpha \Delta t}\norm[0]{\boldsymbol{\varepsilon^{q}}_h}^2_{\mathcal{F}_{\Omega}^n(t_{n+1})}
      +  e^{-\alpha \Delta t}\norm[0]{\varepsilon^{\lambda}_h}^2_{\partial\mathcal{E}_{\mathcal{S}}^n(t_{n+1})}
      \\
      \leq
      & \norm[0]{\boldsymbol{\varepsilon^{q-}}_h}^2_{\mathcal{F}_{\Omega}^n(t_n)}
      + \norm[0]{\varepsilon^{\lambda-}_h}^2_{\partial\mathcal{E}_{\mathcal{S}}^n(t_n)}
      \\
      &+ C \Delta t^{-1}\norm[0]{\boldsymbol{q} - \boldsymbol{\Pi_V q}}_{f_n, \mathcal{T}^n} \norm[0]{\boldsymbol{\varepsilon^q}_h}_{f_n, \mathcal{T}^n}      
      + 2\alpha \norm[0]{\boldsymbol{\Pi_V q} - \boldsymbol{q}}_{f_n, \mathcal{T}^n} \norm[0]{\boldsymbol{\varepsilon^q}_h}_{f_n, \mathcal{T}^n}
      \\
      &+ C\Delta t^{-1/2}\norm[0]{\boldsymbol{\Pi_V q} - \boldsymbol{q}}_{f_n,\mathcal{F}_{\Omega}^n(t_{n+1})} \norm[0]{\boldsymbol{\varepsilon^q}_h}_{f_n, \mathcal{T}^n}
      \\ 
      &+ C\Delta t^{-1/2}\norm[0]{\boldsymbol{\Pi_{V^-} q} - \boldsymbol{q}}_{f_n,\mathcal{F}_{\Omega}^n(t_n)} \norm[0]{\boldsymbol{\varepsilon^q}_h}_{f_n, \mathcal{T}^n}
      \\
      &+ C \Delta t^{-1/2} \norm[0]{P^{f_n}_M v - v}_{f_n, \partial \mathcal{E}_{\mathcal{S}}^n(t_{n+1})} \norm[0]{\varepsilon^{\lambda}_h}_{f_n, \mathcal{F}_{\mathcal{S}}^n}
      \\
      &+ C \Delta t^{-1/2} \norm[0]{P^{f_{n-1}}_{M^-} v - v}_{f_n, \partial \mathcal{E}_{\mathcal{S}}^n(t_n)} \norm[0]{\varepsilon^{\lambda}_h}_{f_n, \mathcal{F}_{\mathcal{S}}^n}.
    \end{split}
  \end{equation}
  Applying Young's inequality to the right hand side of
  \cref{eq:norms_error2}, we obtain
  \begin{equation}
    \label{eq:norms_error3}
    \begin{split}
      & \alpha \norm[0]{\boldsymbol{\varepsilon^q}_h}_{f_n, \mathcal{T}^n}^2
      + \alpha \norm[0]{\varepsilon^{\lambda}_h}_{f_n, \mathcal{F}_{\mathcal{S}}^n}^2
      +  e^{-\alpha \Delta t} \norm[0]{\boldsymbol{\varepsilon^{q}}_h}^2_{\mathcal{F}_{\Omega}^n(t_{n+1})}
      +  e^{-\alpha \Delta t} \norm[0]{\varepsilon^{\lambda}_h}^2_{\partial\mathcal{E}_{\mathcal{S}}^n(t_{n+1})}
      \\
      \leq & 
      \norm[0]{\boldsymbol{\varepsilon^{q-}}_h}^2_{\mathcal{F}_{\Omega}^n(t_n)}
      + \norm[0]{\varepsilon^{\lambda-}_h}^2_{\partial\mathcal{E}_{\mathcal{S}}^n(t_n)}
      + \frac{C}{4\delta_1} \Delta t^{-2}\norm[0]{\boldsymbol{q} - \boldsymbol{\Pi_V q}}^2_{f_n, \mathcal{T}^n}
      + \delta_1 \norm[0]{\boldsymbol{\varepsilon^q}_h}^2_{f_n, \mathcal{T}^n}
      \\
      &+ \frac{C}{4\delta_1} \norm[0]{\boldsymbol{\Pi_V q} - \boldsymbol{q}}^2_{f_n, \mathcal{T}^n}
      + \delta_1 \norm[0]{\boldsymbol{\varepsilon^q}_h}^2_{f_n, \mathcal{T}^n}
      + \frac{C}{4\delta_1} \Delta t^{-1}\norm[0]{\boldsymbol{\Pi_V q} - \boldsymbol{q}}^2_{f_n,\mathcal{F}_{\Omega}^n(t_{n+1})}
      +  \delta_1 \norm[0]{\boldsymbol{\varepsilon^q}_h}^2_{f_n, \mathcal{T}^n}
      \\
      &+ \frac{C}{4\delta_1} \Delta t^{-1}\norm[0]{\boldsymbol{\Pi_{V^-} q} - \boldsymbol{q}}^2_{f_n,\mathcal{F}_{\Omega}^n(t_n)}
      +  \delta_1 \norm[0]{\boldsymbol{\varepsilon^q}_h}^2_{f_n, \mathcal{T}^n}
      \\
      &
      + \frac{C}{4\delta_2} \Delta t^{-1} \norm[0]{P^{f_n}_M v - v}^2_{f_n, \partial \mathcal{E}_{\mathcal{S}}^n(t_{n+1})}
      +  \delta_2 \norm[0]{\varepsilon^{\lambda}_h}^2_{f_n, \mathcal{F}_{\mathcal{S}}^n}
      \\
      &+ \frac{C}{4\delta_2} \Delta t^{-1} \norm[0]{P^{f_{n-1}}_{M^-} v - v}^2_{f_n, \partial \mathcal{E}_{\mathcal{S}}^n(t_n)}
      + \delta_2 \norm[0]{\varepsilon^{\lambda}_h}^2_{f_n, \mathcal{F}_{\mathcal{S}}^n},
    \end{split}
  \end{equation}
  where $\delta_1, \delta_2 > 0$ are free to choose
  constants. Collecting terms,
  \begin{equation}
    \label{eq:norms_error4}
    \begin{split}
      &(\alpha - 4\delta_1) \norm[0]{\boldsymbol{\varepsilon^q}_h}_{f_n, \mathcal{T}^n}^2
      + (\alpha - 2\delta_2) \norm[0]{\varepsilon^{\lambda}_h}_{f_n, \mathcal{F}_{\mathcal{S}}^n}^2
      +  e^{-\alpha \Delta t} \norm[0]{\boldsymbol{\varepsilon^{q}}_h}^2_{\mathcal{F}_{\Omega}^n(t_{n+1})}
      +  e^{-\alpha \Delta t} \norm[0]{\varepsilon^{\lambda}_h}^2_{\partial\mathcal{E}_{\mathcal{S}}^n(t_{n+1})}
      \\      
      \leq
      &  \norm[0]{\boldsymbol{\varepsilon^{q-}}_h}^2_{\mathcal{F}_{\Omega}^n(t_n)}
      +  \norm[0]{\varepsilon^{\lambda-}_h}^2_{\partial\mathcal{E}_{\mathcal{S}}^n(t_n)}
      + C \Delta t^{-2}\norm[0]{\boldsymbol{q} - \boldsymbol{\Pi_V q}}^2_{f_n, \mathcal{T}^n}
      \\
      &+ C \norm[0]{\boldsymbol{\Pi_V q} - \boldsymbol{q}}^2_{f_n, \mathcal{T}^n}
      + C \Delta t^{-1}\norm[0]{\boldsymbol{\Pi_V q} - \boldsymbol{q}}^2_{f_n,\mathcal{F}_{\Omega}^n(t_{n+1})}
      \\
      &+ C \Delta t^{-1}\norm[0]{\boldsymbol{\Pi_{V^-} q} - \boldsymbol{q}}^2_{f_n,\mathcal{F}_{\Omega}^n(t_n)}
      + C \Delta t^{-1} \norm[0]{P^{f_n}_M v - v}^2_{f_n, \partial \mathcal{E}_{\mathcal{S}}^n(t_{n+1})}
      \\
      &+ C \Delta t^{-1} \norm[0]{P^{f_{n-1}}_{M^-} v - v}^2_{f_n, \partial \mathcal{E}_{\mathcal{S}}^n(t_n)}.
    \end{split}
  \end{equation}
  The result follows  by choosing $\delta_1 = \alpha / 8$ and $\delta_2 = \alpha/4$.
\end{proof}

\begin{remark}
	If $(\boldsymbol{q}, v) \in \sbr[0]{H^{(s_t,
			s_s)}(\mathcal{E}^n)}^2
	\times H^{(s_t,s_s)}(\mathcal{E}^n)$, with $1 < s_t \leq p+1$ and $1 \leq s_s \leq p+1$, then combining \Cref{lem:bounds_epsilons}, \Cref{cor:weightedL2proj_bounds} and \Cref{thm:Piproj_bounds} gives the following leading order terms
	\begin{equation}
		\begin{split}
			\tfrac{\alpha}{2} \norm[0]{\boldsymbol{\varepsilon^q}_h}^2_{f_n, \mathcal{T}^n}
			&+ \tfrac{\alpha}{2} \norm[0]{\varepsilon^{\lambda}_h}^2_{f_n, \mathcal{F}_{\mathcal{S}}^n}
			+  e^{-\alpha \Delta t}\norm[0]{\boldsymbol{\varepsilon^{q}}_h}^2_{\mathcal{F}_{\Omega}^n(t_{n+1})}
			+  e^{-\alpha \Delta t}\norm[0]{\varepsilon^{\lambda}_h}^2_{\partial\mathcal{E}_{\mathcal{S}}^n(t_{n+1})}
			\\      
			\leq
			&  \norm[0]{\boldsymbol{\varepsilon^{q-}}_h}^2_{\mathcal{F}_{\Omega}^n(t_n)}
			+  \norm[0]{\varepsilon^{\lambda-}_h}^2_{\partial\mathcal{E}_{\mathcal{S}}^n(t_n)}
			\\
			& + C\del[1]{\Delta t^{-2}h^{2s_s} + \Delta t^{2s_t-2}} \norm{\boldsymbol{q}}^2_{H^{(s_t, s_s)}(\mathcal{E}^n)}
			\\
			& + C\del[1]{\Delta t^{-2}h^{2s_s} + \Delta t^{2s_t-2}} \norm{v}^2_{H^{(s_t, s_s)}(\mathcal{E}^n)} 
			\\
			& + C\del[1]{\Delta t^{-2}h^{2s_s} + \Delta t^{2s_t-2}}\norm{v}^2_{H^{(s_t, s_s)}(\partial\mathcal{E}_{\mathcal{S}}^n)},
		\end{split}
	\end{equation}
	where $h = \max_{\mathcal{K}} h_K$.
\end{remark}

To prove the \emph{a priori} error estimates we will use the
following lemma:

\begin{lemma}
	\label{lem:sum_error_bounds}
	Let $A_n, B_n, D_n \geq 0$ for all $n = 0,1, \cdots N-1$, and $\alpha > 0$. Moreover, assume that there exists a constant $C \geq 0$ such that $C\Delta t B_n \leq A_n$ for all $n$. If $B_0 = 0$, $(N-1)\Delta t = T$ and
	\begin{equation}
		\label{eq:eq_An_Bn}
		A_n + e^{-\alpha \Delta t}B_n \leq B_{n-1} + D_n,
	\end{equation}
	then,
	\begin{equation}
		A_n \leq C \sum_{i=1}^{n}D_i,
	\end{equation}
	where $C > 0$ depends on $\alpha$ and $T$.
	
\end{lemma}

\begin{proof}
	Since $C \Delta t B_n \leq A_n$, by \cref{eq:eq_An_Bn}, we have
	\begin{equation}
	\label{eq:eq_An_Bn2}
	\del[1]{C\Delta t + e^{-\alpha \Delta t}}B_n \leq B_{n-1} + D_n.
	\end{equation}
	%
	Let $\gamma = C\Delta t + e^{-\alpha \Delta t}$. By induction, we have the following
	\begin{equation}
		\label{eq:induction_Bn}
		\begin{split}
			B_n & \leq \gamma^{-1} B_{n-1} + \gamma^{-1}D_n
			\\
			& \leq \gamma^{-2} B_{n-2} + \gamma^{-2}D_{n-1} + \gamma^{-1}D_n
			\\
			& \leq \gamma^{-3} B_{n-3} + \gamma^{-3} D_{n-2} + \gamma^{-2}D_{n-1} + \gamma^{-1}D_n
			\\
			& \,\,\,\vdots
			\\
			& \leq \sum_{i = 1}^{n}\gamma^{i-n-1}D_i,
		\end{split}
	\end{equation}
	where we used that $B_0 = 0$. Note that $-n \leq i-n-1 \leq
        -1$, $i \in \cbr{1,\hdots,n}$. This implies that if
        $\gamma \geq 1$ then $\gamma^{i-n-1} \leq \gamma^{-1} \leq 1$
        while if $\gamma < 1$ then $\gamma^{i-n-1} \leq
        \gamma^{-n}$. We next bound
        $\gamma^{-n}$. Note that
	\begin{equation}
          e^{-\alpha \Delta t} \leq C_1 \Delta t + e^{-\alpha \Delta t} = \gamma.
	\end{equation}
	Since $n \leq N-1$, we have that
        $\gamma^{-n} \leq e^{\alpha n \Delta t} \leq e^{\alpha T} \leq
        C_2$. Therefore $\gamma^{i-n-1} \le \max\cbr{1, C_2} = C$.
        
        We may therefore bound $B_n$ in \cref{eq:induction_Bn} as:
	\begin{equation}
          \label{eq:B_n_bound}
          B_n \leq C \sum_{i=1}^{n}D_i.
	\end{equation}
	In order to obtain the final result, recall that, by
        \cref{eq:eq_An_Bn}, we have the following:
	\begin{equation}
          A_n \leq B_{n-1} + D_n.
	\end{equation}
	Using \cref{eq:B_n_bound} we obtain
	\begin{equation}
          A_n \leq C \sum_{i = 1}^{n-1}D_i + D_n
          \leq C \sum_{i=1}^{n}D_i,
	\end{equation}
	which completes the proof.
\end{proof}

The following theorem gives the final error bounds.

\begin{theorem}
  \label{thm:final_error_bounds}
  Let $h = \max_{\mathcal{K} \in \mathcal{T}^n} h_K$. Assume that the
  spatial shape-regularity condition \cref{eq:shape_reg} holds and
  that the triangulation $\mathcal{T}^n $does not have any hanging
  nodes. Suppose that $(\boldsymbol{q}, v)$ solves
  \cref{eq:mixedformlinfreesurface}, with
  $\boldsymbol{q} \in \sbr[0]{H^{(s_t,s_s)}(\mathcal{E}^n)}^2$ and
  $v \in H^{(s_t,s_s)}(\mathcal{E}^n)$, with $1/2 < s_t \leq p+1$ and
  $1 \leq s_s \leq p+1$. Then we have the following estimates:
  \begin{equation}
	  \begin{split}
		  \sum_{n=0}^{N-1} &\norm[1]{\boldsymbol{q} - \boldsymbol{q}_h}_{f_n, \mathcal{T}^n} + \sum_{n=0}^{N-1} \norm[1]{v - \lambda_h}_{f_n, \mathcal{F}_{\mathcal{S}}^n}
		  \\
		  & \le
		  C\del[1]{\Delta t^{-1}h^{s_s} + \Delta t^{s_t-1}} \norm{\boldsymbol{q}}_{H^{(s_t, s_s)}(\mathcal{E}^n)}
		  \\
		  & + C\del[1]{\Delta t^{-1}h^{s_s} + \Delta t^{s_t-1}} \norm{v}_{H^{(s_t, s_s)}(\mathcal{E}^n)} 
		  \\
		  & + C\del[1]{\Delta t^{-1}h^{s_s} + \Delta t^{s_t-1}}\norm{v}_{H^{(s_t, s_s)}(\partial\mathcal{E}_{\mathcal{S}}^n)},
	  \end{split}
  \end{equation}
  where $C > 0$ depends on $\alpha$ and the final time $T$.
\end{theorem}

\begin{proof}
	Let
	\begin{equation}
		\begin{split}
			A_n &= \tfrac{\alpha}{2} 
				\norm[0]{\boldsymbol{\varepsilon^q}_h}^2_{f_{n},
				\mathcal{T}^{n}} + \tfrac{\alpha}{2} \norm[0]{\varepsilon^{\lambda}_h}^2_{f_{n},
				\mathcal{F}_{\mathcal{S}}^{n}},
			\\
			B_{n} &=
			\norm[0]{\boldsymbol{\varepsilon^{q}}_h}^2_{\mathcal{F}_{\Omega}^n(t_{n+1})}
			+
			\norm[0]{\varepsilon^{\lambda}_h}^2_{\partial\mathcal{E}_{\mathcal{S}}^n(t_{n+1})}.
		\end{split}
	\end{equation}
	Note that
	\begin{equation}
		B_{n-1} =
		\norm[0]{\boldsymbol{\varepsilon^{q}}_h}^2_{\mathcal{F}_{\Omega}^{n-1}(t_{n})}
		+
		\norm[0]{\varepsilon^{\lambda}_h}^2_{\partial\mathcal{E}_{\mathcal{S}}^{n-1}(t_{n})} = \norm[0]{\boldsymbol{\varepsilon^{q-}}_h}^2_{\mathcal{F}_{\Omega}^n(t_{n})}
		+
		\norm[0]{\varepsilon^{\lambda-}_h}^2_{\partial\mathcal{E}_{\mathcal{S}}^n(t_{n})}.
	\end{equation}
	By \cref{eq:inv_trace_ineq_K0} and the equivalence of the norms
	$\norm{\cdot}_{\mathcal{K}}$ and $\norm{\cdot}_{f_n,\mathcal{K}}$, we have the following:
	\begin{equation}
		C \Delta t B_n \leq \norm[0]{\boldsymbol{\varepsilon^{q}}_h}^2_{\mathcal{T}^n}
		+
		\norm[0]{\varepsilon^{\lambda}_h}^2_{\mathcal{F}_{\mathcal{S}}^n} \leq C A_n.
	\end{equation}
	Moreover, let $D_n$ be defined as follows:
	\begin{equation}
		\begin{split}
			D_n = & C\del[1]{\Delta t^{-2}h^{2s_s} + \Delta t^{2s_t-2}} \norm{\boldsymbol{q}}^2_{H^{(s_t, s_s)}(\mathcal{E}^n)}
			\\
			& + C\del[1]{\Delta t^{-2}h^{2s_s} + \Delta t^{2s_t-2}} \norm{v}^2_{H^{(s_t, s_s)}(\mathcal{E}^n)} 
			\\
			& + C\del[1]{\Delta t^{-2}h^{2s_s} + \Delta t^{2s_t-2}}\norm{v}^2_{H^{(s_t, s_s)}(\partial\mathcal{E}_{\mathcal{S}}^n)}.
		\end{split}
	\end{equation}
	With these definitions, \Cref{lem:sum_error_bounds} gives the following bound for the projection errors:
	\begin{equation}
		\begin{split}
			\tfrac{\alpha}{2} & \norm[0]{\boldsymbol{\varepsilon^q}_h}^2_{f_n, \mathcal{T}^n}
			+ \tfrac{\alpha}{2} \norm[0]{\varepsilon^{\lambda}_h}^2_{f_n, \mathcal{F}_{\mathcal{S}}^n}
			\\
			& \leq C \del[1]{\Delta t^{-2}h^{2s_s} + \Delta t^{2s_t-2}} \sum_{k=0}^{n}\norm{\boldsymbol{q}}^2_{H^{(s_t, s_s)}(\mathcal{E}^k)}
			\\
			& + C \del[1]{\Delta t^{-2}h^{2s_s} + \Delta t^{2s_t-2}} \sum_{k=0}^{n} \norm{v}^2_{H^{(s_t, s_s)}(\mathcal{E}^k)} 
			\\
			& + C \del[1]{\Delta t^{-2}h^{2s_s} + \Delta t^{2s_t-2}} \sum_{k=0}^{n} \norm{v}^2_{H^{(s_t, s_s)}(\partial\mathcal{E}_{\mathcal{S}}^k)}.
		\end{split}
	\end{equation}
	Summing over all time slabs, we obtain:
	\begin{equation}
		\label{eq:final_bound_proj_errors}
		\begin{split}
			\tfrac{\alpha}{2}& \sum_{n=0}^{N-1}\norm[0]{\boldsymbol{\varepsilon^q}_h}^2_{f_n, \mathcal{T}^n}
			+ \tfrac{\alpha}{2} \sum_{n=0}^{N-1}\norm[0]{\varepsilon^{\lambda}_h}^2_{f_n, \mathcal{F}_{\mathcal{S}}^n}
			\\
			&\leq C(N-1)\del[1]{\Delta t^{-2}h^{2s_s} + \Delta t^{2s_t-2}} \norm{\boldsymbol{q}}^2_{H^{(s_t, s_s)}(\mathcal{E})}
			\\
			& + C (N-1) \del[1]{\Delta t^{-2}h^{2s_s} + \Delta t^{2s_t-2}}  \norm{v}^2_{H^{(s_t, s_s)}(\mathcal{E})} 
			\\
			& + C (N-1) \del[1]{\Delta t^{-2}h^{2s_s} + \Delta t^{2s_t-2}} \norm{v}^2_{H^{(s_t, s_s)}(\partial\mathcal{E}_{\mathcal{S}})}.
		\end{split}
	\end{equation}
  By the triangle inequality note that
  \begin{equation}
    \begin{split}
      \sum_{n=0}^{N-1} \norm{\boldsymbol{q} - \boldsymbol{q}_h}_{f_n, \mathcal{T}^n}
      &\le \sum_{n=0}^{N-1} \norm{\boldsymbol{q} -\boldsymbol{\Pi _V q}}_{f_n, \mathcal{T}^n} + \sum_{n=0}^{N-1} \norm[0]{\boldsymbol{\varepsilon^q}_h}_{f_n, \mathcal{T}^n},
      \\
      \sum_{n=0}^{N-1} \norm{v - \lambda_h}_{f_n, \mathcal{F}_{\mathcal{S}}^n}
      &\le \sum_{n=0}^{N-1} \norm[0]{v -P_M^{f_n} v}_{f_n, \mathcal{F}_{\mathcal{S}}^n} + \sum_{n=0}^{N-1} \norm[0]{\varepsilon_h^{\lambda}}_{f_n, \mathcal{F}_{\mathcal{S}}^n}.      
    \end{split}
  \end{equation}
  The result follows by \Cref{thm:Piproj_bounds} and
  \cref{eq:final_bound_proj_errors}.
\end{proof}


\begin{remark}
  \label{rem:leadingorderterms}
  Assuming
  $\boldsymbol{q} \in \sbr[0]{H^{(p+1, p+1)}(\mathcal{E})}^2$ and
  $v \in H^{(p+1, p+1)}(\mathcal{E})$, the error estimates in
  \Cref{thm:final_error_bounds} give the following leading order
  terms:
  \begin{equation}
    \sum_{n=0}^{N-1} \norm[0]{\boldsymbol{q} - \boldsymbol{q}_h}_{f_n, \mathcal{T}^n}
    + \sum_{n=0}^{N-1} \norm[0]{v - \lambda_h}_{f_n, \mathcal{F}_{\mathcal{S}}^n}
    \leq
    C \del[1]{\Delta t^{-1}h^{p+1} + \Delta t^{p}}.
  \end{equation}
  %
\end{remark}

\section{Numerical results}
\label{sec:numerical_results}

In this section we verify the theoretical results of the previous
sections. The space-time HDG method for the linear free-surface
problem \cref{eq:mixedformlinfreesurface} is implemented using the
modular finite-element method (MFEM) library \cite{mfem-library}. The
linear systems of algebraic equations are solved by the direct solver
MUMPS \cite{ MUMPS:2, MUMPS:1} through PETSc \cite{petsc-user-ref,
  petsc-efficient}.

To obtain the space-time mesh, we first triangulate the spatial domain
$\Omega$. We then extrude the spatial triangles in the time direction
to obtain space-time prisms.

In all our simulations we take $\tau=5$ and $\alpha=0.1$.

\subsection{Linear waves in an unbounded domain}
\label{ss:linearwavesunboundedD}

We consider the time harmonic linear free-surface waves example in an
unbounded domain from \cite[Section 8]{Vegt:2005}. We consider the
domain $\Omega = [-1,\,1] \times [-1,\,0]$ and apply periodic boundary
conditions at $x_1 = -1$ and $x_1 = 1$. The analytical solution to
this problem is given by
\begin{subequations}
  \begin{align}
    \phi(\boldsymbol{x}, t) &= \phi_0\cosh(k(x_2 + 1))\cos(\omega t - k x_1),
    \\
    \zeta(x_1, t) &= -\partial_t\phi(x_1, 0, t) = \phi_0 \omega \cosh(k) 
                    \sin(\omega t - k x_1),
  \end{align}
\end{subequations}
where $\phi_0$ denotes the amplitude of the velocity potential, $k$ is
the wave number which is related to the wavelength $\lambda_w$ by
$k=2\pi/\lambda_w$, and $\omega$ is the frequency of the oscillations
which satisfies the dispersion relation $\omega^2 = k\tanh(k)$.

We take $\lambda_w = 1$ and $\phi_0$ such that the maximum amplitude
of the wave height is 0.05. In \cref{tab:ConvRates_space_prisms},
\cref{tab:ConvRates_time_prisms}, and \cref{tab:ConvRates_spacetime_prisms} we show the
approximation errors and convergence rates for the velocity
$\boldsymbol{q}_h$ on the entire space-time domain $\mathcal{E}$ and
for the free-surface height $\lambda_h$ on the entire free-surface
boundary $\partial \mathcal{E}_{\mathcal{S}}$. We test convergence in
space, in time, and in space-time separately.

We first test convergence in space. To ensure the spatial error
dominates over the temporal error we take a small time step
$\Delta t = 10^{-5}$ when $p=1$ and $\Delta t = 10^{-4}$ when
$p=2$. We compute the error after 200 (when $p=1$) or 20 (when $p=2$)
time steps. As observed in \cref{tab:ConvRates_space_prisms}, the
error is of order $\mathcal{O}(h^{p+1})$.

We next consider convergence in time. For this we compute up to a
final time $T=1$. To ensure that the temporal error dominates over the
spatial error we use a mesh consisting of 73728 elements when $p=1$
and 36864 elements when $p=2$. We observe in
\cref{tab:ConvRates_time_prisms} that the error is of order
$\mathcal{O}(\Delta t^{p+1})$.

We note that the rates of convergence in space and time separately are
better than predicted from \cref{rem:leadingorderterms}. We now
consider convergence in space-time in which we refine the spatial mesh
and time step simultaneously. We compute the solution up to a final
time of $T = 1$. The initial time step is $\Delta t = 0.25$ and the
initial mesh has 18 elements. We observe in
\cref{tab:ConvRates_spacetime_prisms} that the error is of order
$\mathcal{O}(\Delta t^p + h^p)$, as expected from our analysis, see
\cref{rem:leadingorderterms}.

Finally, we consider a case where $h$ is fixed so that the spatial
mesh consists of 1152 triangles and the number of global
degrees-of-freedom is 13920. We take $p = 1$ and solve the problem up
to a final time $T = 1$. From \cref{rem:leadingorderterms}, we see
that if $h$ is fixed, eventually the dominant term in the error will
be of order $\mathcal{O}(\Delta t^{-1}h^{p+1})$, which results in
divergence of the solution. This effect can be observed in
\cref{tab:ConvRates_divergence} where the errors start to increase
after three levels of refinement in time. Unlike standard time
stepping methods, for space-time methods $\Delta t$ has to be chosen
carefully depending on its relation to the spatial mesh size $h$.

\begin{center}
  \begin{table}[tbp]
    \centering
    \small
    \caption{Spatial rates of convergence for linear waves in an
      unbounded domain, see \cref{ss:linearwavesunboundedD}.}
    \begin{tabular}{lcccccccc}
      \toprule
      \multirow{2}{*}{}
      & & \multicolumn{2}{c}{$\boldsymbol{q}_h$}  & \multicolumn{2}{c}{$\lambda_h$}\\
       & {DOFs}
       & {$L^2(\mathcal{E})$-error} & {Order}
       & {$L^2(\partial \mathcal{E}_{\mathcal{S}})$-error} & {Order}\\
      \midrule
      & {228}   & {1.1e-3} & {-} & {2.5e-2} & {-}\\
      $p = 1$ & {888}  & {3.2e-4} & {1.7} & {1.4e-2} & {0.9}\\
      & {3504}  & {8.5e-5} & {1.9}& {3.4e-3} & {2.0}\\
      & {13920} & {2.2e-5} & {2.0} & {8.2e-4} & {2.1}\\
      & {55488} & {5.4e-6} & {2.0} & {1.9e-4} & {2.1}\\
      \\
      & {486}    & {4.0e-4} & {-} & {1.5e-3} & {-}\\
      $p = 2$ & {1890}   & {6.0e-5} & {2.7} & {2.3e-4} & {2.8}\\
      & {7452}   & {7.9e-6} & {2.9} & {3.5e-5} & {2.7}\\
      & {29592}  & {1.0e-6} & {3.0} & {4.8e-6} & {2.9}\\
      \bottomrule        
    \end{tabular}
    \label{tab:ConvRates_space_prisms}
  \end{table}
\end{center}

\begin{center}
  \begin{table}[tbp]
    \centering
    \small
    \caption{Time rates of convergence for linear waves in an
      unbounded domain, see \cref{ss:linearwavesunboundedD}.}
    \begin{tabular}{lccccc}
      \toprule
      \multirow{2}{*}{}
      & & \multicolumn{2}{c}{$\boldsymbol{q}_h$}  & \multicolumn{2}{c}{$\lambda_h$}\\
      & {$\Delta t$} & {$L^2(\mathcal{E})$-error} & {Order} & {$L^2(\partial \mathcal{E}_{\mathcal{S}})$-error} & {Order}\\
      
      \midrule
      & 1 & {1.7e-2} & {-} & {1.7e-2} & {-}\\
      $p = 1$ & 1/2 & {5.1e-3} & {1.8} & {5.1e-3} & {1.8}\\
      & 1/4 & {1.2e-3} & {2.1}& {1.2e-3} & {2.1}\\
      & 1/8 & {3.0e-4} & {2.0} & {3.0e-4} & {2.0}\\
      & 1/16 & {8.2e-5} & {1.9} & {7.9e-5} & {1.9}\\
      \\
      & 1 & {3.8e-3} & {-} & {3.8e-3} & {-}\\
      $p = 2$ & 1/2 & {4.8e-4} & {3.0} & {4.8e-4} & {3.0}\\
      & 1/4 &  {5.9e-5} & {3.0}& {5.9e-5} & {3.0}\\
      & 1/8 & {7.5e-6} & {3.0} & {7.5e-6} & {3.0}\\
      & 1/16 & {1.6e-6} & {2.3} & {1.3e-6} & {2.5}\\
      \bottomrule        
    \end{tabular}
    \label{tab:ConvRates_time_prisms}
  \end{table}
\end{center}

\begin{center}
  \begin{table}[tbp]
    \centering
    \small
    \caption{Space-time rates of convergence for linear waves in an
      unbounded domain, see \cref{ss:linearwavesunboundedD}.}
    \begin{tabular}{lcccccccc}
      \toprule
      \multirow{2}{*}{}
      & & \multicolumn{2}{c}{$\boldsymbol{q}_h$}  & \multicolumn{2}{c}{$\lambda_h$}\\
      & {DOFs}
        & {$L^2(\mathcal{E})$-error} & {Order}
      & {$L^2(\partial \mathcal{E}_{\mathcal{S}})$-error} & {Order}\\
      \midrule
      & {228}   & {3.5e-2} & {-} & {3.4e-2} & {-}\\
      $p = 1$ & {888}  & {1.7e-2} & {1.1} & {1.5e-2} & {1.2}\\
      & {3504}  & {7.2e-3} & {1.2}& {5.9e-3} & {1.3}\\
      & {13920} & {3.2e-3} & {1.2} & {2.7e-3} & {1.2}\\
      & {55488} & {1.5e-3} & {1.1} & {1.3e-3} & {1.1}\\
      \\
      & {486}    & {1.6e-2} & {-} & {1.3e-2} & {-}\\
      $p = 2$ & {1890}   & {3.4e-3} & {2.2} & {2.3e-3} & {2.5}\\
      & {7452}   & {6.7e-4} & {2.4} & {4.4e-4} & {2.4}\\
      & {29592}  & {1.4e-4} & {2.3} & {9.8e-5} & {2.2}\\
      & {117936} & {3.1e-5} & {2.2} & {2.4e-5} & {2.1}\\
      \bottomrule        
    \end{tabular}
    \label{tab:ConvRates_spacetime_prisms}
  \end{table}
\end{center}

\begin{center}
	\begin{table}[tbp]
		\centering
		\small
		\caption{Time rates of convergence for a coarse mesh for linear waves in an
			unbounded domain, see \cref{ss:linearwavesunboundedD}.}
		\begin{tabular}{ccccc}
			\toprule
			\multirow{2}{*}{}
			& \multicolumn{2}{c}{$\boldsymbol{q}_h$}  & \multicolumn{2}{c}{$\lambda_h$}\\
			{$\Delta t$} & {$L^2(\mathcal{E})$-error} & {Order} & {$L^2(\partial \mathcal{E}_{\mathcal{S}})$-error} & {Order}\\
			
			\midrule
			1 & {1.8e-2} & {-} & {1.8e-2} & {-}\\
			1/2 & {5.3e-3} & {1.7} & {5.2e-3} & {1.8}\\
			1/4 & {1.8e-3} & {1.6}& {1.5e-3} & {1.8}\\
			1/8 & {1.4e-3} & {0.3} & {9.9e-4} & {0.6}\\
			1/16 & {2.0e-3} & {-0.5} & {1.5e-3} & {-0.6}\\
			1/32 & {3.2e-3} & {-0.7} & {2.6e-3} & {-0.8}\\
			1/64 & {5.6e-3} & {-0.8} & {5.0e-3} & {-0.9}\\
			1/128 & {1.0e-2} & {-0.9} & {9.5e-3} & {-0.9}\\
			1/256 & {1.8e-2} & {-0.8} & {1.8e-2} & {-0.9}\\
			\bottomrule        
		\end{tabular}
		\label{tab:ConvRates_divergence}
	\end{table}
\end{center}

\subsection{Simulation of water waves in a water tank}
\label{ss:waterwavestank}

In this example we consider waves generated by a piston-type wave
maker. This test case is proposed in \cite{Westhuis:2001}. In this
case we consider the spatial domain
$\Omega = \sbr{0, 10} \times \sbr{-1, 0}$. We apply homogeneous
Neumann boundary conditions on $x_1 = 10$ and $x_2 = -1$. The wave
maker is located on the left side of the domain, i.e., at $x_1 = 0$,
where the boundary condition is given by
\begin{equation}
  \boldsymbol{q}\cdot \boldsymbol{n} = T(t),
\end{equation}
where $T(t) = ia\exp(-a f t)$ with $a=0.05$ the amplitude of the wave
and $f=1.8138$ the frequency of the wave. Only the real part of $T(t)$
produces physical solutions. We compute the solution for
$t \in [0, 53.4]$ and take $\Delta t = 0.2$. The mesh consists of 512
prismatic elements which are constructed by extruding spatial
triangles in the time direction. \Cref{fig:wavemaker} shows the
free-surface elevation or wave height at different time levels for
polynomial orders $p=1$, $p=2$ and $p=3$. At time $t = 4$, the first
wave leaves the wave maker which is located at $x_1 = 0$. At
$t = 25.8$ the waves reach the right wall of the water tank. At time
$t = 53.4$, waves have hit the right wall and have started traveling
in the opposite direction. We furthermore note that the discretization
is less diffusive as the polynomial degree increases.
 
\begin{figure}
  \centering
  \subfloat[Free-surface elevation at $t = 4$.]
  {
    \includegraphics[width=0.46\linewidth]{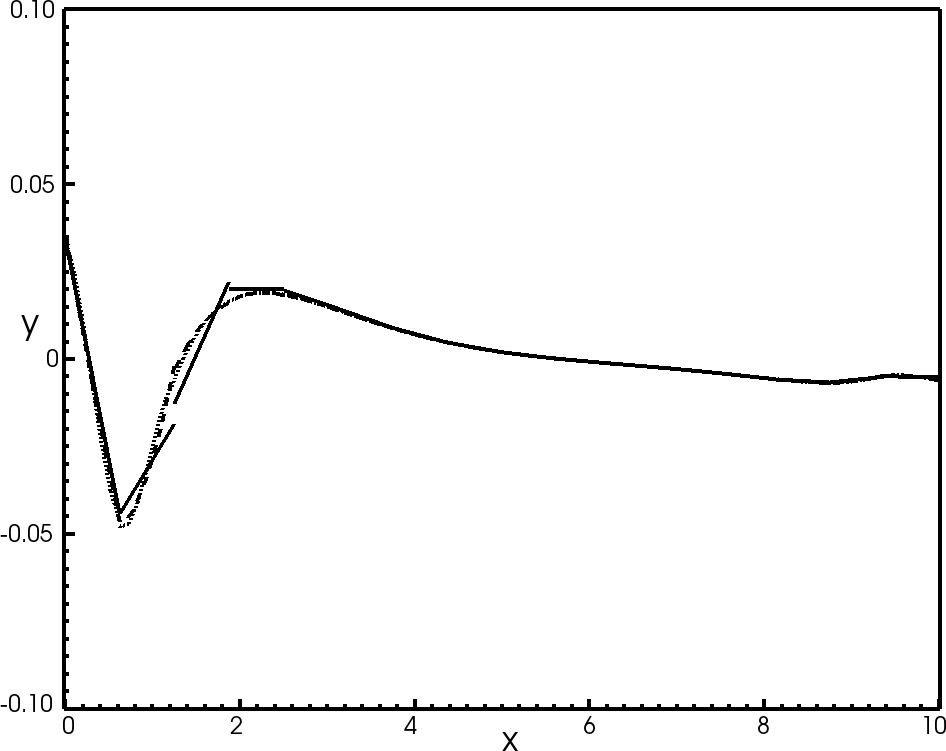}
  }
  \quad
  \subfloat[Free-surface elevation at $t = 25.8$.]
  {
    \includegraphics[width=0.46\linewidth]{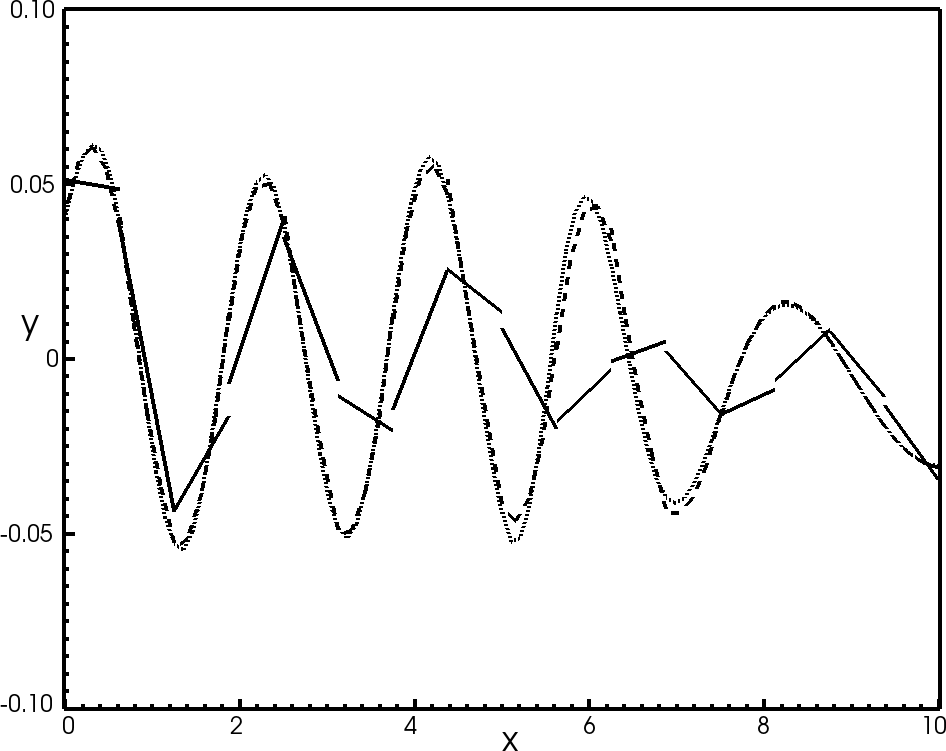}
  }
  \\
  \subfloat[Free-surface elevation at $t = 53.4$.]
  {
    \includegraphics[width=0.46\linewidth]{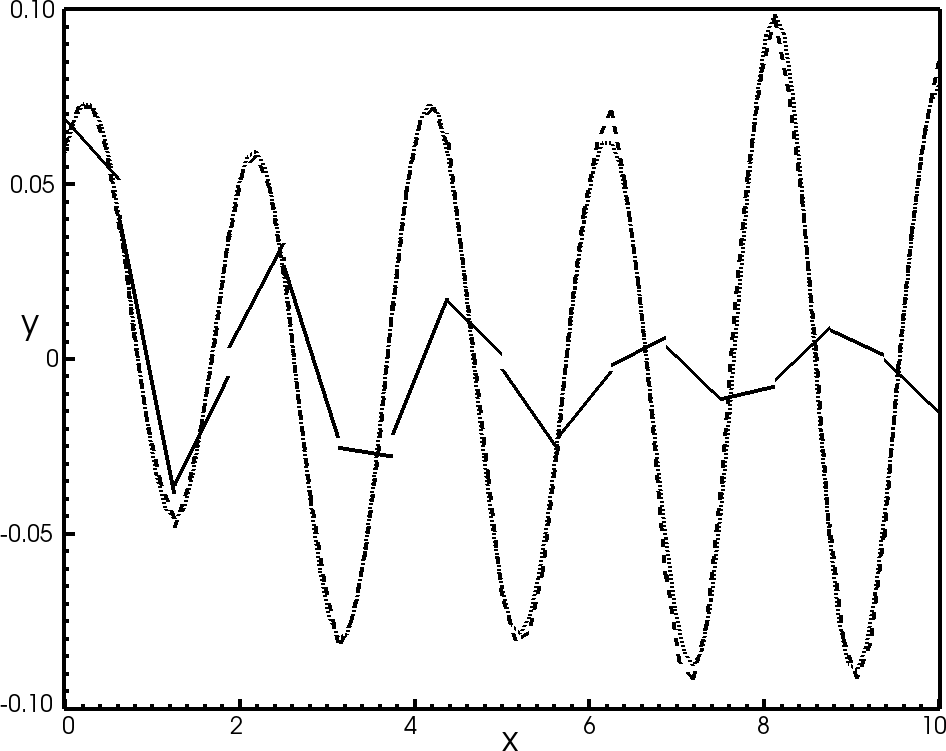}
  }
  \caption{Simulation of water waves in a water tank, see
    \cref{ss:waterwavestank}. The free-surface elevation at different
    time levels for polynomial degree $p=1$ (solid line), $p=2$
    (dashed line), and $p=3$ (dotted line).}
  \label{fig:wavemaker}
\end{figure}

\section{Conclusions}
\label{sec:conclusions}

In this paper we presented a space-time hybridizible discontinuous
Galerkin method for the mixed form of the linear free-surface problem
on prismatic space-time meshes. The use of weighted inner products
allows us to show well-posedness of the discrete problem. An \emph{a
  priori} error analysis was performed by using a projection operator
tailored to the discretization. Additionally, this analysis explicitly
specifies the dependency on the spatial mesh size and time
step. Numerical tests verify our theoretical results.

\bibliographystyle{amsplain}
\bibliography{references}
\end{document}